\documentclass[onefignum,onetabnum]{siamart190516}

\usepackage{lipsum}
\usepackage{graphicx}
\usepackage{epstopdf}
\usepackage{algorithmic}
\usepackage{mathrsfs}
\usepackage{epsfig, graphicx}
\usepackage{latexsym,amsfonts,amsbsy,amssymb}
\newtheorem{conj}{Conjecture}[section]

\newtheorem{problem}{Problem}
\graphicspath{{include-figure-AA/}}
\usepackage{tablefootnote}
\crefformat{footnote}{#2\footnotemark[#1]#3}
\ifpdf
  \DeclareGraphicsExtensions{.eps,.pdf,.png,.jpg}
\else
  \DeclareGraphicsExtensions{.eps}
\fi

\newcommand{\ds}{\displaystyle}
\newsiamremark{remark}{Remark}
\newsiamremark{hypothesis}{Hypothesis}
\crefname{hypothesis}{Hypothesis}{Hypotheses}
\newsiamthm{claim}{Claim}

\headers{Fixed-Point Analysis of Anderson Acceleration}{Hans De Sterck and Yunhui He}

\title{Linear Asymptotic Convergence of Anderson Acceleration: Fixed-Point Analysis\thanks{Submitted to the editors DATE.
\funding{This work was funded by XXXX}}}

\author{Hans De Sterck\thanks{Department of Applied Mathematics
University of Waterloo, 200 University Ave W, Waterloo, ON N2L 3G1, Canada
  (\email{hdesterck@uwaterloo.ca}, \email{yunhui.he@uwaterloo.ca}).}
  \and Yunhui He\footnotemark[2]}

\usepackage{amsopn}

\usepackage{xspace,color}

\begin{document}

\maketitle

\begin{abstract}
We study the asymptotic convergence of AA($m$), i.e., Anderson acceleration with window size $m$ for accelerating fixed-point methods $x_{k+1}=q(x_{k})$, $x_k \in \mathbb{R}^n$. Convergence acceleration by AA($m$) has been widely observed but is not well understood.
We consider the case where the fixed-point iteration function $q(x)$ is differentiable and the convergence of the fixed-point method itself is root-linear. We identify numerically several conspicuous properties of AA($m$) convergence: First, AA($m$) sequences $\{x_k\}$ converge root-linearly but the root-linear convergence factor depends strongly on the initial condition.
Second, the AA($m$) acceleration coefficients $\boldsymbol{\beta}^{(k)}$ do not converge but oscillate as $\{x_k\}$ converges to $x^*$. To shed light on these observations, we write the AA($m$) iteration as an augmented fixed-point iteration $\boldsymbol{z}_{k+1} =\Psi(\boldsymbol{z}_k)$, $\boldsymbol{z}_k \in \mathbb{R}^{n(m+1)}$ and analyze the continuity and differentiability properties of $\Psi(\boldsymbol{z})$
and $\boldsymbol{\beta}(\boldsymbol{z})$.
We find that the vector of acceleration coefficients $\boldsymbol{\beta}(\boldsymbol{z})$ is not continuous at the fixed point $\boldsymbol{z}^*$. However, we show that, despite the discontinuity of $\boldsymbol{\beta}(\boldsymbol{z})$, the iteration function $\Psi(\boldsymbol{z})$ is Lipschitz continuous and directionally differentiable at $\boldsymbol{z}^*$ for AA(1), and we generalize this to AA($m$) with $m>1$ for most cases. Furthermore, we find that $\Psi(\boldsymbol{z})$ is not differentiable at $\boldsymbol{z}^*$. We then discuss how these theoretical findings relate to the observed convergence behaviour of AA($m$). The discontinuity of $\boldsymbol{\beta}(\boldsymbol{z})$ at $\boldsymbol{z}^*$ allows $\boldsymbol{\beta}^{(k)}$ to oscillate as $\{x_k\}$ converges to $x^*$, and the non-differentiability of $\Psi(\boldsymbol{z})$ allows AA($m$) sequences to converge with root-linear convergence factors that strongly depend on the initial condition.
Additional numerical results illustrate our findings for several linear and nonlinear fixed-point iterations $x_{k+1}=q(x_{k})$ and for various values of the window size $m$.
\end{abstract}

\begin{keywords}
Anderson acceleration,  fixed-point method, root-linear convergence, asymptotic convergence factor
\end{keywords}

\begin{AMS}
65B05, 
65F10, 
65H10, 
65K10 
\end{AMS}

\section{Introduction}
This paper concerns convergence acceleration methods for fixed-point (FP) iterations of the type
\begin{equation}\label{eq:fixed-point}
  x_{k+1}=q(x_{k}), \quad x_k\in\mathbb{R}^n,  \quad k=0,1,2,\ldots, \tag{FP}
\end{equation}
that seek to approximate a fixed point $x^*=q(x^*)$.
Specifically, we consider the following nonlinear acceleration iteration with window size $m$:
\begin{equation}\label{eq:AA-iteration}
  x_{k+1}= q(x_k) + \sum_{i=1}^{\min(k,m)}\beta_{i}^{(k)}(q(x_k)-q(x_{k-i})) \qquad k=0,1,2,\ldots,
\end{equation}
where the coefficients $\beta_{i}^{(k)}$ are determined by solving a small optimization
problem in every step $k$ that minimizes a linearized residual in the new iterate $x_{k+1}$. Method
\cref{eq:AA-iteration} is known as \emph{Anderson acceleration (AA)} \cite{anderson1965iterative}.
More precisely, defining the residuals $r(x)$ of the fixed-point iteration by
\begin{equation}\label{eq:resid}
	r(x)=x-q(x),
\end{equation}
AA($m$), with window size $m$, solves in every iteration the optimization problem
\begin{equation}\label{eq:Andersonbetas}
	\min_{\{\beta_i^{(k)} \}} \bigg\| r(x_k) + \sum_{i=1}^{\min(k,m)} \beta_i^{(k)} ( r(x_k) - r(x_{k-i})) \bigg\|,
\end{equation}
with up to $m$ variables, and optimization problem \cref{eq:Andersonbetas} is normally posed in the $2$-norm.
Our discussion will focus on the case where iteration \cref{eq:fixed-point} is, by itself, a convergent iteration, but this is, in fact, not necessary for Anderson iteration \cref{eq:AA-iteration} to converge or be effective.

Assume that $k>m$. Define $r_k =x_k-q(x_k)$ and
\begin{equation}\label{eq:beta-vector-form}
  \boldsymbol{\beta}^{(k)} =\begin{bmatrix} \beta_1^{(k)} \\ \vdots \\ \beta_m^{(k)} \end{bmatrix},\quad
  R_k = \begin{bmatrix} r_k-r_{k-1} &  r_k-r_{k-2} & \cdots & r_k-r_{k-m} \end{bmatrix}.
\end{equation}
Then, using the 2-norm in \cref{eq:Andersonbetas}, the solution of the least-squares problem is given  by
\begin{equation}\label{eq:AAm-beta-form}
 \boldsymbol{\beta}^{(k)}  = -(R_k^TR_k)^{-1} R_k^Tr_k,
\end{equation}
if $R_{k}^TR_k$ is invertible. Or more generally, we can write
\begin{equation}\label{eq:AAm-beta-form-pseudo}
 \boldsymbol{\beta}^{(k)}  = -R_k^{\dag}r_k,
\end{equation}
where $R_k^{\dag}$ is the pseudo-inverse of $R_k$, and we note that
\begin{equation}\label{eq:two-ways-pseudo-inverse}
 R_k^{\dagger}= \big(R_k^TR_k\big)^{\dagger} R_k^T.
\end{equation}
This covers the case where $R_{k}^TR_k$ is not invertible, by taking $\boldsymbol{\beta}^{(k)}$ as
the minimum-norm solution of the least-squares problem in this case.

The specific case of AA($m$) with $m=1$ in \cref{eq:AA-iteration} reads
\begin{equation}
\label{eq:anderson-1-step}
  x_{k+1} = (1+\beta_k) q(x_k) -\beta_k q(x_{k-1}),
\end{equation}
where we have defined
$$\beta_k=\beta_1^{(k)}.$$
When $m=1$ and $r_k \ne r_{k-1}$,
\begin{equation}\label{eq:AA-1-step-beta}
  \beta_k = \ds \frac{-r_k^T(r_k-r_{k-1})}{\|r_k-r_{k-1}\|^2}.
\end{equation}
When $r_k = r_{k-1}$, we can, according to \cref{eq:AAm-beta-form-pseudo}, take $\beta_k=0$.
Let $\boldsymbol{z}_k=\begin{bmatrix} x_k \\ x_{k-1}\end{bmatrix}$. Then we can in turn write AA(1) as
a fixed-point iteration,
\begin{equation}\label{eq:AA-fixed-point}
  \boldsymbol{z}_{k+1} =\Psi(\boldsymbol{z}_k), \tag{AA}
\end{equation}
with
\begin{equation}\label{eq:AA(1)-system}
 \Psi(\boldsymbol{z}_k)
  = \begin{bmatrix}
   q(x_k)+\beta(\boldsymbol{z}_k)\big( q(x_k)-q(x_{k-1})\big)\\
   x_k
  \end{bmatrix},
\end{equation}
where $\beta(\boldsymbol{z}_k) =\beta_k$. As we explain in some more detail below, AA($m$) for $m>1$ can also be written in the form of fixed-point iteration \cref{eq:AA-fixed-point} using a similar lifting approach, with $\boldsymbol{z}_k \in \mathbb{R}^{n(m+1)}$. In what follows, vectors such as $\boldsymbol{z}_k$ that live in the augmented space $\mathbb{R}^{n(m+1)}$ will be indicated by bold font.

In this paper, we are interested in how the asymptotic convergence speed of iteration \cref{eq:fixed-point} relates to the asymptotic convergence speed of the accelerated iteration \cref{eq:AA-fixed-point}. Specifically, we consider the case where $q(x)$ is differentiable at $x^*$, such that the asymptotic convergence of \cref{eq:fixed-point} is linear. We seek to investigate the improvement in asymptotic convergence speed resulting from the acceleration of \cref{eq:fixed-point} by  \cref{eq:AA-fixed-point}.

For reasons that will become clear below, the relevant notion of convergence is root-linear (or \emph{r-linear}) convergence \cite{ortega2000iterative}:
\begin{definition}[r-linear convergence of a sequence]
Let $\{x_k\}$  be any sequence that converges to $x^*$. Define
\begin{equation*}
  \rho_{\{x_k\}} = \limsup\limits_{k\rightarrow \infty}\|x^*-x_k\|^{\frac{1}{k}}.
\end{equation*}
We say $\{x_k\}$ converges r-linearly with r-linear convergence factor $\rho_{\{x_k\}}$ if $\rho_{\{x_k\}}\in(0,1)$, and $r$-superlinearly if $\rho_{\{x_k\}}=0$. The ``r-'' prefix stands for ``root''.
\end{definition}

Since the r-linear convergence factor of an iteration sequence resulting from $x_{k+1}=q(x_k)$, for a given iteration function $q(x)$, may depend on the initial guess $x_0$ and on the specific fixed point $x^*$ the iteration converges to in the case multiple fixed points exist, we need to consider the worst-case r-linear convergence factor for convergence of method $x_{k+1}=q(x_k)$ to a specific fixed point $x^*$:
\begin{definition}[r-linear convergence of a fixed-point iteration]\label{def:rho-method}
Consider fixed-point iteration $x_{k+1}=q(x_k)$. We define the set of iteration sequences that converge to a given fixed point $x^*$ as
\begin{equation*}
C(q, x^*)=\Big\{  \{x_k\}_{k=0}^{\infty} | \quad x_{k+1} = q(x_k) \textrm{ for } k=0,1,\ldots, \textrm{ and } \lim_{k \rightarrow \infty}x_k=x^*\Big\},
\end{equation*}
and the worst-case r-linear convergence factor over $C(q, x^*)$ is defined as
\begin{equation}\label{eq:r-factor-defi}
  \rho_{q,x^*} = \sup \Big\{ \rho_{\{x_k\}} |\quad  \{x_k\}\in C(q, x^*) \Big\}.
\end{equation}
We say that the FP method converges r-linearly to $x^*$ with r-linear convergence factor $\rho_{q,x^*}$ if $\rho_{q,x^*} \in(0,1)$.
\end{definition}

The following classical theorem (see, e.g., \cite{ortega2000iterative}) shows that, if the iteration function $q(x)$ in
\cref{eq:fixed-point} is differentiable at $x^*$, the worst-case r-linear convergence factor, $\rho_{q,x^*}$, is determined by the spectral radius of the Jacobian $q'(x)$ evaluated at $x^*$:
\begin{theorem}\label{thm:Ostrowski-Theorem}[Ostrowski Theorem]
Suppose that $q: D\subset \mathbb{R}^n \rightarrow \mathbb{R}^n$ has a
fixed point $x^*$ that is an interior point of $D$, and is differentiable at $x^*$. If the spectral radius
of $q'(x^*)$ satisfies $0<\rho(q'(x^*)) < 1$, then the FP method converges r-linearly with $\rho_{q,x^*} = \rho(q'(x^*))$.
\end{theorem}

It is useful to consider the special case where the iteration functions $q(x)$ in \cref{eq:fixed-point} is affine, i.e.,
$q(x)=M \, x+b$ and
\begin{equation}\label{eq:fixed-point-linear}
  x_{k+1}=M \, x_{k} + b, \quad  k=0,1,2,\ldots,
\end{equation}
where $x_k, b\in\mathbb{R}^n$ and $M\in\mathbb{R}^{n \times n}$.
Since the error propagation equation for iteration \cref{eq:fixed-point-linear} is the linear iteration
\begin{equation}\label{eq:fixed-point-linear-error}
  e_{k+1}=M \, e_{k} \quad  k=0,1,2,\ldots,
\end{equation}
where the error of iterate $x_k$ is defined by $e_k=x^*-x_k$,
we call iteration \cref{eq:fixed-point} linear when $q(x)$ is affine, and nonlinear otherwise.
It is well-known that, in the linear case, AA($m$) with infinite window size is essentially equivalent
to the GMRES iterative method applied to $(I-M) \, x=b$ \cite{walker2011anderson}.

Given an iteration \cref{eq:fixed-point} and a set of initial conditions $x_0$ that converge to a fixed point $x^*$, we can define the set of all sequences in $C(q, x^*)$ that converge with a smaller r-linear convergence factor than $\rho_{q,x^*}$ as
\begin{equation}\label{eq:S-definition}
  S = \Big\{ \{x_k\}\in C(q, x^*) | \quad  \rho_{\{x_k\}} <\rho_{q,x^*}  \Big\}.
\end{equation}
In the linear case it is easy to see that, when $M$ is diagonalizable and has eigenvalues that are not all of equal magnitude, $\rho_{\{x_k\}}= \rho_{q,x^*}$, except for initial conditions $x_0$ that lie in a set of measure zero in $\mathbb{R}^n$. That is, in this case the set $S$ has measure zero ($|S|=0$) in $C(q, x^*)$.

At this point it is useful to consider a simple example to motivate the questions we address in this paper.
The simple linear example we consider is\\

\begin{problem}\label{prob:linear2x2}
\begin{equation}\label{eq:q-linear-2x2-simple}
  x_{k+1} = M x_k, \qquad
   M=\begin{bmatrix}
  2/3 & 1/4\\
  0  & 1/3
  \end{bmatrix}.
\end{equation}
Clearly, the eigenvalues of $M$ are $\lambda_1=2/3$ and $\lambda_2=1/3$.
\end{problem}

\begin{figure}[h]
\centering
\includegraphics[width=.49\textwidth]{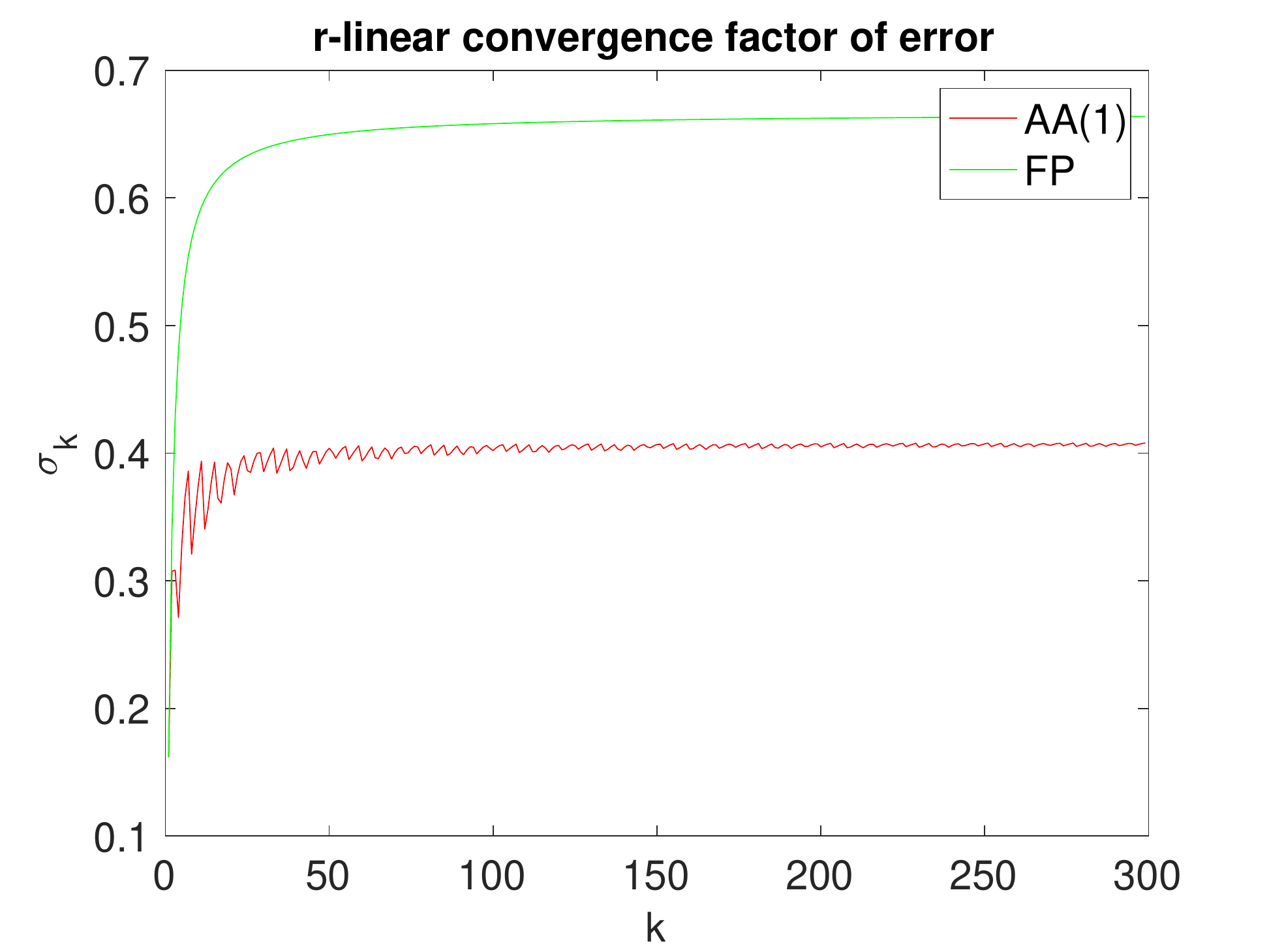}
\includegraphics[width=.49\textwidth]{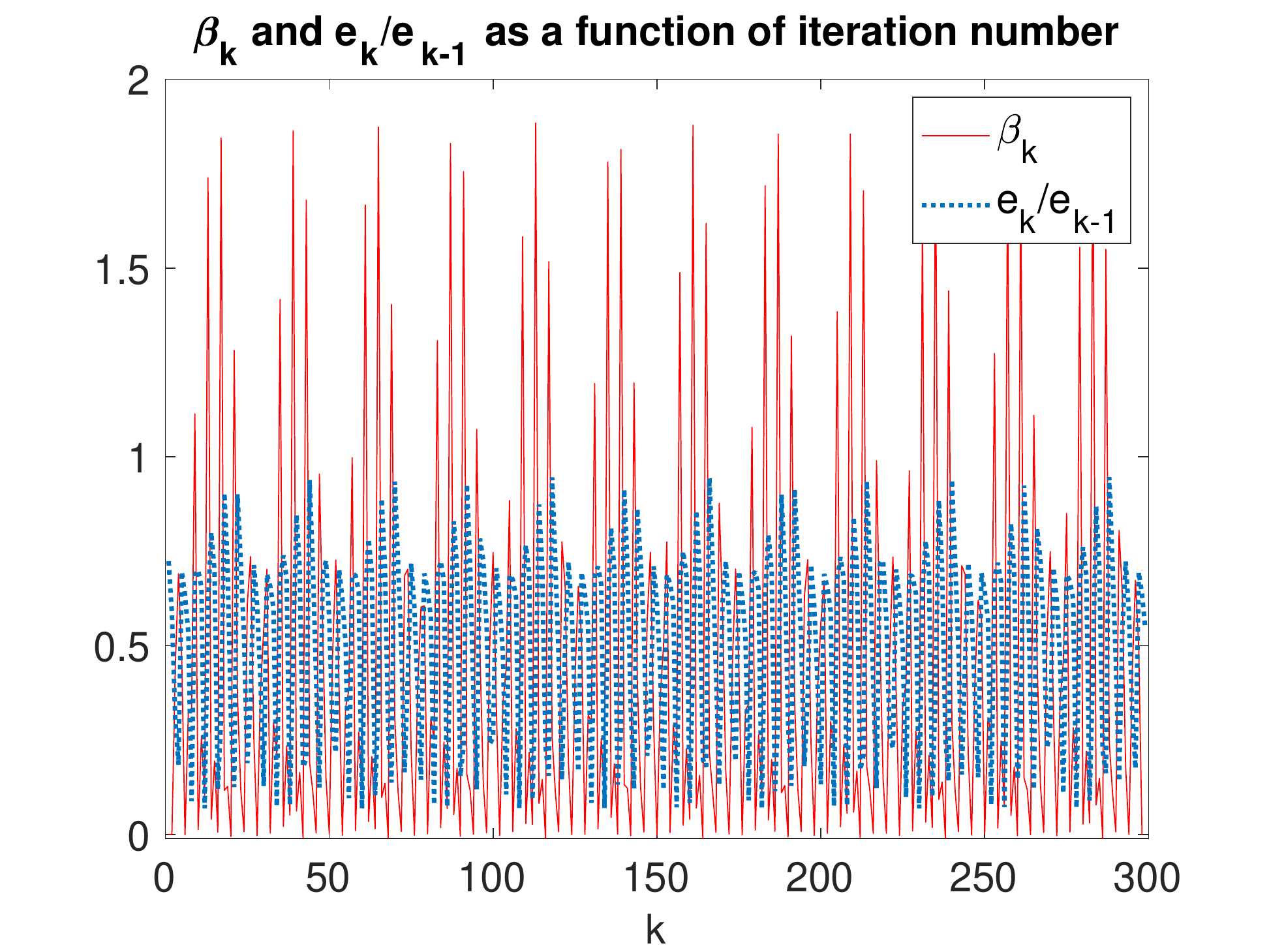}
\caption{\cref{prob:linear2x2} with initial guess $x_0=[0.2,0.1]^T$: (left panel) Root-averaged error $\sigma_k$ as a function of iteration number $k$ for FP iteration \cref{eq:fixed-point} and AA(1) iteration \cref{eq:anderson-1-step}. (right panel) AA(1) coefficient $\beta_k$ and error ratio $e_k/e_{k-1}$ as a function of iteration number $k$.}
\label{fig:simple-beta}
\end{figure}

\cref{fig:simple-beta} (left panel) shows convergence curves for the root-averaged error
\begin{equation}\label{eq:sigma}
  \sigma_k = \|x^*-x_k\|^{ \frac{1}{k}},
\end{equation}
of both the FP iteration \cref{eq:q-linear-2x2-simple} and its AA(1) acceleration \cref{eq:anderson-1-step},
for initial condition $x_0=[0.2,0.1]^T$.
It is easy to see that, for all initial conditions, except when $x_0$ lies in the eigenvector direction of $\lambda_2=1/3$, $\sigma_k$ for FP iteration \cref{eq:q-linear-2x2-simple} must converge to $\rho(M)=\lambda_1=2/3$; the FP $\sigma_k$ convergence curve in \cref{fig:simple-beta} is consistent with this, and confirms that the sequence $\{x_k\}$ generated by FP iteration \cref{eq:q-linear-2x2-simple} converges r-linearly with convergence factor $\rho_{\{x_k\}}=\rho(M)=\lambda_1$. This is also consistent with \cref{thm:Ostrowski-Theorem}, with $q'(x^*)=M$.

The left panel of \cref{fig:simple-beta} also indicates that the AA(1) sequence $\{x_k\}$ converges r-linearly:
$\sigma_k$ for AA(1) appears to converge to a value $\in (0,1)$ that is smaller than $\rho(M)=2/3$, indicating asymptotic acceleration of
FP iteration \cref{eq:q-linear-2x2-simple} by AA(1).
The right panel of \cref{fig:simple-beta} shows, perhaps surprisingly, that the $\beta_k$ sequence of AA(1) does not converge as $k \rightarrow \infty$: it oscillates as $x_k$ converges to $x^*$. The figure indicates that this is related to oscillations in the error ratio $e_k/e_{k-1}$: the error ratio $e_k/e_{k-1}$ does not converge but oscillates as
$k \rightarrow \infty$ and $x_k \rightarrow x^*$, reflecting the well-known fact that q-linear convergence is often not obtained in cases where \cref{thm:Ostrowski-Theorem} guarantees r-linear convergence. Further numerical results in this paper will show that this convergence behavior is generic, both in the case of linear and nonlinear iterations \cref{eq:fixed-point}, and also for window size $m>1$: in most cases, $\{x_k\}$ for AA($m$) applied to \cref{eq:fixed-point} converges r-linearly, but $\boldsymbol{\beta}^{(k)}$ oscillates as $k \rightarrow \infty$. This paper will provide analysis of the AA($m$) fixed-point function $\Psi(\boldsymbol{z})$ in iteration \cref{eq:AA-fixed-point} that sheds light on the mechanism by which AA($m$) can converge r-linearly while $\boldsymbol{\beta}^{(k)}$ does not converge.

\begin{figure}[h]
\centering
\includegraphics[width=.49\textwidth]{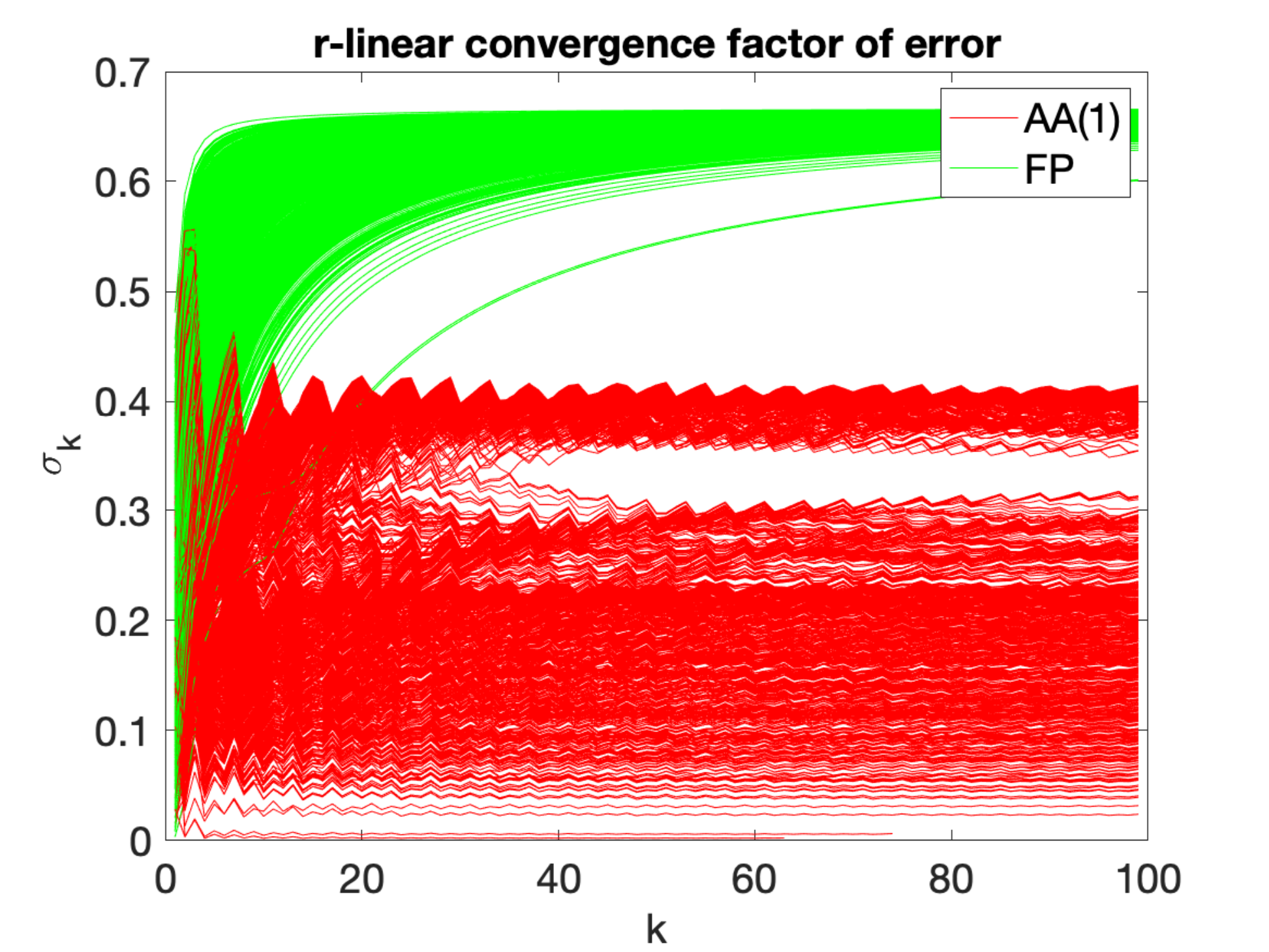}
\includegraphics[width=.49\textwidth]{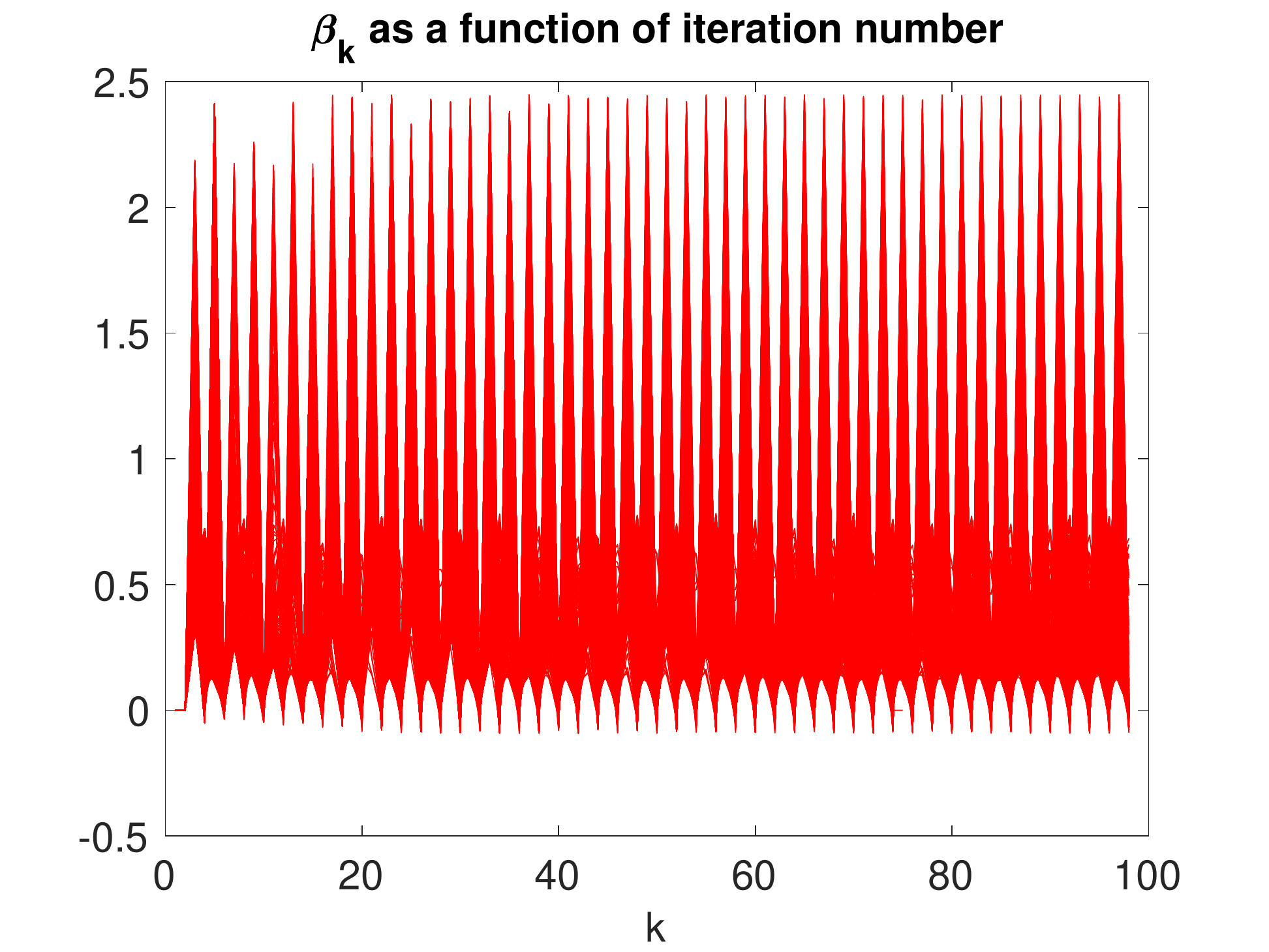}
\includegraphics[width=.325\textwidth]{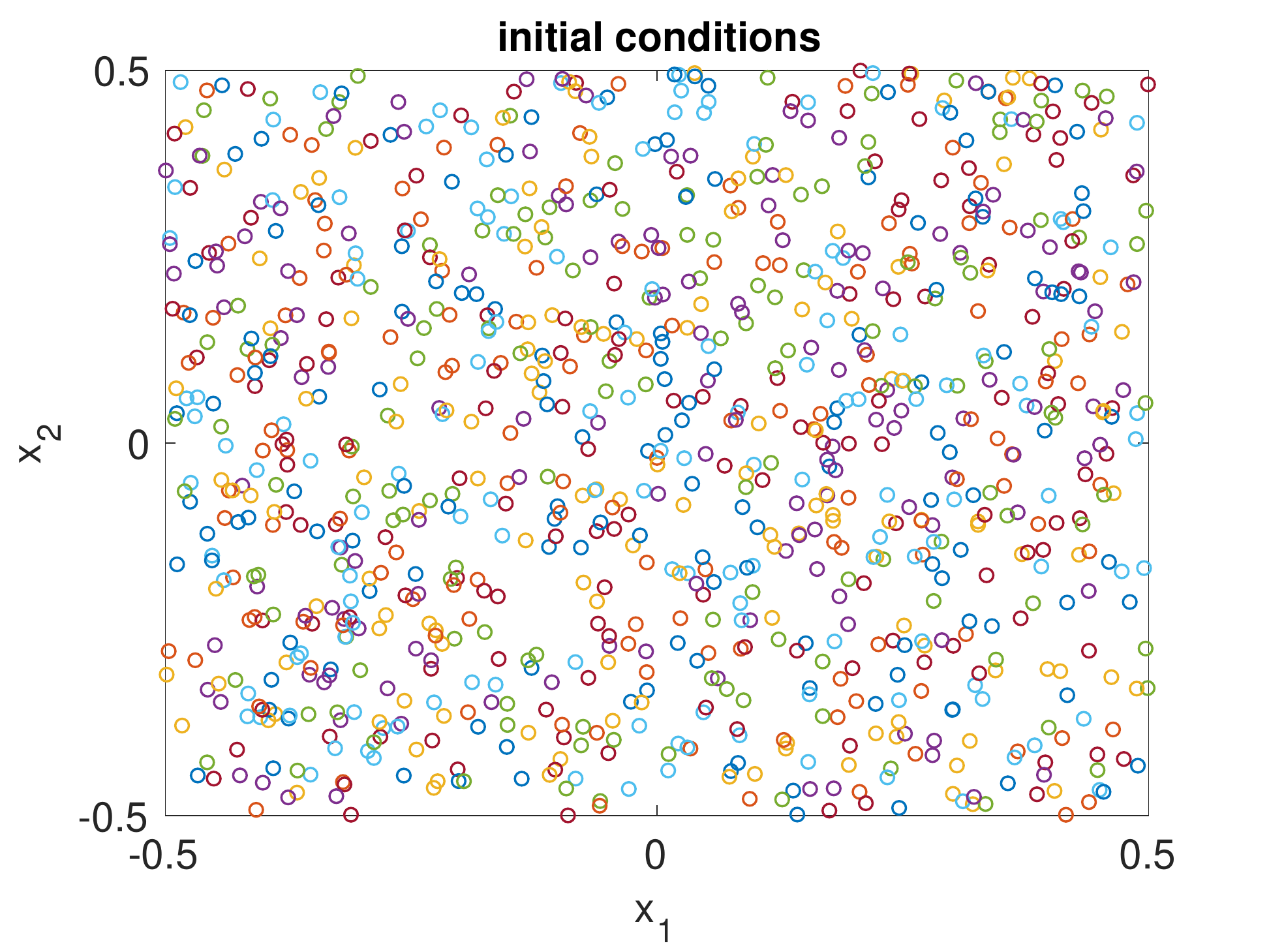}
\includegraphics[width=.325\textwidth]{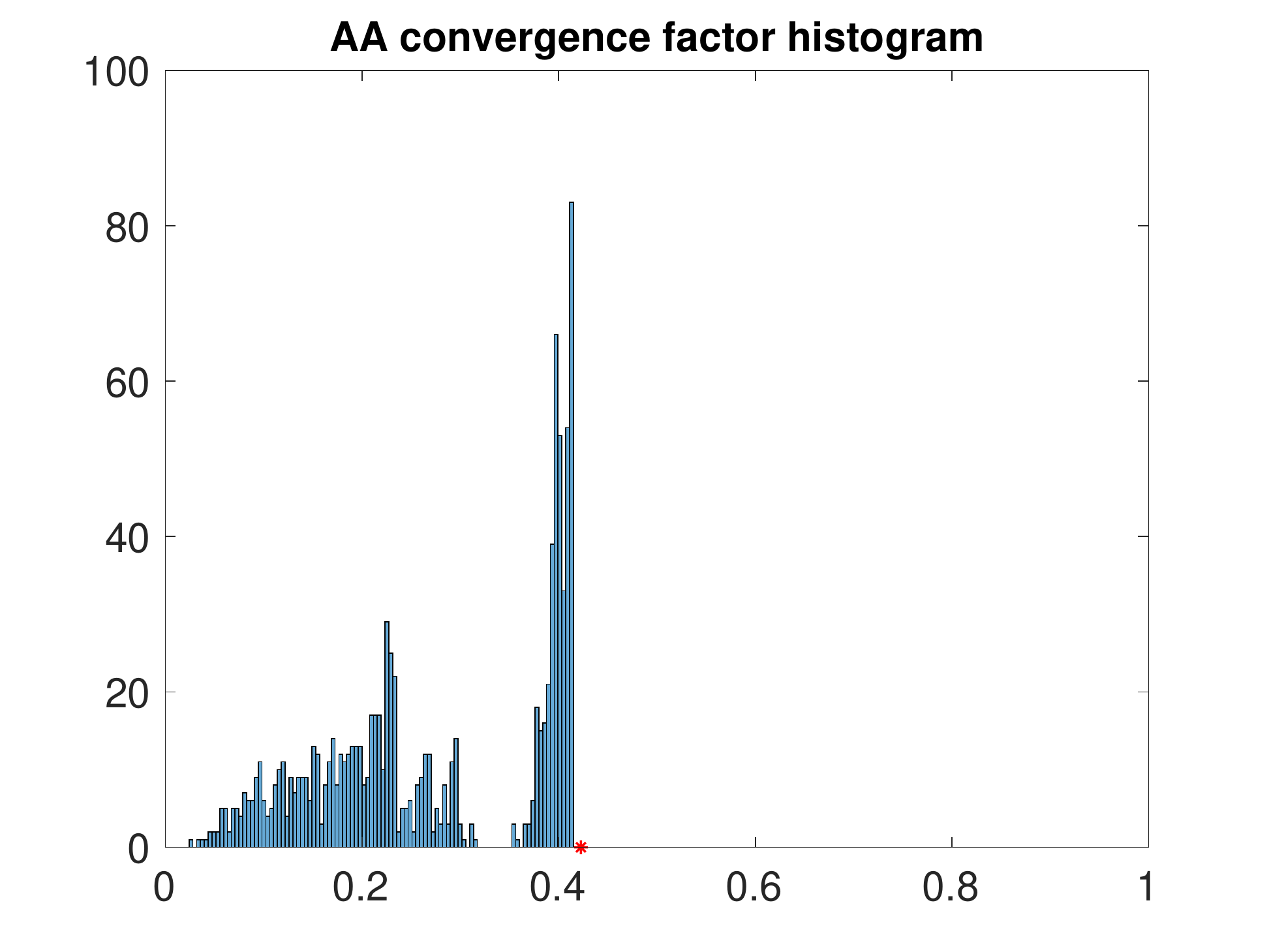}
\includegraphics[width=.325\textwidth]{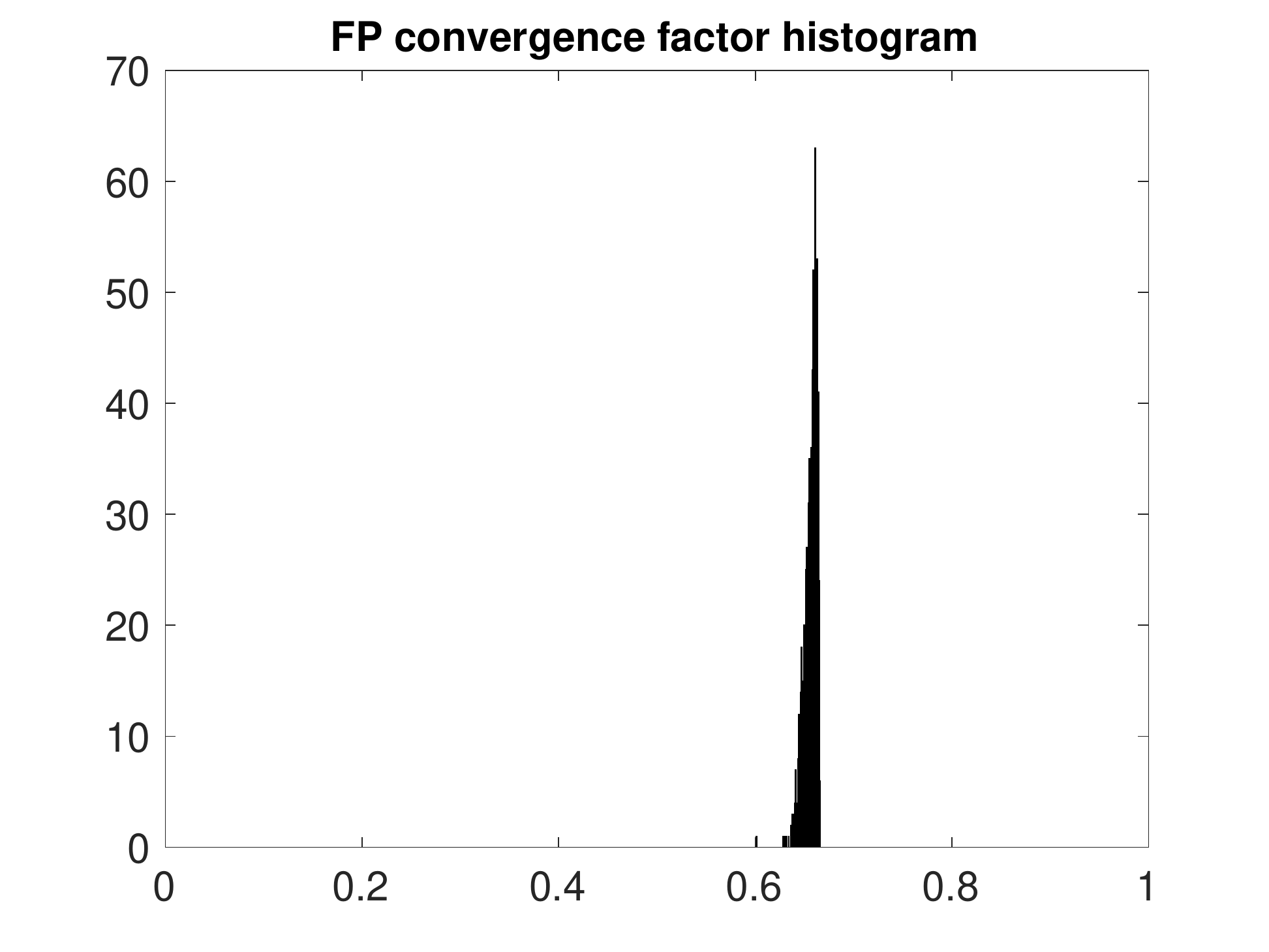}
\caption{Convergence behavior for \cref{prob:linear2x2} with 1,000 random initial guesses. The green and red $\sigma_k$ curves in the top left panel correspond to the random initial conditions $x_0$ indicated in the top right panel. The middle panels show histograms for the final values of the red and green curves in the top-left panel at iteration $k=100$. The bottom panel shows $\beta_k$ sequences for the 1,000 initial conditions.} \label{fig:simple-MC}
\end{figure}

Next, \cref{fig:simple-MC} shows further numerical results for the simple \cref{prob:linear2x2} of \cref{eq:q-linear-2x2-simple}
that identify additional convergence properties of AA(1) which will be investigated in this paper. In the tests of \cref{fig:simple-MC}, we run the FP iteration \cref{eq:q-linear-2x2-simple} and its AA(1) acceleration \cref{eq:anderson-1-step} for 1,000 initial conditions $x_0$ that are chosen uniformly randomly within the square $[-0.25,0.25]^2$.
As expected, for the FP iteration \cref{eq:q-linear-2x2-simple} $\sigma_k$ converges to $\rho(M)=\lambda_1=2/3$ for all random initial guesses, corresponding to $S$ from \cref{eq:S-definition} satisfying $|S|=0$. For the AA(1) acceleration \cref{eq:anderson-1-step}, however, the numerical results indicate that the iteration sequences $\{x_k\}$ still each converge r-linearly, but the r-linear convergence factors $\rho_{\{x_k\}}$ now strongly depend on the initial condition, indicating that $|S|>0$ for AA(1) applied to the linear \cref{prob:linear2x2}.
Furthermore, the $\beta_k$ convergence curves for AA(1) depend on the initial condition, but they oscillate and do not converge.
And finally, the numerical results suggest that the accelerated iteration \cref{eq:AA-fixed-point} for \cref{prob:linear2x2}
may have an asymptotic r-linear convergence factor $\rho_{\Psi,x^*}$, see \cref{def:rho-method},
that is strictly smaller than $\rho_{q,x^*}=\rho(M)=2/3$ of iteration \cref{eq:fixed-point}, indicating an improved AA(1) asymptotic converge factor. As we explain in some more detail below, there are no known theoretical results that can quantify the asymptotic convergence improvement of AA($m$) compared to iteration \cref{eq:fixed-point}, not even for simple specific functions $q(x)$ such as the 2 $\times 2$ linear case of \cref{prob:linear2x2}. The numerical results in \cref{fig:simple-MC} are an indication that an upper bound for $\rho_{\Psi,x^*}$ that is $< \rho_{q,x^*}$ should exist for \cref{prob:linear2x2}, and further numerical results in this paper suggest the same for higher-dimensional and nonlinear iterations \cref{eq:fixed-point}. As explained below, this open question motivates the analysis of the AA($m$) fixed-point function $\Psi(\boldsymbol{z})$ in iteration \cref{eq:AA-fixed-point} that is the subject of this paper.

At this point it is useful to recall in some detail what is known about convergence of AA($m$).
While Anderson acceleration method \cref{eq:AA-iteration} dates back to 1965 \cite{anderson1965iterative} and has since been used to speed up convergence for fixed-point methods in many areas of scientific computing with often excellent results, very little was known about the convergence of AA($m$) until the 2015 paper \cite{toth2015convergence}. In this paper, Toth and Kelley proved two convergence results.
First, they showed that, for the linear case \cref{eq:fixed-point-linear} where $q(x)=Mx+b$, if $\|M\|=c<1$, $x_k$ converges at least r-linearly with r-linear convergence factor not worse than $c$, for any initial guess. This result is important in that it establishes convergence of AA($m$), but it does not provide information on AA($m$) actually accelerating the fixed-point convergence asymptotically. Also, it only covers the case where $\|M\|=c<1$, and thus excludes practically relevant cases such as when $\rho(M)<1$ with $\|M\|>1$.
Second, \cite{toth2015convergence} also considered the nonlinear case. Assuming that the AA($m$) coefficients are bounded, i.e., $\sum_{i=1}^{m}|\beta^{(k)}_i| < c_\beta \ \forall k$, that $q(x)$ is differentiable with $\|q'(x)\|\leq c<1$, and that $q'(x)$ is Lipschitz continuous,  it was shown in \cite{toth2015convergence} that $x_k$ converges at least r-linearly with r-linear convergence factor not worse than $c$, for any initial guess sufficiently close to $x^*$. This result, however, also does not show an asymptotic improvement over $c$, and the boundedness of the $\beta^{(k)}_i$ remains as a strong assumption.
As far as we are aware, no further results have been obtained that can quantify the improvement in linear asymptotic convergence factor $\rho_{\Psi,x^*}$ that is often observed for AA($m$) with finite $m$, compared to the linear asymptotic convergence factor $\rho_{q,x^*}$ of iteration \cref{eq:fixed-point} when it converges r-linearly.

In more recent work, Evans et al.\ \cite{evans2020proof} were able to quantify the per-iteration convergence improvement of AA($m$).
They showed that, to first order, the convergence gain provided by AA in step $k$ is quantified by a factor $\theta_k \le 1$ that equals the ratio of the optimal value defined in \cref{eq:Andersonbetas} to $\|r(x_k)\|_2$.
However, $\theta_k \le 1$ may oscillate, and it is not clear how $\theta_k$ may be evaluated or bounded in practice or how it may translate to an improved linear asymptotic convergence factor $\rho_{\Psi,x^*}$ for AA($m$) compared to $\rho_{q,x^*}$ as observed, for example, in the numerical results of \cref{fig:simple-MC}.

The strong dependence in \cref{fig:simple-MC} of the AA(1) r-linear convergence factors $\rho_{\{x_k\}}$ on the initial guess may, at first, seem surprising. In light of \cref{thm:Ostrowski-Theorem}, one would expect, for the accelerated iteration \cref{eq:AA-fixed-point}, that the Jacobian of $\Psi(\boldsymbol{z})$ evaluated at the fixed point $\boldsymbol{z}^*$ would determine the linear convergence factor $\rho_{\{x_k\}}$ for most initial guesses, if $\Psi'(\boldsymbol{z})$ were differentiable at $\boldsymbol{z}^*$. The strong dependence of the r-linear convergence factors $\rho_{\{x_k\}}$ on the initial guess in \cref{fig:simple-MC} suggests otherwise. To shed light on numerical observations as in \cref{fig:simple-MC}, we analyze in this paper the differentiability properties of the AA($m$) fixed-point iteration function $\Psi(\boldsymbol{z})$ in iteration \cref{eq:AA-fixed-point}. We find, indeed, that, while being Lipschitz continuous, $\Psi(\boldsymbol{z})$ is not differentiable at $\boldsymbol{z}^*$. Further analysis reveals, however, that $\Psi(\boldsymbol{z})$ is directionally differentiable at $\boldsymbol{z}^*$ in all directions, and we obtain closed-form expressions for these directional derivatives. This allows us to compute the Lipschitz constant of $\Psi(\boldsymbol{z})$ at $\boldsymbol{z}^*$, and to investigate whether this Lipschitz constant may relate to numerically observed AA($m$) convergence factors as in \cref{fig:simple-MC}.

In addition to the fixed-point analysis presented in this paper, it has to be noted that further insight in the convergence behavior revealed by numerical tests as in \cref{fig:simple-MC} may be obtained from formulating AA($m$) with finite $m$ as a Krylov subspace method and deriving further theoretical properties of AA($m$) iterations from that formulation. Such a Krylov formulation is developed for AA($m$) in a companion paper \cite{AAKrylov} to the current paper. This leads to further explanations for numerical observations as in \cref{fig:simple-MC}, including, for example, the apparent gap in the AA(1) $\sigma_k$ spectrum that can be observed in the top-left panel of \cref{fig:simple-MC}.

The remainder of this paper is organized as follows.
\Cref{sec:AA1} provides a detailed analysis of the continuity and differentiability of the AA(1) fixed-point iteration function $\Psi(\boldsymbol{z})$ in \cref{eq:AA-fixed-point}, and \Cref{sec:AAm} extends this analysis to AA($m$). We investigate the continuity of the AA coefficients $\beta^{(k)}_i$ and of the fixed-point iteration function $\Psi(\boldsymbol{z})$ at the fixed point $\boldsymbol{z}^*$, and consider Lipschitz continuity, directional differentiability and differentiability of $\Psi(\boldsymbol{z})$. Some of the longer proofs are relegated to appendices.
\Cref{sec:numerics} contains numerical results that further illustrate how our theoretical findings relate to the asymptotic convergence of AA($m$), both for linear and nonlinear iterations \cref{eq:fixed-point} and in higher dimensions than in the simple $2 \times 2$ problem of \cref{fig:simple-MC}. To place our results in a broader context, we also compare AA($m$) convergence behavior in the linear case with GMRES and restarted GMRES($m$). We conclude in \Cref{sec:disc}.



\section{Analysis of $\Psi(\boldsymbol{z})$ for AA(1)}
\label{sec:AA1}
In this paper we analyze the continuity and differentiability properties of the AA($m$) fixed-point iteration function $\Psi(\boldsymbol{z})$ in iteration \cref{eq:AA-fixed-point}. To aid the analysis of $\Psi(\boldsymbol{z})$, it will also be useful to study the continuity of the AA($m$) coefficients $\boldsymbol{\beta}^{(k)}$ given in \cref{eq:AAm-beta-form-pseudo}.
We first consider, in this section, the case where $m=1$. We present results that cover general, nonlinear iteration functions $q(x)$. In \Cref{sec:AAm} we extend these results to $m>1$.

Throughout this paper, we assume that $q(x)$ in iteration \cref{eq:fixed-point} is continuously differentiable in a neighborhood of $x^*\in\mathbb{R}^{n}$, such that the Jacobian matrix $q'(x)\in\mathbb{R}^{n \times n}$ exists and is continous.

First, for AA(1), let us redefine $\Psi(\boldsymbol{z})$ in \cref{eq:AA(1)-system} for iteration \cref{eq:AA-fixed-point} as follows:
\begin{equation}\label{eq:Psi1}
 \Psi(\boldsymbol{z}) =
  \begin{bmatrix}
  q(x) +\beta(\boldsymbol{z}) (q(x)-q(y))\\
   x
  \end{bmatrix},
\end{equation}
where
\begin{equation}\label{eq:z1}
\boldsymbol{z} =\begin{bmatrix} x\\ y\end{bmatrix} \quad \in\mathbb{R}^{2n},
\end{equation}
and $\beta_k$ in \cref{eq:AA-1-step-beta} is written as
\begin{equation}\label{eq:beta1}
\beta(\boldsymbol{z}) =
\begin{cases}
  \ds \ds \frac{-r^T(x)(r(x)-r(y))}{\|r(x)-r(y)\|^2}, & \text{if}\,\, r(x)\neq r(y), \\
   0,  & \text{if}\,\, r(x)=r(y),
\end{cases}
\end{equation}
with $r(x) =x-q(x)$.
Note that when $\boldsymbol{z}=\boldsymbol{z}^*=\ds \begin{bmatrix} x^*\\ x^*\end{bmatrix}$, $\Psi(\boldsymbol{z})=\boldsymbol{z}$.

For the linear case, we define $q(x)$ in \cref{eq:fixed-point} as
\begin{equation}\label{eq:linear-q-form}
  q(x) = Mx + b, \quad M \in\mathbb{R}^{n \times n}, \quad x, \, b \in\mathbb{R}^{n}
\end{equation}
where one seeks to solve $A\,x=b$, with
\begin{equation}
A=I-M=I-q'(x).
\end{equation}
In the linear case, we will assume that matrix $A=I-q'(x)$ in $A\,x=b$ is nonsingular.
Similarly, we will usually assume in the nonlinear case that $r'(x)=I-q'(x)$ is nonsingular. We also exclude the trivial case where $A=I$ and $M=0$, or, more generally, $q'(x)=0$.

In the linear case \cref{eq:beta1} simplifies to
\begin{equation}
\beta(\boldsymbol{z}) =
\begin{cases}
  \ds \frac{-(Ax-b)^TA(x-y)}{(x-y)^TA^TA(x-y)}, & \text{if}\,\, x\neq y, \\
   0,  & \text{if}\,\, x=y,
\end{cases}
\end{equation}
and when $x\neq y$
\begin{equation}\label{eq:Psilin}
 \Psi(\boldsymbol{z}) =
  \begin{bmatrix}
  (I-A)x+b- \ds \frac{(Ax-b)^TA(x-y)}{(x-y)^TA^TA(x-y)}(I-A)(x-y)\\
   x
  \end{bmatrix}.
\end{equation}
When $x= y$,
\begin{equation}\label{eq:form-Psi-x=y-linear}
 \Psi(\boldsymbol{z}) =
  \begin{bmatrix}
  q(x)\\
   x
  \end{bmatrix}=\begin{bmatrix}
 Mx+b\\
   x
  \end{bmatrix}.
\end{equation}

Before providing our detailed analysis of the differentiability properties of $\Psi(\boldsymbol{z})$, we summarize our results in \cref{tab:continuity-differentiability-Lip}. While the proof for some of these results is elementary, the table provides a complete overview of the differentiability properties of $\Psi(\boldsymbol{z})$ which are useful to understand the convergence behavior of AA(1) viewed as the fixed-point method \cref{eq:AA-fixed-point}.

\newsavebox{\smlmat}
\savebox{\smlmat}{$\small \boldsymbol{z} =\left[ \begin{array}{c} x \\ y \end{array}\right]$}
\begin{table}[h!]
 \caption{Continuity and differentiability properties of $\beta(\boldsymbol{z})$ and $\Psi(\boldsymbol{z})$ at~\usebox{\smlmat}
 for AA($m$) iteration \cref{eq:AA-fixed-point} with $m=1$, where $\beta(\boldsymbol{z})$
 and $\Psi(\boldsymbol{z})$ are given by \cref{eq:beta1}
 and \cref{eq:Psi1}.
 }
\centering
\begin{tabular}{|l|c|c|c|}
\hline
                                      &$r(x)\neq r(y)$        &$x=y, \ r(x) \neq 0$     & $x=y=x^*$        \\
\hline  continuity of $\beta(\boldsymbol{z})$        &$\surd$          &$\times$          & $\times$  \\
\hline  continuity of $\Psi(\boldsymbol{z})$       &$\surd$          &$\times$          & $\surd$  \\
\hline Lipschitz continuity of $\Psi(\boldsymbol{z})$          &$\surd$          &$\times$          & $\surd$  \\
\hline Gateaux-differentiability of $\Psi(\boldsymbol{z})$        &$\surd$          &$\times$          &$\surd$  \\
\hline differentiability of $\Psi(\boldsymbol{z})$          &$\surd$          &$\times$          & $\times$  \\
\hline
\end{tabular}\label{tab:continuity-differentiability-Lip}
\end{table}

%
%
\subsection{Continuity of $\beta(\boldsymbol{z})$ for AA(1)}
\begin{proposition}
   $\beta(\boldsymbol{z})$ in \cref{eq:beta1} is continuous at $\boldsymbol{z} =\begin{bmatrix} x\\ y\end{bmatrix}$ when $r(x)\neq r(y)$.
\end{proposition}
\begin{proof}
Since $r(x)$ is a continuous function, $r(x)-r(y)$ and $\|(r(x)-r(y))\|^2$ are continuous functions. It follows that $\beta(\boldsymbol{z})$ in \cref{eq:beta1} is continuous when $r(x)\neq r(y)$.
\end{proof}


\begin{proposition}\label{prop:betaxx}
$\beta(\boldsymbol{z})$ in \cref{eq:beta1} is not continuous at $\boldsymbol{z} =\begin{bmatrix} x\\ y\end{bmatrix}$ when $x= y$ with $r(x) \neq 0$.
\end{proposition}
\begin{proof}
Consider  $\boldsymbol{z}=\begin{bmatrix} x\\ y+d\end{bmatrix},$ where $x= y$ with $r(x) \neq 0$.
It is sufficient to find a path for $d$ along which $\beta(\boldsymbol{z})$ is not continuous as $d \rightarrow 0$.
Consider the case where $d=\epsilon\, e$, with $e$ a unit vector
in $\mathbb{R}^n$ and $r(x)-r(y+d(\epsilon))\neq 0$. Then
\begin{align*}
 \beta(\boldsymbol{z}(\epsilon))& = - \ds \frac{r(x)^T(r(x)-r(y+d(\epsilon)))}{\|r(x)-r(y+d(\epsilon))\|^2},\\
&=   - \ds \frac{r(x)^T(r(x)-r(y+\epsilon e))}{\|r(x)-r(y+\epsilon e)\|^2}.
\end{align*}
Since $r(y+\epsilon e)=r(y)+ r'(y)\epsilon e +Q(\epsilon e) \epsilon e$ with $\lim_{\epsilon \rightarrow 0} Q(\epsilon e)=0$,
we have
\begin{align*}
\lim_{\epsilon \rightarrow 0} \beta(\boldsymbol{z}(\epsilon))& =
\lim_{\epsilon \rightarrow 0} \ds \frac{r(x)^T (r'(y) \epsilon e + Q(\epsilon e) \epsilon e) }{\|r'(y)\epsilon e +Q(\epsilon e) \epsilon e\|^2},
\end{align*}
where $r'(y)e +Q(\epsilon e) e\neq 0$ for sufficiently small $\epsilon$, since $r'(y)$ is nonsingular. Since $r'(y)$ is nonsingular, we have that $r(x)^T r'(y) e \neq 0$ for all unit vectors $e$, except for the unit vectors orthogonal to
$r'(y)^Tr(x)\neq 0$.
So for almost all unit vectors $e$ we have that
\begin{align*}
\beta(\boldsymbol{z}(\epsilon))& =
\ds \frac{r(x)^T (r'(y) e + Q(\epsilon e) e) }{\epsilon \ \|r'(y) e +Q(\epsilon e) e\|^2} \rightarrow \pm \infty \textrm{ as } \epsilon \rightarrow 0.
\end{align*}
Thus, $\beta(\boldsymbol{z})$ is not continuous at $\boldsymbol{z} =\begin{bmatrix} x\\ y\end{bmatrix}$ when $x= y$  with $r(x) \neq 0$.
\end{proof}
\begin{remark}
\cref{prop:betaxx} also holds in the more general case when $r(x)= r(y) \neq r(x^*)$.
\end{remark}
\begin{proposition}\label{prop:betax*}
  $\beta(\boldsymbol{z})$ in \cref{eq:beta1} is not continuous at $\boldsymbol{z} =\begin{bmatrix} x^*\\ x^*\end{bmatrix}$.
\end{proposition}
\begin{proof}
We investigate the limiting behavior of $\beta(\boldsymbol{z})$ along radial paths approaching $\boldsymbol{z}^* =\begin{bmatrix} x^*\\ x^*\end{bmatrix}$. We set $\boldsymbol{z}(\epsilon)=\begin{bmatrix} x^*+\epsilon d_1\\ x^*+\epsilon d_2\end{bmatrix}$.
Note that $\beta(\boldsymbol{z}(\epsilon))=0$ when $d_1 = d_2$. When $d_1 \neq d_2$ and $r(x^*+\epsilon d_1)-r(x^*+\epsilon d_2)\neq 0$, we have
\begin{align*}
\beta(\boldsymbol{z}(\epsilon))& = - \ds \frac{r(x^*+\epsilon d_1)^T( r(x^*+\epsilon d_1)-r(x^*+\epsilon d_2))}{\|r(x^*+\epsilon d_1)-r(x^*+\epsilon d_2)\|^2},
\end{align*}
and, using $r(x^*+\epsilon d)=r(x^*)+ r'(x^*)\epsilon d +Q(\epsilon d) \epsilon d$ with $\lim_{\epsilon \rightarrow 0} Q(\epsilon d)=0$, we obtain
\begin{align*}
\beta(\boldsymbol{z}(\epsilon))& = - \ds \frac{(r'(x^*)\epsilon d_1 +Q(\epsilon d_1) \epsilon d_1)^T( r'(x^*)\epsilon (d_1-d_2) +Q(\epsilon d_1) \epsilon d_1 - Q(\epsilon d_2) \epsilon d_2)}{\| r'(x^*)\epsilon (d_1-d_2) +Q(\epsilon d_1) \epsilon d_1 - Q(\epsilon d_2) \epsilon d_2 \|^2}\\
&=- \ds \frac{(r'(x^*)d_1 +Q(\epsilon d_1)  d_1)^T( r'(x^*) (d_1-d_2) +Q(\epsilon d_1)  d_1 - Q(\epsilon d_2)  d_2)}{\| r'(x^*) (d_1-d_2) +Q(\epsilon d_1)  d_1 - Q(\epsilon d_2)  d_2 \|^2},
\end{align*}
so
\begin{align*}
\lim_{\epsilon \rightarrow 0} \beta(\boldsymbol{z}(\epsilon))
&=- \ds \frac{(r'(x^*)d_1)^T r'(x^*) (d_1-d_2)}{\| r'(x^*) (d_1-d_2)\|^2}.
\end{align*}
Note that, for example, $\lim_{\epsilon \rightarrow 0} \beta(\boldsymbol{z}(\epsilon))=0$ when $d_1=0$ and when $d_1=d_2$, and
$\lim_{\epsilon \rightarrow 0} \beta(\boldsymbol{z}(\epsilon))=-1$ when $d_2=0$.
Since the limit depends on the choice of $d_1$ and $d_2$, $\beta(\boldsymbol{z})$ is not continuous at $\boldsymbol{z} =\begin{bmatrix} x^*\\ x^*\end{bmatrix}$.
\end{proof}

\begin{remark}
It is interesting to note the difference in the limiting behavior of $\beta(\boldsymbol{z})$ along radial paths in the proofs of \cref{prop:betaxx} and \cref{prop:betax*}. The proof of \cref{prop:betax*} shows that, when approaching $\boldsymbol{z} =\begin{bmatrix} x^*\\ x^*\end{bmatrix}$ along radial paths, $\lim_{\epsilon \rightarrow 0} \beta(\boldsymbol{z}(\epsilon))$ is finite for any fixed $d_1$ and $d_2$. This is in contrast to the limiting behavior of $\beta(\boldsymbol{z})$ when approaching $\boldsymbol{z} =\begin{bmatrix} x\\ x\end{bmatrix}$ along radial paths with $x \neq x^*$ in \cref{prop:betaxx}, where $\beta(\boldsymbol{z}(\epsilon))$ grows without bound. The AA($m$) convergence proof in \cite{toth2015convergence} relies on the unproven assumption that $|\beta(\boldsymbol{z}_k)|$ is bounded above as $x_k \rightarrow x^*$.
\end{remark}


%
%
\subsection{Continuity and differentiability of $\Psi(\boldsymbol{z})$ for AA(1)}
In this section, we discuss the continuity and differentiability of $\Psi(\boldsymbol{z})$ at $\boldsymbol{z} =\begin{bmatrix} x\\ y\end{bmatrix}$. We consider the three cases of \cref{tab:continuity-differentiability-Lip}:   $r(x)\neq r(y)$, $x = y$ with $r(x)\neq 0$, and $x=y=x^*$.

\begin{proposition}\label{prop:Psi-x-not-y}
$\Psi(\boldsymbol{z})$ in \cref{eq:Psi1} is continuous and differentiable at $\boldsymbol{z} =\begin{bmatrix} x\\ y\end{bmatrix}$ when $r(x)\neq r(y)$.
\end{proposition}
\begin{proof}
Recall that
\begin{equation*}
 \Psi(\boldsymbol{z}) =
  \begin{bmatrix}
  q(x)+\beta(\boldsymbol{z})(q(x)-q(y))\\
   x
  \end{bmatrix},
\end{equation*}
where  $\beta(\boldsymbol{z}) =\ds \frac{-r^T(x)(r(x)-r(y))}{\|r(x)-r(y)\|^2}$ when $r(x)\neq r(y)$.
Since $q(x), q(y), r(x)$, and $r(y)$ are continuous and differentiable functions at $x$ and $y$,
$\Psi(\boldsymbol{z})$ is continuous and differentiable at $\boldsymbol{z} =\begin{bmatrix} x\\ y\end{bmatrix}$ when $r(x)\neq r(y)$.
\end{proof}

\begin{proposition}\label{prop:Psixx}
$\Psi(\boldsymbol{z})$ in \cref{eq:Psi1} is not continuous and not differentiable at $\boldsymbol{z} =\begin{bmatrix} x\\ y\end{bmatrix}$ when $x= y$ with $r(x) \neq 0$.
\end{proposition}
\begin{proof}
Let $\boldsymbol{z}(\epsilon)=\begin{bmatrix} x\\ y+\epsilon e\end{bmatrix},$ where $x= y$ with $r(x) \neq 0$, and $e$ is a unit vector in $\mathbb{R}^n$.
From the proof of \cref{prop:betaxx}, we have
\begin{align*}
\beta(\boldsymbol{z}(\epsilon))& =
\ds \frac{r(x)^T (r'(y) e + Q(\epsilon e) e) }{\epsilon \ \|r'(y) e +Q(\epsilon e) e\|^2}.
\end{align*}
Plugging this into \cref{eq:Psi1} and using $q(y+\epsilon e)=q(y) + q'(y) \epsilon e + P(\epsilon e) \epsilon e$ with $\lim_{\epsilon \rightarrow 0} P(\epsilon e)=0$, we obtain
\begin{align*}
 \Psi(\boldsymbol{z}(\epsilon)) &=
  \begin{bmatrix}
  q(x)+\ds \frac{r(x)^T (r'(y) e + Q(\epsilon e) e) }{\epsilon \ \|r'(y) e +Q(\epsilon e) e\|^2}
    (q(x)-q(y)-  q'(y) \epsilon e - P(\epsilon e) \epsilon e)\\
   x
  \end{bmatrix},\\
  &=
  \begin{bmatrix}
  q(x)+\ds \frac{r(x)^T (r'(x) e + Q(\epsilon e) e) }{ \ \|r'(x) e +Q(\epsilon e) e\|^2}
    (q'(x) e - P(\epsilon e) e)\\
   x
  \end{bmatrix},
\end{align*}
and
\begin{align*}
\lim_{\epsilon \rightarrow 0} \Psi(\boldsymbol{z}(\epsilon)) &=
  \begin{bmatrix}
  q(x)+\ds \frac{r(x)^T r'(x) e}{ \ \|r'(x) e\|^2}
    q'(x) e\\
   x
  \end{bmatrix}.
\end{align*}
While the limit is finite for any $e$, it depends on the choice of $e$, so $\Psi(\boldsymbol{z})$ is not continuous at $\boldsymbol{z}=\begin{bmatrix}x \\ x\end{bmatrix}$ where $r(x)\neq 0$. It follows that  $\Psi(\boldsymbol{z})$ is not differentiable at $\boldsymbol{z}=\begin{bmatrix}x \\ x\end{bmatrix}$.
\end{proof}

The following proposition establishes the Lipschitz continuity of $\Psi(\boldsymbol{z})$ at $\boldsymbol{z}^* =\begin{bmatrix} x^*\\ x^*\end{bmatrix}$. The differentiability of $\Psi(\boldsymbol{z})$ at $\boldsymbol{z}^*$ is investigated in the next subsection.

\begin{proposition}\label{prop:Psix*}
$\Psi(\boldsymbol{z})$ in \cref{eq:Psi1} is Lipschitz continuous at $\boldsymbol{z} =\begin{bmatrix} x\\ y\end{bmatrix}$ when $x= y= x^*$ with global Lipschitz constant $L=( \|A^{-1}\| \|A\| + 1 ) \|I-A\|+1$ in the linear case, and with local Lipschitz constant $L=3+(4+4/c_r)\|r'(x^*)\|$ in the nonlinear case, where $c_r$ is a problem-dependent constant.
\end{proposition}
\begin{proof}
See \cref{app:proof-psi-lipsch}.

\end{proof}
\subsection{Differentiability of $\Psi(\boldsymbol{z})$ at $\boldsymbol{z}^*$ for AA(1)}
In this subsection we investigate the differentiability of $\Psi(\boldsymbol{z})$ at $\boldsymbol{z}^*$.
We first consider the directional derivatives of $\Psi(\boldsymbol{z})$ at $\boldsymbol{z}^*$.
\begin{definition}[Directional Derivative]
  Let $F: U\subset \mathbb{R}^n \longrightarrow \mathbb{R}^m$ be a function on the open
set $U$. We call $\mathfrak{D}F(x, d)$ defined by
\begin{equation}\label{def-direction-diff}
 \mathfrak{ D}F(x,d) = \lim_{h\downarrow 0}\ds \frac{F(x+hd)-F(x)}{h}
\end{equation}
the directional derivative of $F$ at $x$ in direction $d$ if the limit exists. We say $F(x)$ is Gateaux differentiable in $x$ if the directional derivative of $F$ exists in $x$ for all directions.
\end{definition}
Note that, if $F$ is differentiable at $x$ with Jacobian $F'(x)$, then $\mathfrak{ D}F(x,d)=F'(x)d$.

\begin{theorem}\label{thm:x=y-directional-dif}
Consider $\Psi(\boldsymbol{z})$ at $\boldsymbol{z}^*=\begin{bmatrix} x^*\\ x^*\end{bmatrix}$ and direction $\boldsymbol{d}=\begin{bmatrix} d_1\\ d_2\end{bmatrix}$.
Let  $M=q'(x^*)$ and $A= I-q'(x^*)$.
Then the directional derivative of $\Psi(\boldsymbol{z})$ at $\boldsymbol{z}^*$ in direction $\boldsymbol{d}$ is given by
\begin{equation}\label{AA(1)-direction-diff}
  \mathfrak{ D} \Psi(\boldsymbol{z}^*,\boldsymbol{ d})
  = \begin{bmatrix}
  (1+\widehat{\beta}(\boldsymbol{d})) M & -\widehat{\beta}(\boldsymbol{d}) M\\
  I  & 0
  \end{bmatrix}\boldsymbol{d},
\end{equation}
where
\begin{equation*}
\widehat{\beta}(\boldsymbol{d}) =
\begin{cases}
 - \ds \frac{d_1^TA^TA(d_1-d_2)}{(d_1-d_2)^TA^TA(d_1-d_2)}, & \text{if}\,\, d_1\neq d_2, \\
  0,  & \text{if}\,\, d_1=d_2.
\end{cases}
\end{equation*}
\end{theorem}
\begin{proof}
See \cref{app:proof-directional}.
\end{proof}
\begin{remark}\label{rmk-Psi-diff-n=1}
When $n=1$ and $d_1\neq d_2$, the result in \cref{eq:DPsi} simplifies considerably: since all quantities are scalar, $\widehat{\beta}(\boldsymbol{d})=\ds -\frac{d_1}{d_1-d_2}$ and \cref{eq:DPsi} can be rewritten as:
\begin{align*}
 \mathfrak{ D} \Psi(\boldsymbol{z}^*,\boldsymbol{ d})&=
  \begin{bmatrix}
  (1+\widehat{\beta}(\boldsymbol{d})) (1-a) & -\widehat{\beta}(\boldsymbol{d}) (1-a) \\
  1  & 0
  \end{bmatrix}\boldsymbol{d}\\
  &=
   \begin{bmatrix}
  \frac{-d_2}{d_1-d_2} (1-a) & \frac{d_1}{d_1-d_2} (1-a) \\
  1  & 0
  \end{bmatrix}\boldsymbol{d},\\
  & =\begin{bmatrix}
  0\\
  d_1
  \end{bmatrix},\\
   &  =\begin{bmatrix}
  0&0 \\
  1  & 0
  \end{bmatrix}\boldsymbol{d}.
\end{align*}
However, when $n=1$ and $d_1= d_2$, we get
\begin{align*}
 \mathfrak{ D} \Psi(\boldsymbol{z}^*,\boldsymbol{ d})&=
\begin{bmatrix}
  (1-a) &0 \\
  1  & 0
  \end{bmatrix}\boldsymbol{d}.
\end{align*}
This shows that, when $n=1$, $\Psi(\boldsymbol{z})$ is not differentiable at $\boldsymbol{z}^*$.
When $n>1$, $\Psi(\boldsymbol{z})$ is also not differentiable at $\boldsymbol{z}^*$, because
the matrix in \cref{eq:DPsi} depends on $\boldsymbol{d}$ and
$\mathfrak{D} \Psi(\boldsymbol{z}^*,\boldsymbol{d})$ cannot be written as $\Psi'(\boldsymbol{z}^*)\boldsymbol{d}$.
\end{remark}

Finally, it is interesting to consider the differentiability results of \cref{tab:continuity-differentiability-Lip} for AA(1) specifically for the scalar case, $n=1$. This is considered in \cref{app:scalar}.

\section{Analysis of $\Psi(\boldsymbol{z})$ for AA(m)}
\label{sec:AAm}
In this section, we extend the properties of AA(1) in \cref{tab:continuity-differentiability-Lip} to AA($m$).

First, let us extend $\Psi(\boldsymbol{z})$ in \cref{eq:AA(1)-system} for iteration \cref{eq:AA-fixed-point} to AA($m$) as follows:
\begin{equation}\label{eq:Psi-AAm}
 \Psi(\boldsymbol{z}) =
  \begin{bmatrix}
  q(z_{m+1}) + \sum_{j=1}^{m}\beta_{j} (q(z_{m+1})-q(z_{m+1-j}))\\
   z_{m+1}\\
   z_m \\
   \vdots \\
   z_2
  \end{bmatrix}=
   \begin{bmatrix}
  q(z_{m+1}) + Q(\boldsymbol{z})\boldsymbol{\beta}(\boldsymbol{z})\\
   z_{m+1}\\
   z_m\\
   \vdots \\
   z_2
  \end{bmatrix},
\end{equation}
where
\begin{equation}\label{eq:def-Q-general}
Q(\boldsymbol{z}) = \begin{bmatrix} q(z_{m+1})-q(z_m) &  q(z_{m+1})-q(z_{m-1}) & \ldots & q(z_{m+1})-q(z_{1}) \end{bmatrix}
\end{equation}
with
$
 \boldsymbol{z} =\begin{bmatrix} z_{m+1}^T & z_m^T  & \ldots & z_{1}^T\end{bmatrix}^T  \in\mathbb{R}^{n(m+1)},
$
and $\boldsymbol{\beta}^{(k)}$ in  \cref{eq:AAm-beta-form-pseudo} is written as
\begin{equation}\label{eq:beta-AAm}
\boldsymbol{\beta}(\boldsymbol{z}) =\begin{bmatrix} \beta_1 \\ \vdots \\ \beta_m \end{bmatrix}=
  -R(\boldsymbol{z})^{\dag}r(z_{m+1}),
\end{equation}
where
\begin{equation}\label{eq:beta-continuous}
  R(\boldsymbol{z}) = \begin{bmatrix} r(z_{m+1})-r(z_m) &  r(z_{m+1})-r(z_{m-1}) & \ldots & r(z_{m+1})-r(z_{1}) \end{bmatrix},
\end{equation}
and  $R(\boldsymbol{z})^{\dag}$ is the pseudo-inverse of $R(\boldsymbol{z})$.

Note that when $\boldsymbol{z}=\boldsymbol{z}^*=\ds \begin{bmatrix} (x^*)^T &  (x^*)^T & \ldots & (x^*)^T \end{bmatrix}^T$, $\Psi(\boldsymbol{z})=\boldsymbol{z}$.  We can  write AA($m$) as
the fixed-point iteration $\boldsymbol{z}_{k+1} =\Psi(\boldsymbol{z}_k)$ with $\Psi(\boldsymbol{z})$ as in \cref{eq:Psi-AAm}, where $\boldsymbol{z}_k =\begin{bmatrix} x_{k+m}^T & x_{k+m-1}^T & x_{k+m-2}^T & \ldots & x_{k}^T\end{bmatrix}^T$.

For convenience, we define the following operator:
\begin{equation}\label{eq:def-D-operator}
  D(\boldsymbol{z})=\begin{bmatrix} z_{m+1}-z_m &  z_{m+1}-z_{m-1} & \cdots & z_{m+1}-z_{1} \end{bmatrix}\quad \in\mathbb{R}^{n\times m}.
\end{equation}
For simplicity, we will sometimes denote $D(\boldsymbol{z}), R(\boldsymbol{z}), Q(\boldsymbol{z})$ and $\boldsymbol{\beta}(\boldsymbol{z})$ by $D,R,Q$ and $\boldsymbol{\beta}$.

 Before providing our detailed analysis of the differentiability properties of $\Psi(\boldsymbol{z})$, we summarize our results in  \cref{tab:continuity-differentiability-Lip-AAm}.  While the proof for some of these results is  elementary, the table provides a complete overview of the differentiability properties of $\Psi(\boldsymbol{z})$  which are useful to understand the convergence behavior of AA($m$) viewed as the fixed-point method \cref{eq:AA-fixed-point}.

\newsavebox{\smlmatt}
\savebox{\smlmatt}{$\small \boldsymbol{z} =\begin{bmatrix} z_{m+1}^T & z_m^T  & \ldots & z_{1}^T \end{bmatrix}^T$}
%
\begin{table}[H]
 \caption{Continuity and differentiability properties of $\boldsymbol{\beta}(\boldsymbol{z})$ and  $\Psi(\boldsymbol{z})$  at~\usebox{\smlmatt} for AA($m$) iteration
\cref{eq:AA-fixed-point}  with $m\geq 1$, where $\boldsymbol{\beta}(\boldsymbol{z})$ and  $\Psi(\boldsymbol{z})$  are given by \cref{eq:beta-AAm} and \cref{eq:Psi-AAm}. }
\centering
\begin{tabular}{|l|c|c|c|}
\hline
 &$R(\boldsymbol{z})$ has full rank            & $z_j=z \, \forall j, \ r(z) \neq 0$    & $\boldsymbol{z}=\boldsymbol{z}^*$     \\
\hline  continuity of $\boldsymbol{\beta}(\boldsymbol{z})$   &$\surd$       & $\times$      & $\times$  \\
\hline  continuity  of $\Psi(\boldsymbol{z})$        &$\surd$       & $\times$      &$\surd$\tablefootnote{We only prove this in the linear case. \label{note1}}  \\
\hline Lipschitz continuity  of $\Psi(\boldsymbol{z})$        &$\surd$      & $\times$      & $\surd$\cref{note1}  \\
\hline  Gateaux-differentiability of $\Psi(\boldsymbol{z})$   &$\surd$      & $\times$      &$\surd$\tablefootnote{We prove this for almost all directions $\boldsymbol{d}$.}  \\
\hline differentiability of $\Psi(\boldsymbol{z})$            &$\surd$      &$\times$       &$\times$  \\
\hline
\end{tabular}\label{tab:continuity-differentiability-Lip-AAm}
\end{table}
\subsection{Continuity of $\boldsymbol{\beta}(\boldsymbol{z})$ for AA($m$)}

\begin{proposition}\label{pro:beta-full-rank-continuous}
   $\boldsymbol{\beta}(\boldsymbol{z})$ in \cref{eq:beta-AAm} is continuous at $\boldsymbol{z}$ when $R(\boldsymbol{z})$ has full rank.
\end{proposition}
\begin{proof}
 Since $r(z)$ is a continuous function,  $r(z_{m+1})-r(z_j)$ for $j=1,2,\ldots,m$ and $r(z_{m+1})$ are continuous functions, which means that $R(\boldsymbol{z})$ is continuous.  If $R(\boldsymbol{z})$ has full rank,  then $R^T(\boldsymbol{z})R(\boldsymbol{z})$ is invertible and the inverse is continuous and  $\boldsymbol{\beta}(\boldsymbol{z})= -(R(\boldsymbol{z})^TR(\boldsymbol{z}))^{-1} R(\boldsymbol{z})^Tr(z_{m+1})$.  It follows that $\boldsymbol{\beta}(\boldsymbol{z})$ in \cref{eq:beta-AAm} is continuous if $R(\boldsymbol{z})$ has full rank.
\end{proof}
\begin{proposition}\label{pro:beta-zj=z}
 $\boldsymbol{\beta}(\boldsymbol{z})$ in  \cref{eq:beta-AAm} is not continuous at $\boldsymbol{z}=\begin{bmatrix} z^T & z^T & \ldots & z^T\end{bmatrix}^T$ with $r(z)\neq 0$.
\end{proposition}
\begin{proof}
Let $\boldsymbol{z}_0=\begin{bmatrix} z^T & z^T & \ldots & z^T\end{bmatrix}^T$  with  $r(z)\neq 0$ and $\boldsymbol{d} =\begin{bmatrix} d_{m+1}^T  & d_m^T & \ldots  & d_1^T\end{bmatrix}^T$, where all $d_j$ are zero except $d_m=d=\epsilon\, e$ with $e$ a unit vector
in $\mathbb{R}^n$ and $r(z)-r(z+d)\neq 0$.   Then,
\begin{equation*}
  R(\boldsymbol{z}_0+\boldsymbol{d}) = \begin{bmatrix} r(z)-r(z+d) & 0 & \ldots & 0\end{bmatrix}.
\end{equation*}
Using  \cref{eq:beta-continuous} and \cref{eq:two-ways-pseudo-inverse}, we have
\begin{align*}
   \boldsymbol{\beta}(\boldsymbol{z}_0+\boldsymbol{d}) &=  -R(\boldsymbol{z}_0+\boldsymbol{d})^{\dag}r(z_{m+1}) ,\\
   & = -\big(R(\boldsymbol{z}_0+\boldsymbol{d})^T R(\boldsymbol{z}_0+\boldsymbol{d})\big)^{\dag} R(\boldsymbol{z}_0+\boldsymbol{d})^T r(z),\\
   & =  \begin{bmatrix} \displaystyle \frac{-(r(z)-r(z+d))^T r(z)}{(r(z)-r(z+d))^T(r(z)-r(z+d))} &  0 & \ldots & 0 \end{bmatrix}^T.
\end{align*}
According to the proof of \cref{prop:betaxx}, $\displaystyle \frac{-(r(z)-r(z+d))^T r(z)}{(r(z)-r(z+d))^T(r(z)-r(z+d))}\rightarrow \pm \infty$ as $\epsilon \rightarrow 0$.
Thus, $\boldsymbol{\beta}(\boldsymbol{z})$  is not continuous at $\boldsymbol{z}_0$ with $r(z)\neq 0$.
\end{proof}
%
%
\begin{proposition}\label{pro:beta-z=z^*}
 $\boldsymbol{\beta}(\boldsymbol{z})$ in  \cref{eq:beta-AAm} is not continuous at $\boldsymbol{z}=\boldsymbol{z}^*$.
\end{proposition}
\begin{proof}
Consider $\boldsymbol{d} =\begin{bmatrix} (\epsilon d_1)^T  & (\epsilon d_2)^T & (\epsilon d_1)^T &\ldots  & ( \epsilon d_1)^T\end{bmatrix}^T$ with $\epsilon\neq 0$ and $ r(x^*+\epsilon d_1)-r(x^*+\epsilon d_2)\neq 0$.   Then,
\begin{equation*}
  R(\boldsymbol{z}^*+\boldsymbol{d}) = \begin{bmatrix} r(x^*+\epsilon d_1)-r(x^*+\epsilon d_2) & 0 & \ldots & 0\end{bmatrix}.
\end{equation*}
For simplicity, let $w=r(x^*+\epsilon d_1)-r(x^*+\epsilon d_2)$. Using  \cref{eq:beta-continuous} and \cref{eq:two-ways-pseudo-inverse}, we have
\begin{align*}
   \boldsymbol{\beta}(\boldsymbol{z}^*+\boldsymbol{d}) &=  -R(\boldsymbol{z}^*+\boldsymbol{d})^{\dag}r(z_{m+1}) ,\\
   & = -\big(R(\boldsymbol{z}^*+\boldsymbol{d})^T R(\boldsymbol{z}^*+\boldsymbol{d})\big)^{\dag} R(\boldsymbol{z}^*+\boldsymbol{d})^T r(x^*+\epsilon d_1),\\
   & =  \begin{bmatrix} \displaystyle \frac{-w^T r(x^*+\epsilon d_1)}{w^Tw} &  0 & \ldots & 0 \end{bmatrix}^T.
\end{align*}
Let $\displaystyle \Delta(\epsilon)=\frac{-w^T r(x^*+\epsilon d_1)}{w^Tw}$.  According to the proof of \cref{prop:betax*},  $\lim_{\epsilon \rightarrow 0} \Delta(\epsilon)=0$ when $d_1=0$ and when $d_1=d_2$, and
$\lim_{\epsilon \rightarrow 0} \Delta(\epsilon)=-1$ when $d_2=0$.
Since the limit depends on the choice of $d_1$ and $d_2$, $\boldsymbol{\beta}(\boldsymbol{z})$ is not continuous at $\boldsymbol{z} =\boldsymbol{z}^*$.
\end{proof}
\subsection{Continuity and differentiability of $\Psi(\boldsymbol{z})$ for  AA($m$)}
In this subsection, we discuss the continuity and differentiability of $\Psi(\boldsymbol{z})$ for  AA($m$). We consider three cases at point $\boldsymbol{z}=\begin{bmatrix} z_{m+1}^T & z_m^T & \ldots & z_{1}^T\end{bmatrix}^T$:
\begin{itemize}
\item[(a)] $R(\boldsymbol{z})$ has full rank;
\item[(b)] $z_{j}=z, j=1,2,\ldots,m+1$, with $r(z) \neq 0$;
\item[(c)]  $z_{m+1}=x^*$ and $R(\boldsymbol{z})$ is rank-deficient.
\end{itemize}
\begin{proposition}\label{prop:Psi-AAm-full-rank-diff}
$\Psi(\boldsymbol{z})$ is continuous at $\boldsymbol{z}$ when $R(\boldsymbol{z})$ has full rank. Furthermore, $\Psi(\boldsymbol{z})$ is differentiable at $\boldsymbol{z}$ when $R(\boldsymbol{z})$ has full rank.
\end{proposition}

\begin{proof}
Recall  $\Psi(\boldsymbol{z})$ in \cref{eq:Psi-AAm}. When $R(\boldsymbol{z})$ in \cref{eq:beta-continuous} has full rank,
$\boldsymbol{\beta}(\boldsymbol{z})$ in \cref{eq:beta-AAm} is continuous by \cref{pro:beta-full-rank-continuous}, and $Q(\boldsymbol{z})$ in \cref{eq:def-Q-general} is continuous. It follows that $\Psi(\boldsymbol{z})$ is continuous.  Furthermore, since $q(z_1), r(z_1), Q(\boldsymbol{z}), R(\boldsymbol{z})$,  and $(R((\boldsymbol{z}))^TR(\boldsymbol{z}))^{-1}$ are differentiable,  $\Psi(\boldsymbol{z})$ is differentiable.
\end{proof}
\begin{proposition}\label{AAm-prop-noncon-rank-case1}
$\Psi(\boldsymbol{z})$ is not continuous at $\boldsymbol{z}$ when  $\boldsymbol{z}=\begin{bmatrix} z^T & z^T & \ldots & z^T \end{bmatrix}^T$ with $r(z)\neq 0$.
\end{proposition}
\begin{proof}
Let $\boldsymbol{z}_0 =\begin{bmatrix} z^T & z^T & \ldots & z^T \end{bmatrix}^T$ with $r(z) \neq 0$ and $\boldsymbol{d} =\begin{bmatrix} 0^T & d^T & 0&  \ldots & 0^T \end{bmatrix}^T$, where $d=\epsilon\, e$ with $e$ a unit vector
in $\mathbb{R}^n$ and $r(z)-r(z+d)\neq 0$.  Then,
\begin{equation*}
  Q(\boldsymbol{z}_0+\boldsymbol{d}) = \begin{bmatrix} q(z)-q(z+d) & 0 & \ldots & 0\end{bmatrix}.
\end{equation*}
Let $w=r(z)-r(z+d)$.  From the proof of \cref{pro:beta-zj=z}, we have
\begin{equation*}
   Q(\boldsymbol{z}_0+\boldsymbol{d})\boldsymbol{\beta}(\boldsymbol{z}_0+\boldsymbol{d}) = \frac{-w^T r(z)}{w^Tw}\big(q(z)-q(z+d)\big).
\end{equation*}
 From \cref{eq:Psi-AAm} we have
\begin{equation*}
 \Psi(\boldsymbol{z}_0+\boldsymbol{d}) =
 \begin{bmatrix}
  q(z) -\displaystyle \frac{w^T r(z)}{w^Tw}\big(q(z)-q(z+d)\big)\\
   z+d_{m+1}\\
   \vdots \\
   z+d_2
  \end{bmatrix}.
\end{equation*}
From the proof of \cref{prop:Psixx}, we know that the limit of  $q(z) -\displaystyle \frac{w^T r(z)}{w^Tw}\big(q(z)-q(z+d)\big)$ as $\epsilon\rightarrow 0$ depends on the choice of $e$. It follows that $\Psi(\boldsymbol{z}_0+\boldsymbol{d})$ is not continuous at $\boldsymbol{z}_0$.
\end{proof}
We analyze the continuity of case (c) for the linear case only, because it is not clear how to generalize \cref{prop:Psix*} for $m>1$ in the nonlinear case. Differentiability for case (c) is discussed in the next subsection.
\begin{proposition}\label{AAm-prop-noncon-rank-zm+1=x*}
In the linear case, $\Psi(\boldsymbol{z})$ is   Lipschitz continuous at $\boldsymbol{z}$ when  $R(\boldsymbol{z})$ is rank-deficient and $z_{m+1}=x^*$.
\end{proposition}
\begin{proof}
See \cref{app:proof-lipsch-m}.
\end{proof}
Note that \cref{AAm-prop-noncon-rank-zm+1=x*} contains the special case that $\boldsymbol{z}=\boldsymbol{z}^*$.

\subsection{Differentiability of $\Psi(\boldsymbol{z})$ at $\boldsymbol{z}^*$  for AA($m$)}

In this subsection we investigate the differentiability of 	$\Psi(\boldsymbol{z})$ at $\boldsymbol{z}^*$. We first consider the directional derivatives of $\Psi(\boldsymbol{z})$ at $\boldsymbol{z}^*$.
\begin{theorem}\label{thm:proof-directional-m}
Let  $M=q'(x^*)$ and $A= r'(x^*)=I-q'(x^*)$. Then, the  directional derivative of $\Psi(\boldsymbol{z})$ at $\boldsymbol{z}^*$ in any direction $\boldsymbol{d}$ such that $D(\boldsymbol{d})$ is full rank is given by
\begin{equation}\label{eq:-direct-nonlinear-AAm-matrix}
  \mathfrak{ D} \Psi(\boldsymbol{z}^*,\boldsymbol{ d}) = \begin{bmatrix}  (1+\sum_{j=1}^{m} \widehat{\beta}_j)M &- \widehat{\beta}_{1}M  &\cdots   & -\widehat{\beta}_{m-1}M &-\widehat{\beta}_{m}M\\
  I&   0&   &    0&  0\\
  0&   I&   &    0&  0\\
  \vdots & \vdots &  &\vdots &\vdots\\
  0&  0&  \cdots&   I& 0
  \end{bmatrix}\boldsymbol{d},
\end{equation}
where $\widehat{\boldsymbol{\beta}}(\boldsymbol{d})=\begin{bmatrix} \widehat{\beta}_1 & \widehat{\beta}_2 & \cdots & \widehat{\beta}_m \end{bmatrix}^T= -( AD(\boldsymbol{d}))^{\dag}A d_{m+1}$
with $D(\boldsymbol{d})$ defined in \cref{eq:def-D-operator}.
\end{theorem}
\begin{proof}
See \cref{app:proof-directional-m}.
\end{proof}
\begin{remark}
The result in \cref{eq:-direct-nonlinear-AAm-matrix} holds for the linear case without the requirement that $D(\boldsymbol{d})$ is full rank because the term $P(hd_{m+1})$ in \cref{eq:expansion-r} vanishes.
For the nonlinear case, when $D(\boldsymbol{d})$ is rank-deficient, we do not know whether \cref{eq:limit-speudo-inverse} holds.
\end{remark}

\begin{remark}
For $m>1$, $\Psi(\boldsymbol{z})$ is  not differentiable at $\boldsymbol{z}^*$, because
the matrix in \cref{eq:-direct-nonlinear-AAm-matrix} depends on $\boldsymbol{d}$ and
$\mathfrak{D} \Psi(\boldsymbol{z}^*,\boldsymbol{d})$ cannot be written as $\Psi'(\boldsymbol{z}^*)\boldsymbol{d}$.
\end{remark}
\section{Numerical results}
\label{sec:numerics}
In this section, we give further numerical results expanding on the AA(1) convergence patterns we identified in \cref{fig:simple-MC} for the simple $2 \times 2$ linear equation of \cref{prob:linear2x2}. We extend the numerical tests to larger linear problems and a nonlinear problem, for $m=1$ and $m>1$. We are also interested in comparing the convergence behavior of AA($m$) with GMRES for the linear problems: we compare the standard windowed version of AA($m$) with GMRES and a restarted version of AA($m$), which is essentially restarted GMRES($m$). We relate the numerical results to the theoretical findings of \Cref{sec:AA1,sec:AAm}.

\begin{figure}[h]
\centering
\includegraphics[width=.6\textwidth]{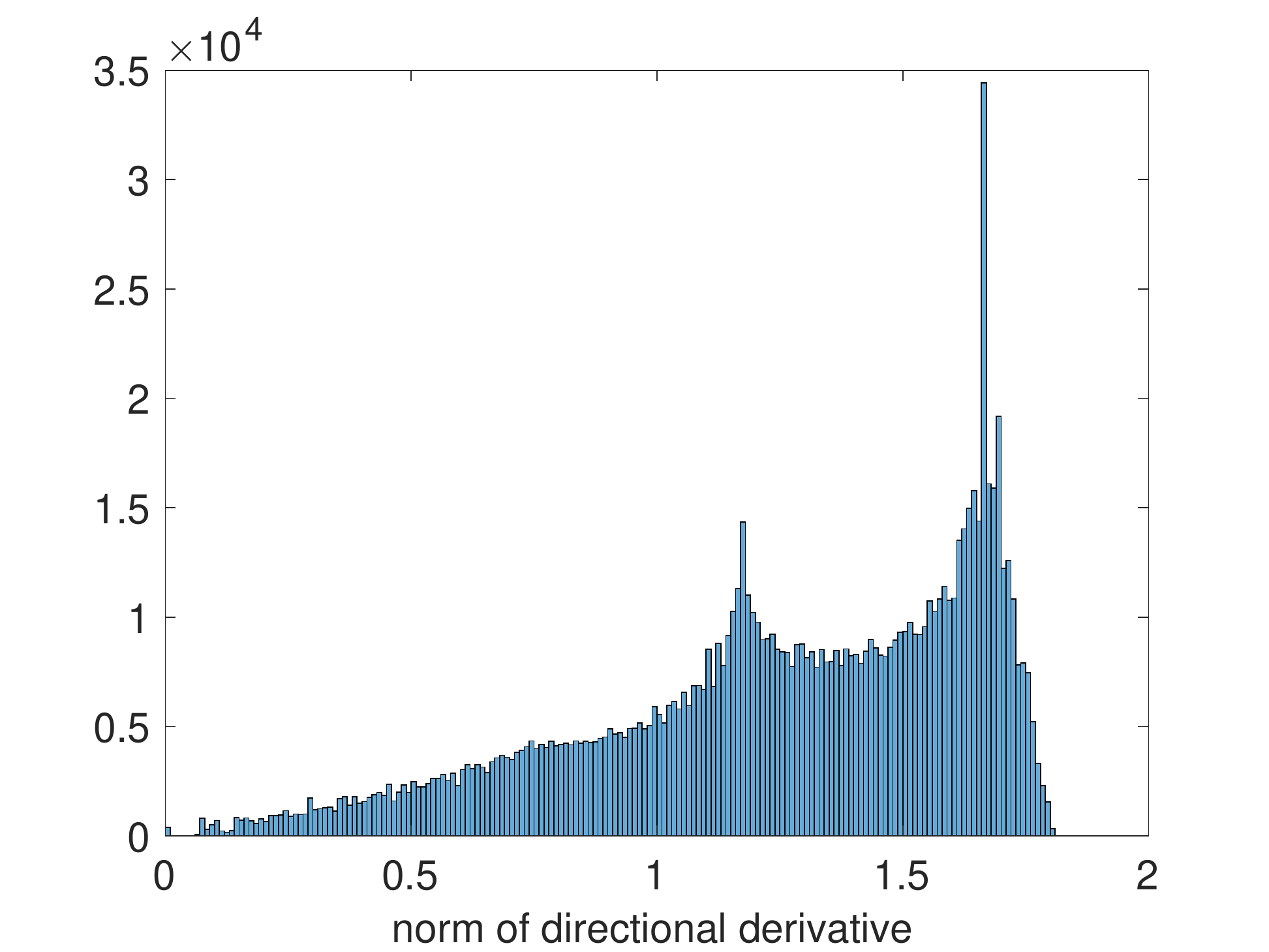}
\caption{
\cref{prob:linear2x2} (linear): Histogram of the norm of the directional derivative $\mathfrak{ D} \Psi(\boldsymbol{z},\boldsymbol{ d})$ of $\Psi(\boldsymbol{z})$ at $\boldsymbol{z}^*$ in unit vector direction $\boldsymbol{ d}$ (see Eq.\ \cref{AA(1)-direction-diff}), for $10^6$ unit vectors on a uniform polar grid in 4D space.} \label{prob2-dir_deriv}
\end{figure}

We first revisit the numerical results of \cref{fig:simple-MC} for the $2 \times 2$ linear equation of \cref{prob:linear2x2}.
As discussed before, \cref{fig:simple-MC} indicates that AA(1) sequences $\{x_k\}$ for \cref{prob:linear2x2}
converge r-linearly with a continuous spectrum of convergence factors $\rho_{\{x_k\}}$, and it appears that
a least upper bound $\rho_{\Psi,x^*}$ for $\rho_{\{x_k\}}$ exists for the AA(1) iteration \cref{eq:AA-fixed-point}
that is smaller than, say, 0.45, and substantially smaller than the r-linear convergence factor $\rho_{q,x^*}=2/3$ of
fixed-point iteration \cref{eq:fixed-point} by itself. Since there currently is no theory to establish the existence or
value of $\rho_{\Psi,x^*}$, it is interesting to investigate, in light of \cref{thm:x=y-directional-dif}, whether the existence of
all directional derivatives of $\Psi(\boldsymbol{z})$ at $\boldsymbol{z}^*$ may tell us something about the existence or value
of $\rho_{\Psi,x^*}$. As is well-known, in the case of an iteration function $\Psi(z)$ that is $L$-Lipschitz in a neighborhood of
$z^*$ with $L<1$, the FP iteration $z_{k+1}=\Psi(z_k)$ converges q-linearly with q-linear convergence factor not worse than $L$ \cite{kelley1995iterative}. In the case of AA(1), $\Psi(\boldsymbol{z})$ is not $L$-Lipschitz in a neighborhood of $\boldsymbol{z}^*$
(since, by \cref{prop:Psixx}, $\Psi(\boldsymbol{z})$ is not continuous at $\boldsymbol{z} =\begin{bmatrix} x\\ y\end{bmatrix}$ when
$x= y$ with $r(x) \neq 0$), but it is still interesting to investigate the size of the directional derivatives that we know by \cref{thm:x=y-directional-dif} exist in all directions at $\boldsymbol{z}^*$.

\cref{prob2-dir_deriv} shows a histogram for \cref{prob:linear2x2} of the norm of the directional derivative $\mathfrak{ D} \Psi(\boldsymbol{z},\boldsymbol{ d})$ of $\Psi(\boldsymbol{z})$ at $\boldsymbol{z}^*$ in unit vector direction $\boldsymbol{ d}$ (see Eq.\ \cref{AA(1)-direction-diff}), for $10^6$ unit vectors on a uniform polar grid in 4D space. The histogram indicates that the unit directional derivatives are bounded above by a value of about 1.6, but it is interesting that this value is greater than 1, which indicates that directional derivates or Lipschitz constants are not a useful avenue to prove the existence of a least upper bound $\rho_{\Psi,x^*}<\rho_{q,x^*}$ for AA(1).

\begin{figure}[h]
\centering
\includegraphics[width=.49\textwidth]{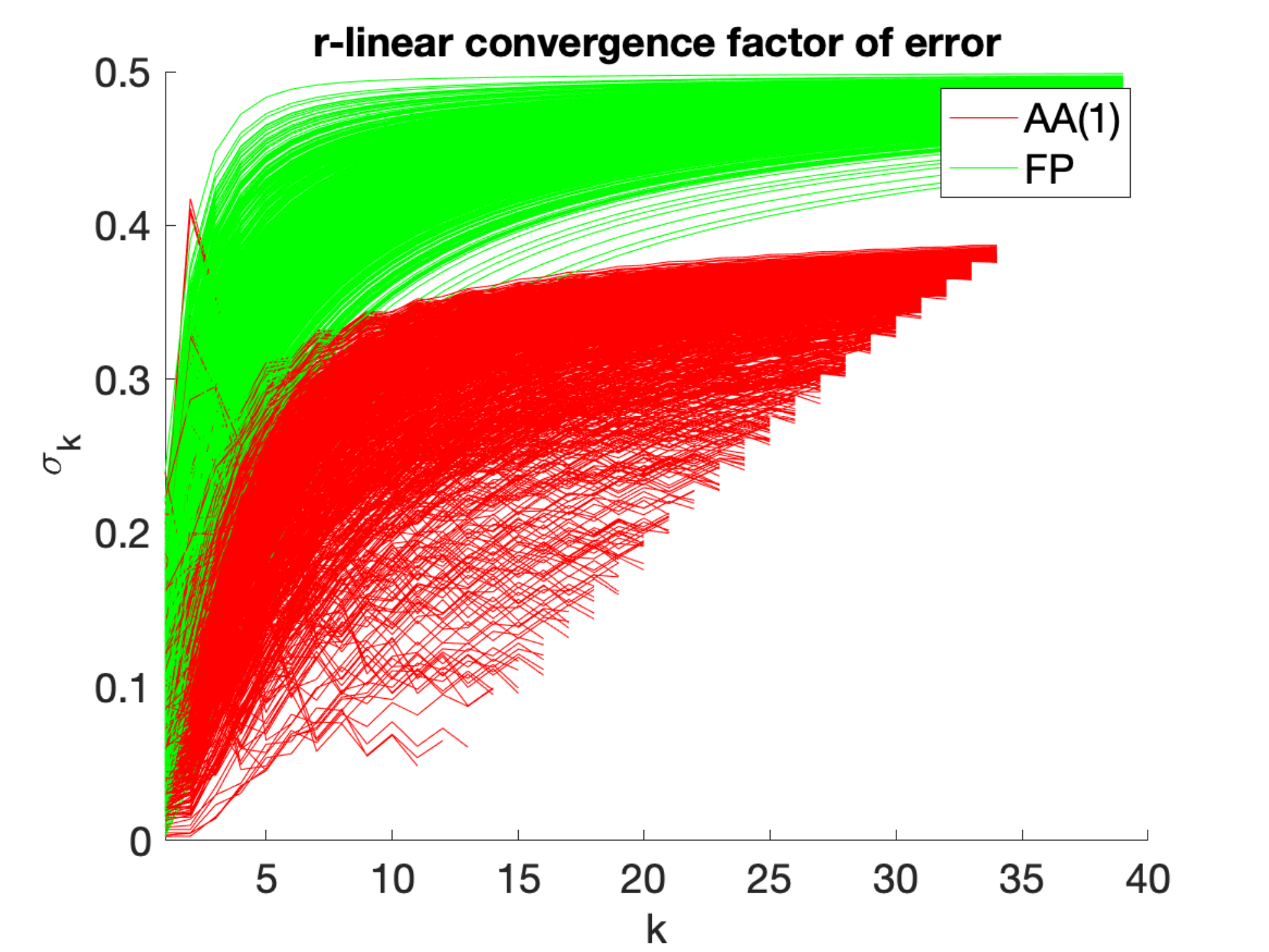}
\includegraphics[width=.49\textwidth]{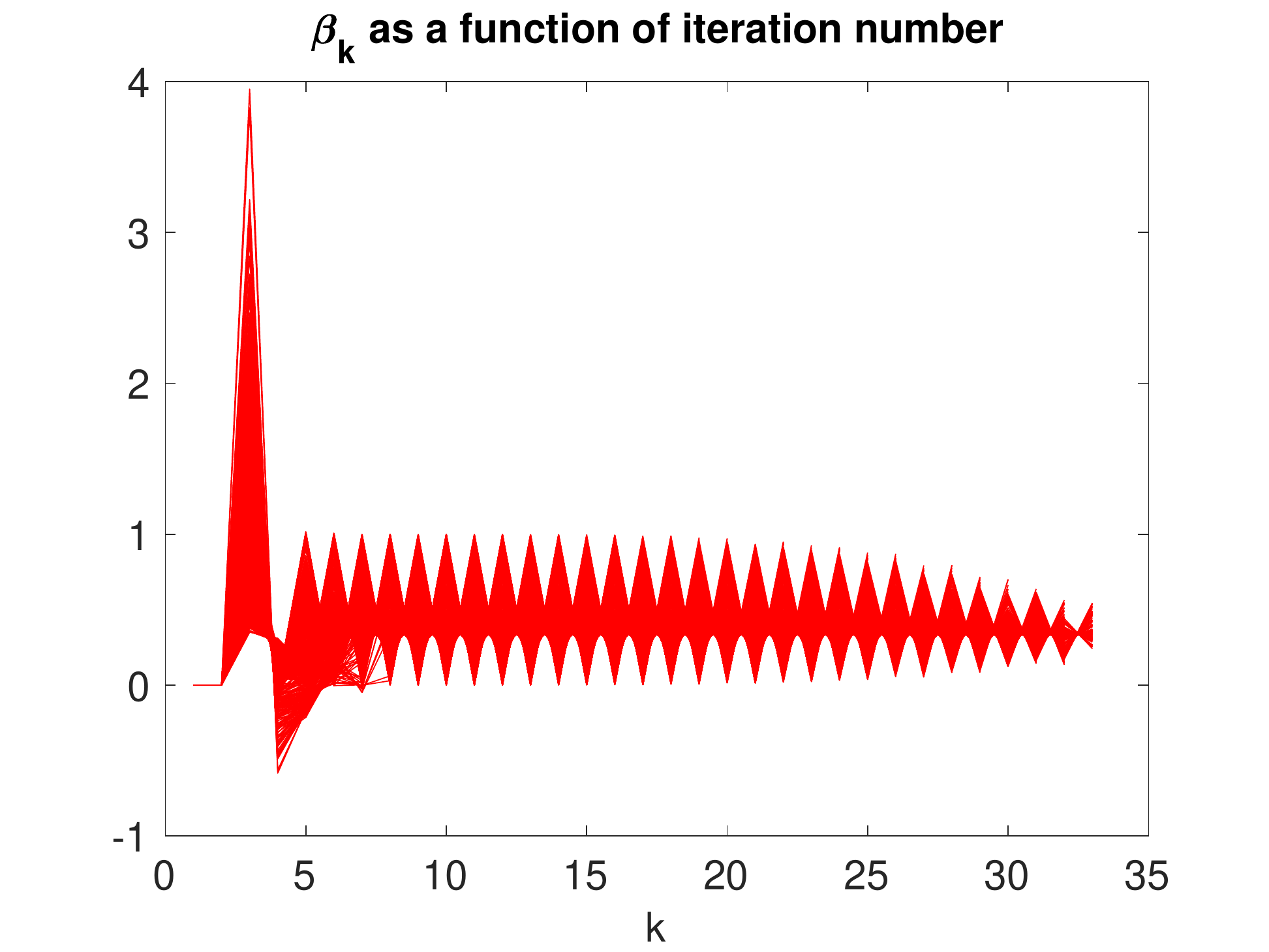}
\includegraphics[width=.325\textwidth]{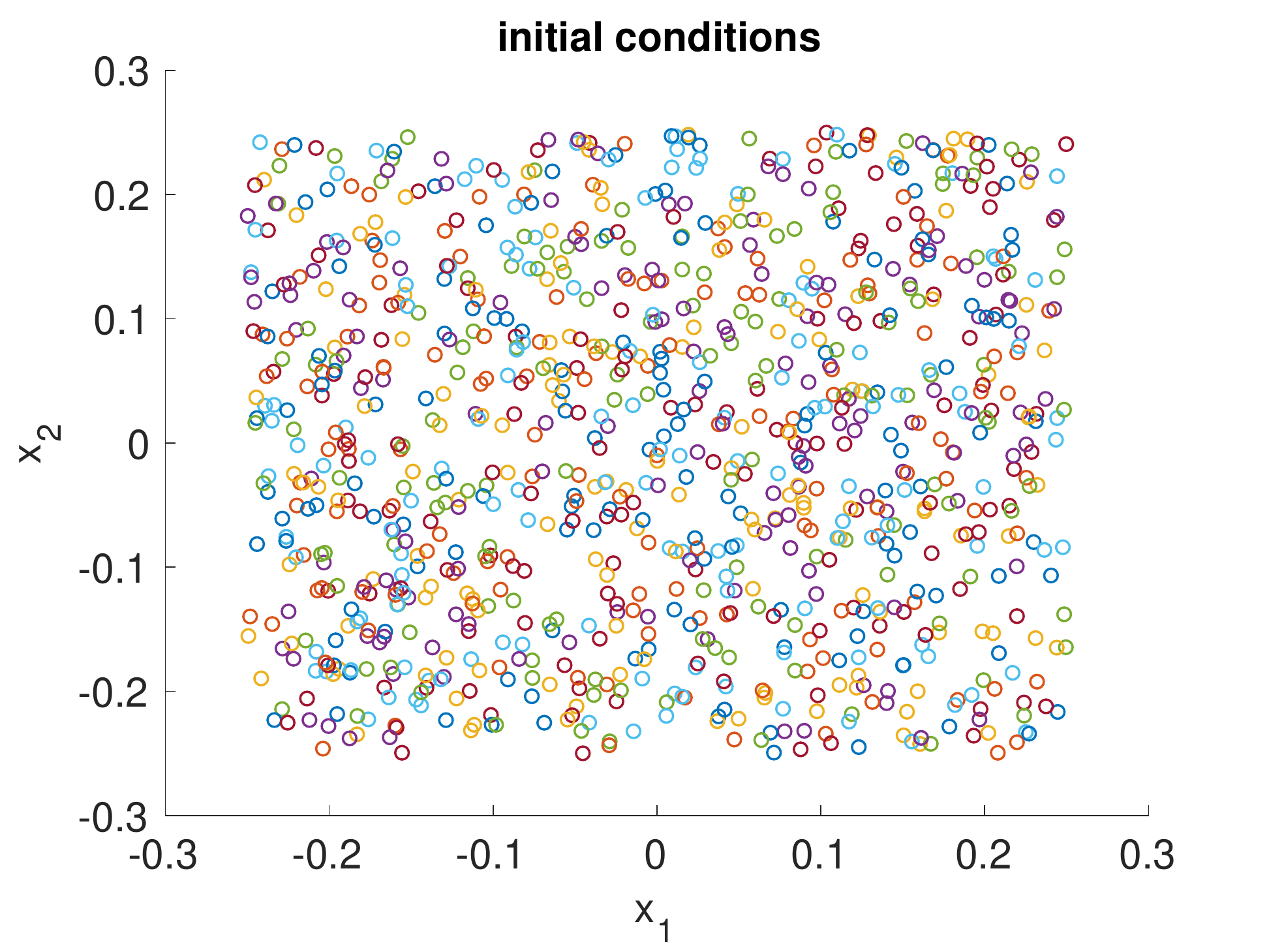}
\includegraphics[width=.325\textwidth]{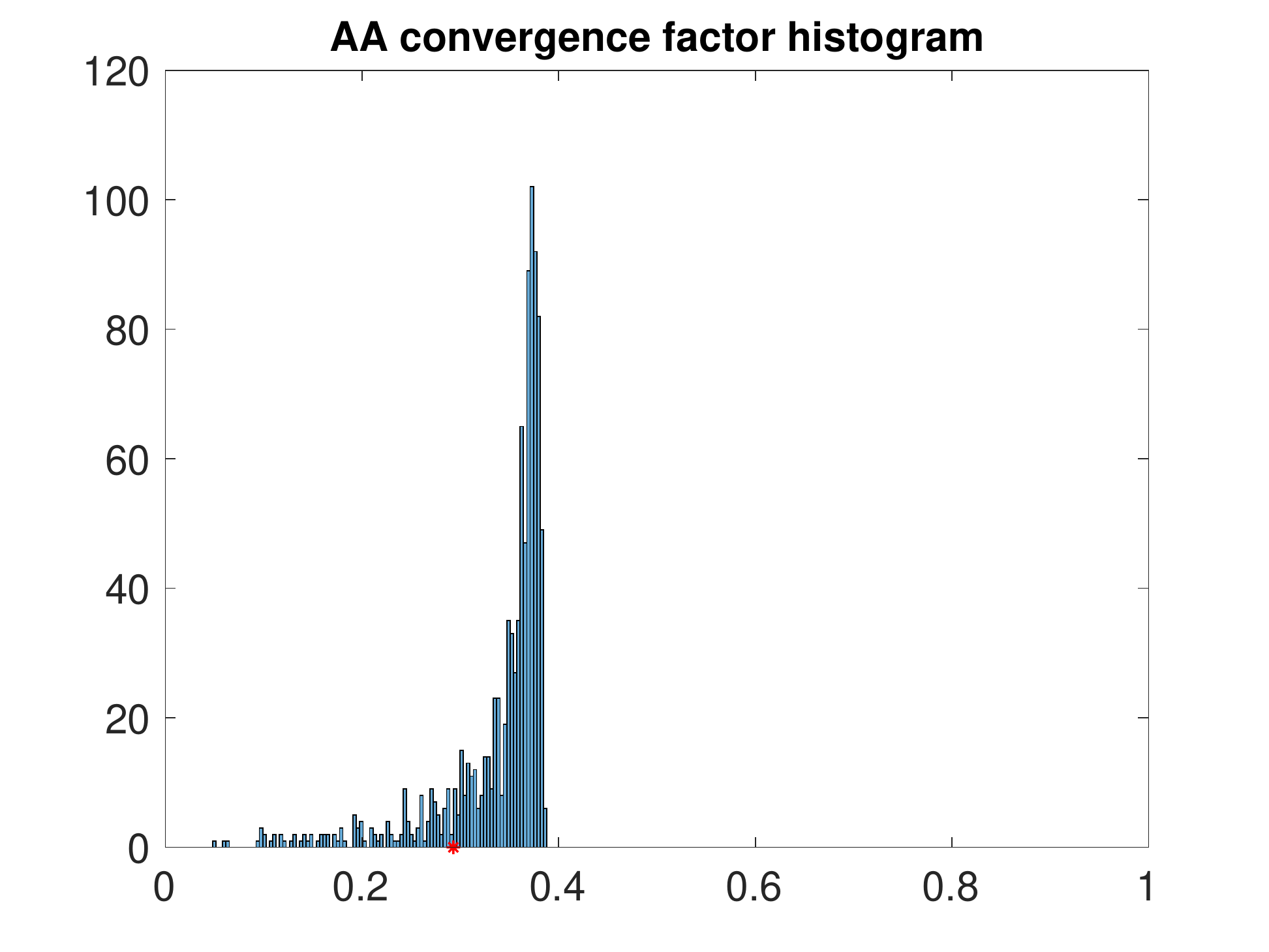}
\includegraphics[width=.325\textwidth]{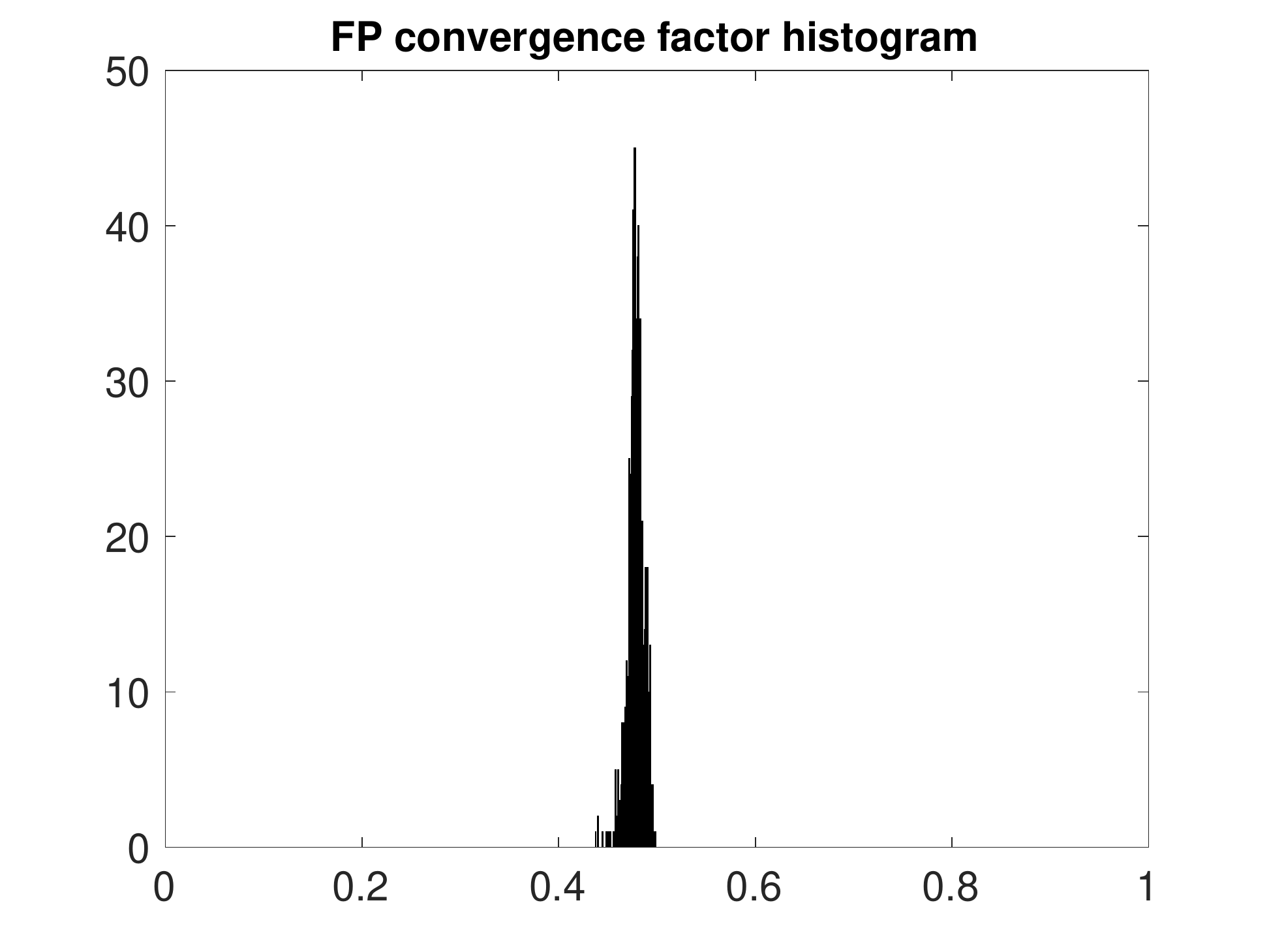}
\caption{FP and AA(1) results for nonlinear \cref{prob:nonlinear2x2} with 1,000 random initial guesses.} \label{prob3-mc}
\end{figure}

We next consider a nonlinear example:
\begin{problem}\label{prob:nonlinear2x2}
Consider the nonlinear system
\begin{align}
x_2= x_1^2 \label{eq:exm1-1}\\
x_1+(x_1-1)^2 +x_2^2 =1 \label{eq:exm1-2}
\end{align}
with solution $(x_1^*,x_2^*) = (0,0)$.
Let $x=[x_1 \ x_2]^T$ and define the FP iteration function
\begin{equation*}
  q(x) = \begin{bmatrix} \ds \frac{1}{2}(x_1+x_1^2+x_2^2) \\ \\ \ds \frac{1}{2}(x_2+x_1^2) \end{bmatrix},
\end{equation*}
with Jacobian matrix
\begin{equation*}
q'(x)=
\begin{bmatrix}
x_1+\ds \frac{1}{2}& x_2\\
x_1 & \ds \frac{1}{2}
\end{bmatrix}.
\end{equation*}
We have
\begin{equation*}
  q'(x^*) = \begin{bmatrix}
  \ds \frac{1}{2}& 0 \\
0 & \ds \frac{1}{2}
  \end{bmatrix}, \ \textrm{and} \quad
\rho(q'(x^*)) =\ds \frac{1}{2}<1.
\end{equation*}
\end{problem}

\begin{figure}[h]
\centering
\includegraphics[width=.49\textwidth]{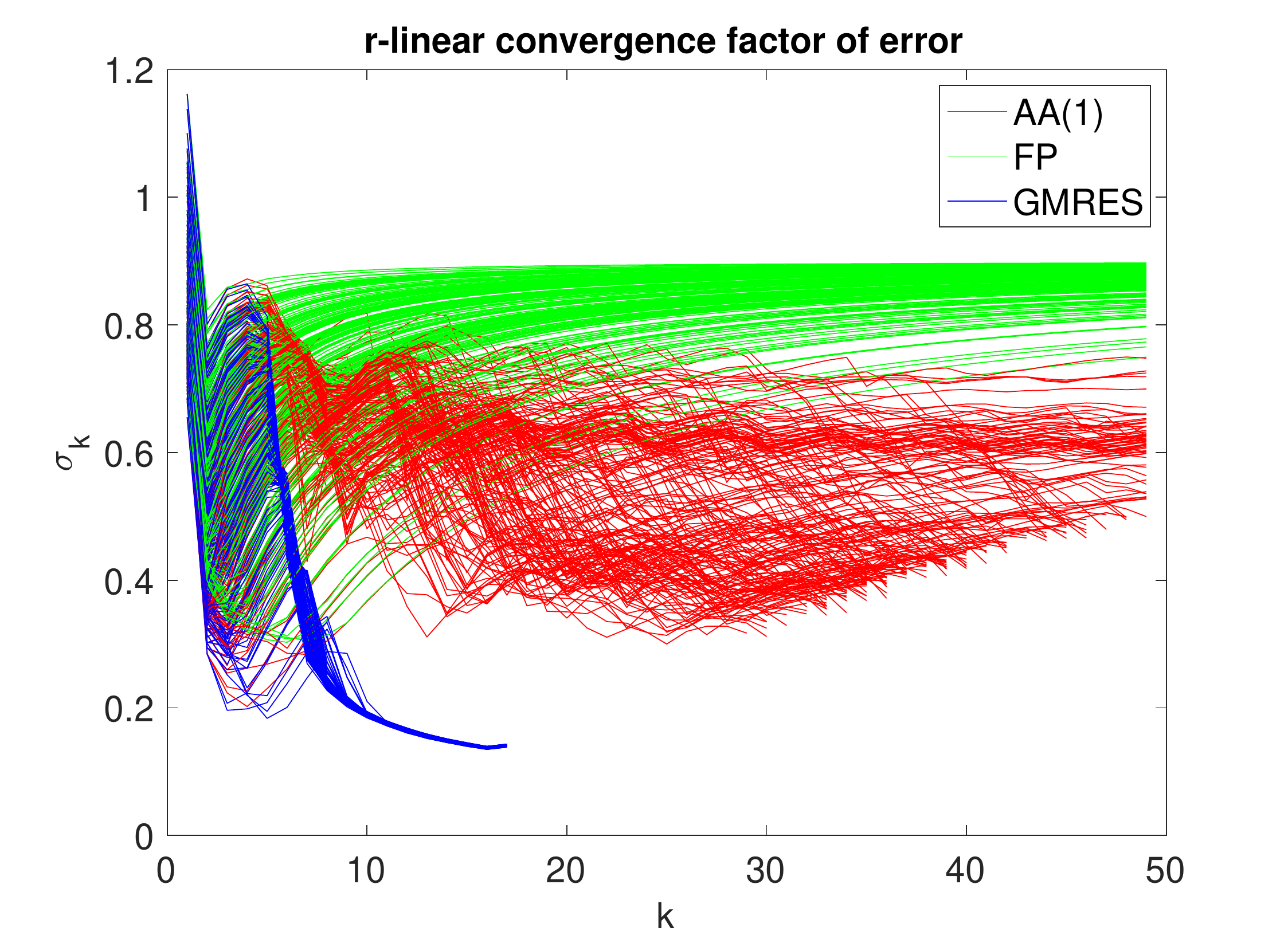}
\includegraphics[width=.49\textwidth]{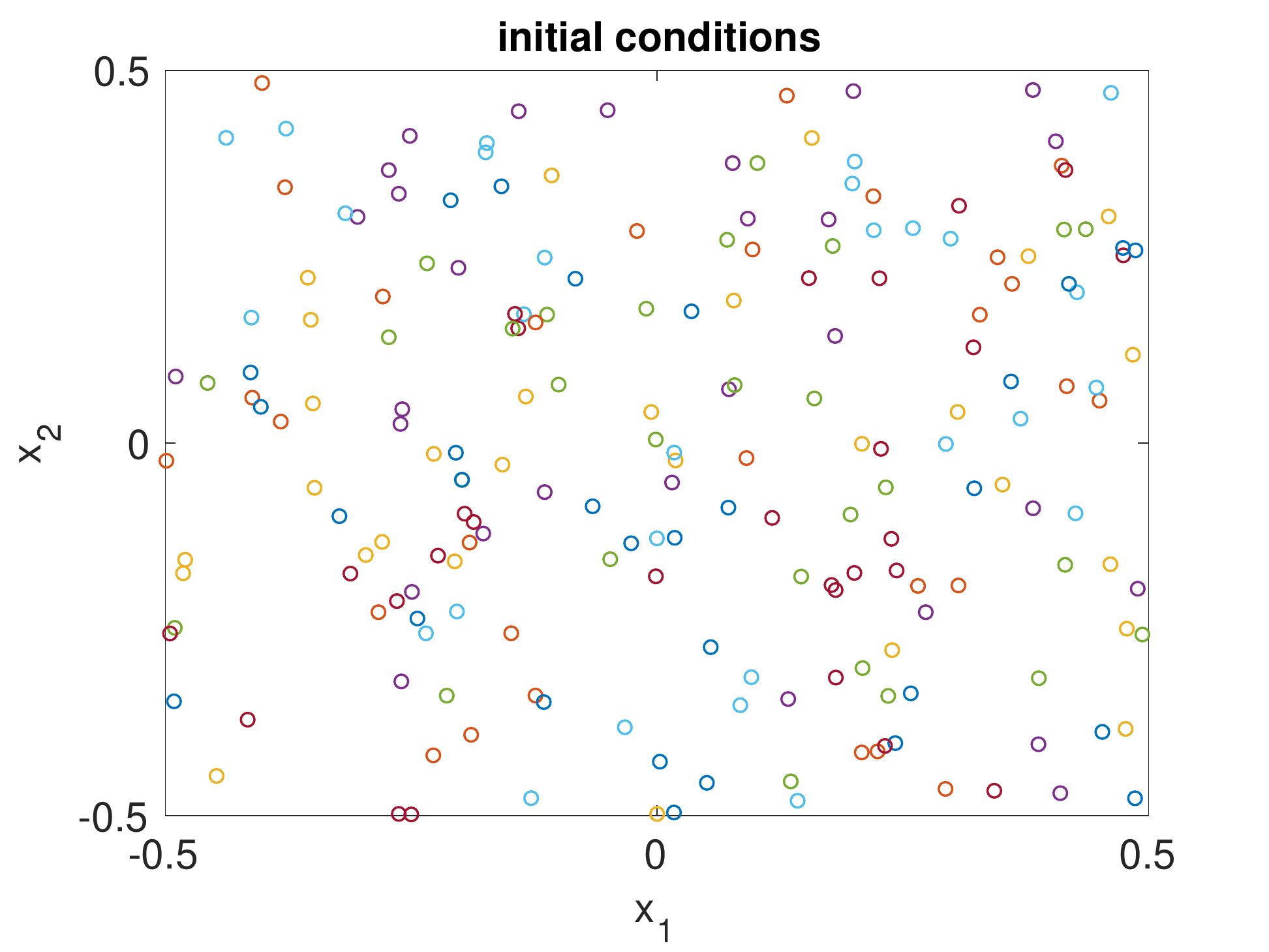}
\includegraphics[width=.325\textwidth]{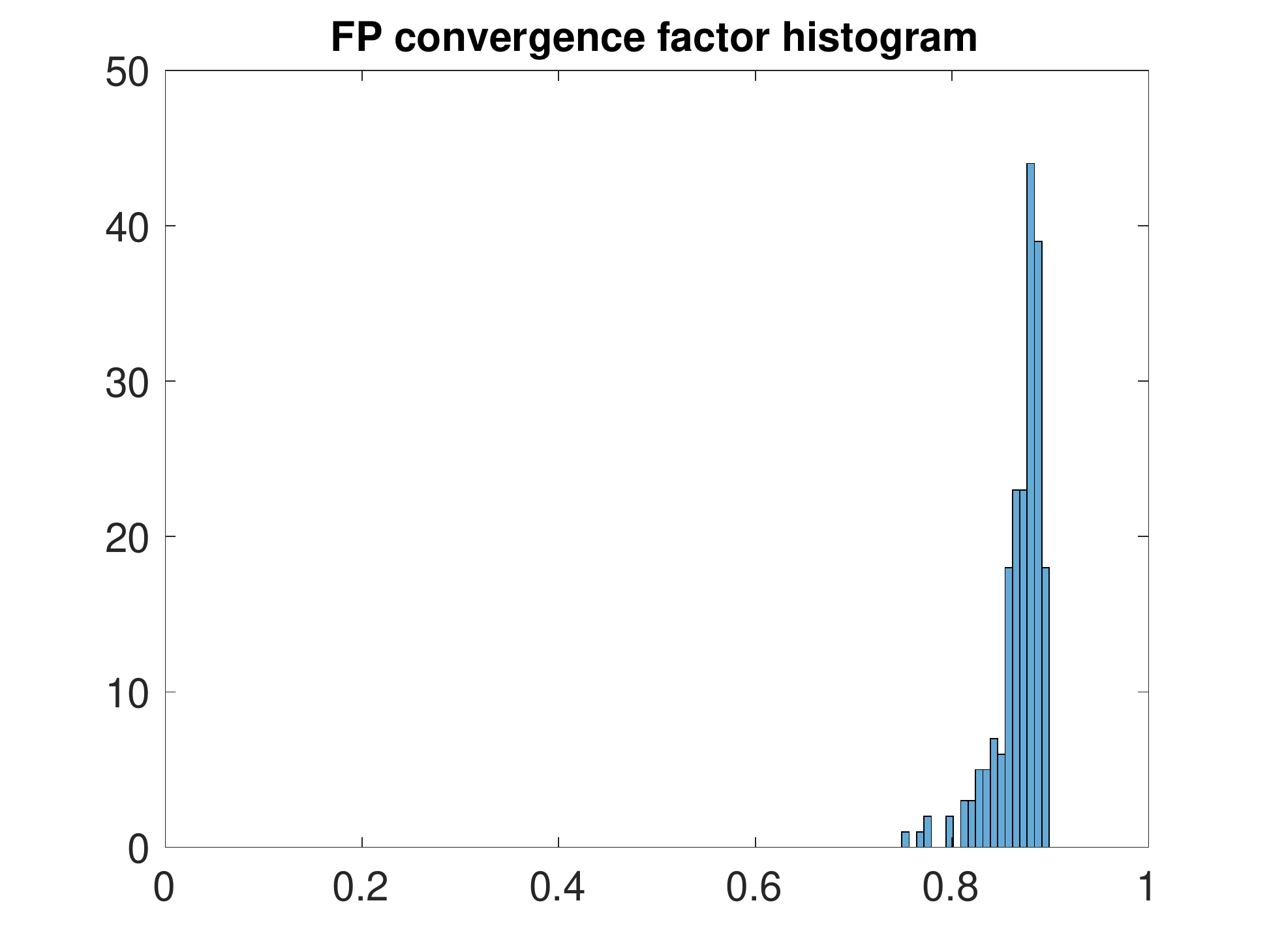}
\includegraphics[width=.325\textwidth]{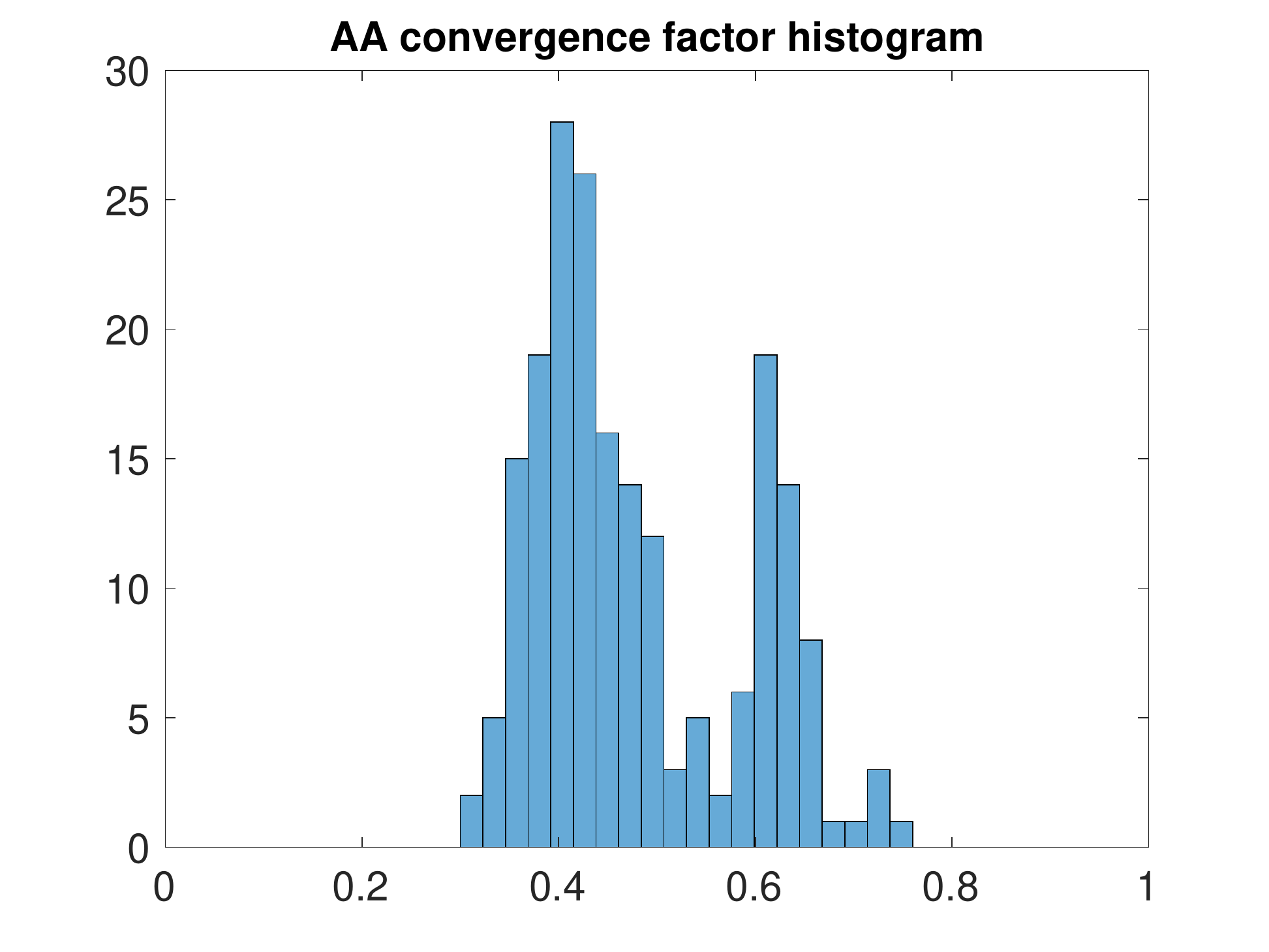}
\includegraphics[width=.325\textwidth]{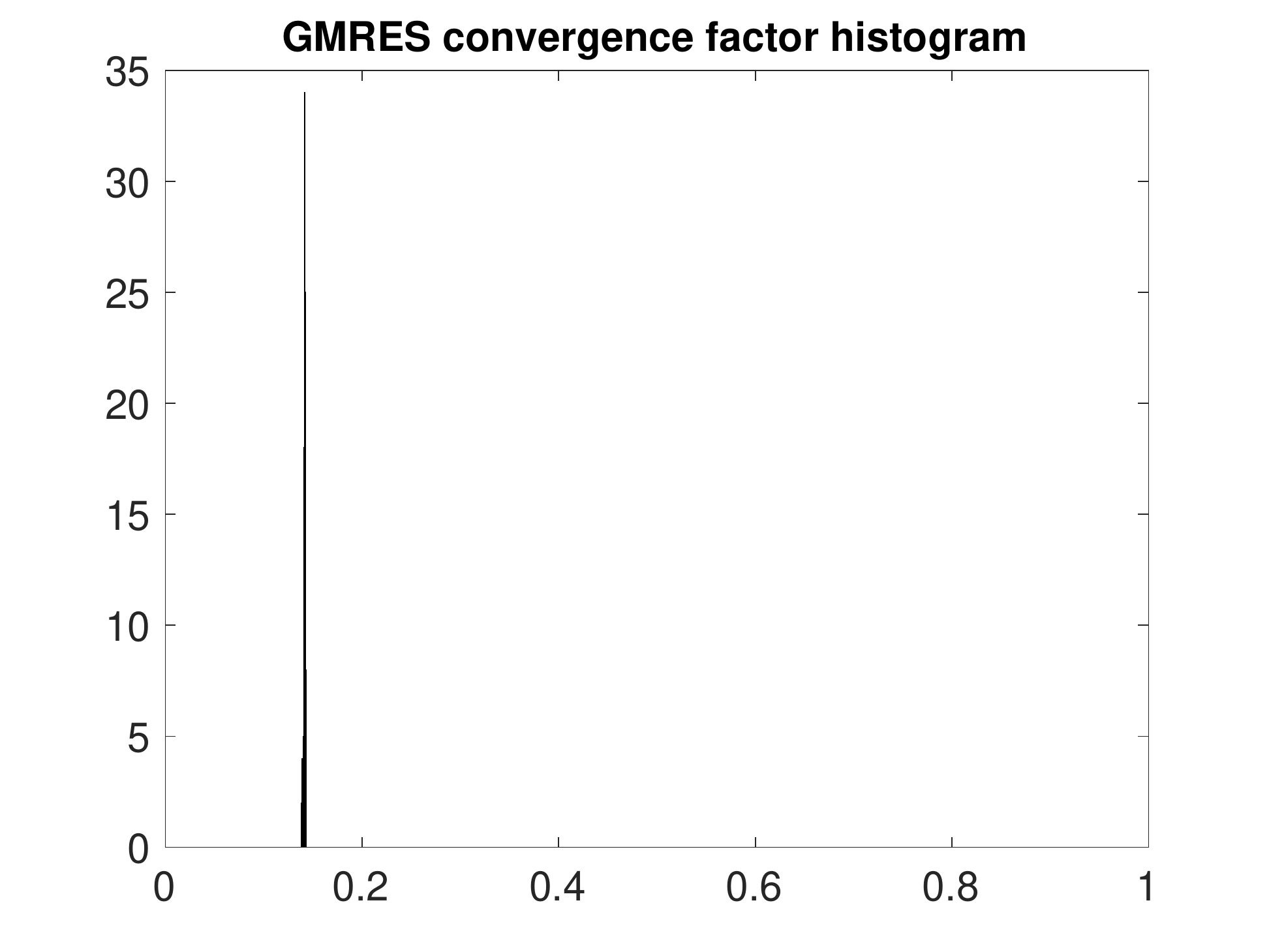}
\caption{
\cref{prob:linear200}: linear problem with $M \in \mathbb{R}^{200 \times 200}$ and $\lambda_1=0.9$, $\lambda_2=-0.3$, $\lambda_3=0.3$, and $\lambda_4=-0.3$, for 200 random initial guesses. Comparison of AA(1) (red) with AA($\infty$) (blue), which is essentially equivalent to GMRES. It can be observed that the asymptotic linear convergence factors of the AA(1) sequences strongly depend on the initial guess, but the asymptotic linear convergence factors of the AA($\infty$) sequences and of the FP sequences do not depend on the initial guess.
} \label{prob18-gmresinf}
\end{figure}

\cref{prob3-mc} shows FP and AA(1) numerical results for the nonlinear \cref{prob:nonlinear2x2}.
The nonlinear results of \cref{prob3-mc} show convergence behavior that is qualitatively similar to the linear results of \cref{fig:simple-MC}:
the AA(1) sequences $\{x_k\}$ converge r-linearly, but the r-linear convergence factors $\rho_{\{x_k\}}$ depend on the initial guess on a set of nonzero measure. It appears that a least upper bound $\rho_{\Psi,x^*}$ for $\rho_{\{x_k\}}$ exists for the AA(1) iteration \cref{eq:AA-fixed-point}
that is smaller than the r-linear convergence factor $\rho_{q,x^*}=1/2$ of fixed-point iteration \cref{eq:fixed-point} by itself.
We also see that the $\beta_k$ sequences oscillate for this nonlinear problem as the AA(1) iteration approaches $x^*$, consistent with the
discontinuity of $\beta(\boldsymbol{z})$ at $\boldsymbol{z}^*$ shown in \cref{prop:betax*}.


We next consider a larger linear problem that we will use for comparison of AA($m$) to GMRES and a restarted version of AA($m$).
\begin{problem}\label{prob:linear200}
Consider the linear iteration
\begin{equation}\label{eq:q-linear200}
  x_{k+1} = M x_k,
\end{equation}
with $M \in \mathbb{R}^{200 \times 200}$, where $M$ is diagonal except that $m_{1,2}=1$.
$M$ has 196 eigenvalues that are spaced uniformly between 0.29325 and 0.03, and 4 eigenvalues $\lambda_1$ to $\lambda_4$
that are specified such that $\lambda_1=0.9$ and $\lambda_1$ to $\lambda_3$ take on values that are specific to the problem instantiation (see results figures). In all cases, $\rho(M)=0.9$.
\end{problem}

\cref{prob18-gmresinf} shows results for \cref{prob:linear200} with eigenvalues $\lambda_1=0.9$, $\lambda_2=-0.3$, $\lambda_3=0.3$, and $\lambda_4=-0.3$. Comparing FP and AA(1) with AA($\infty$), which is essentially equivalent to GMRES. As in previous examples, the asymptotic linear convergence factors of the AA(1) sequences strongly depend on the initial guess, but it is interesting to observe that the asymptotic linear convergence factors of the AA($\infty$) sequences do not depend on the initial guess.

\begin{figure}[h]
\centering
\includegraphics[width=.49\textwidth]{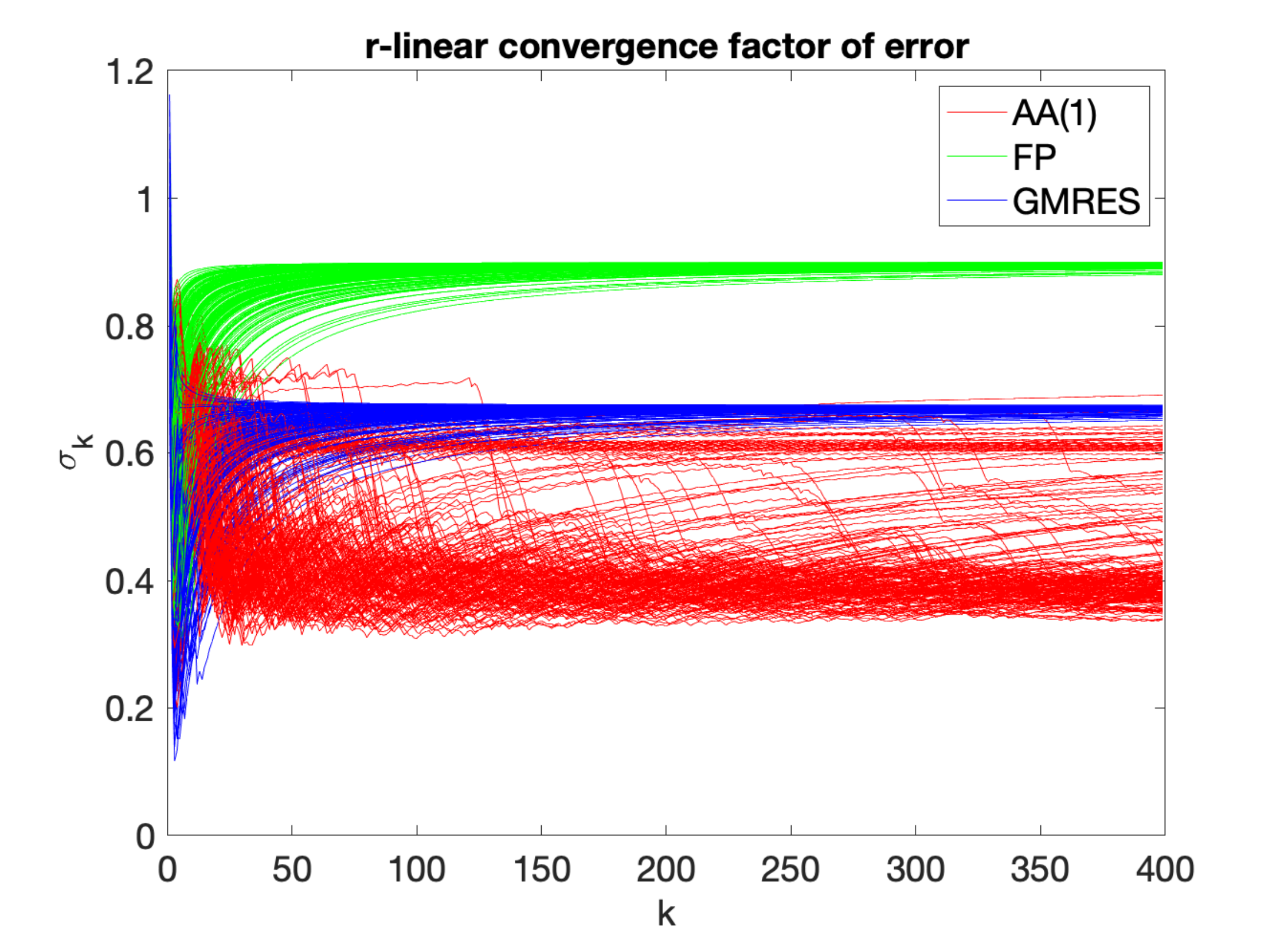}
\includegraphics[width=.49\textwidth]{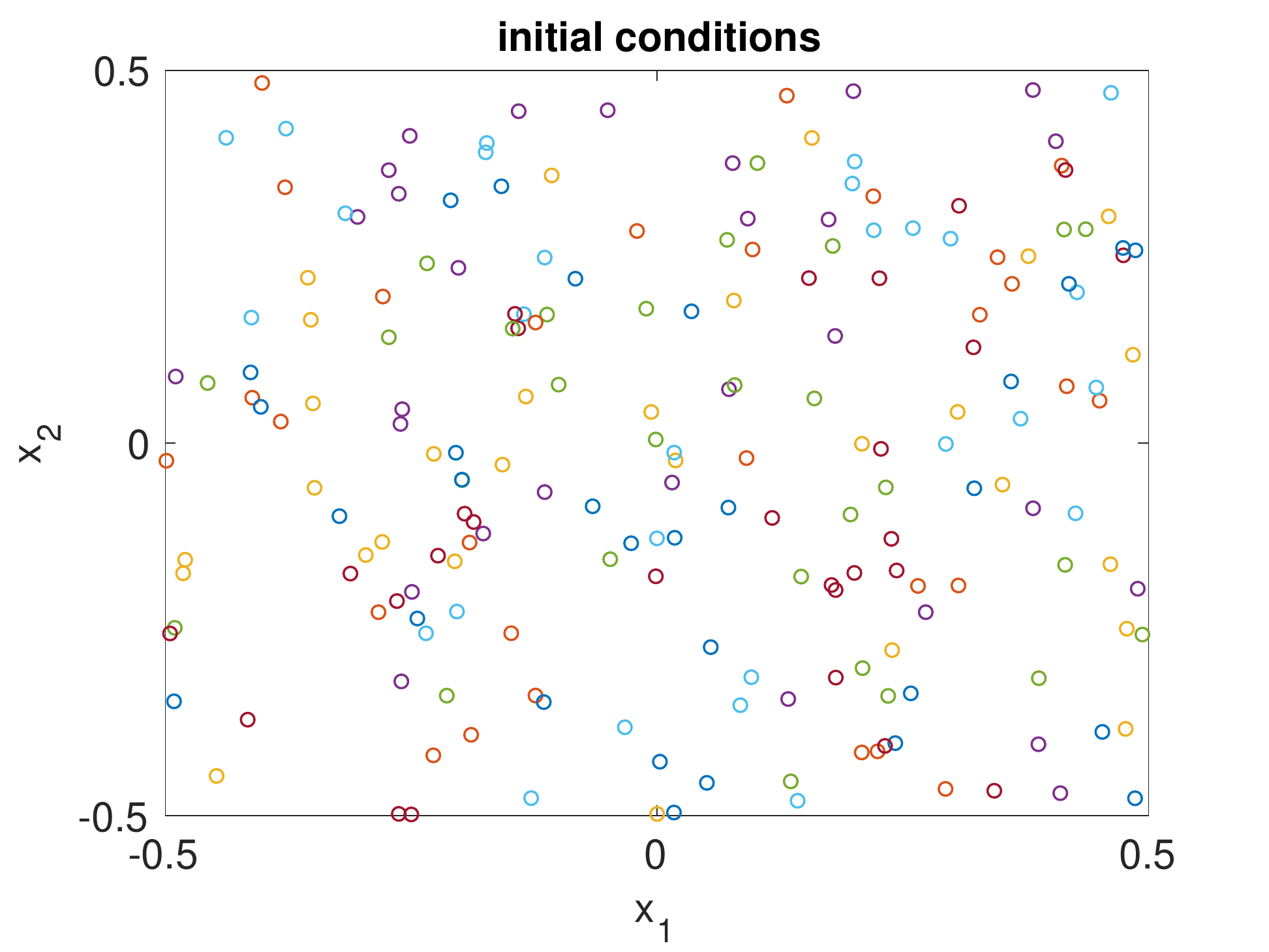}
\includegraphics[width=.325\textwidth]{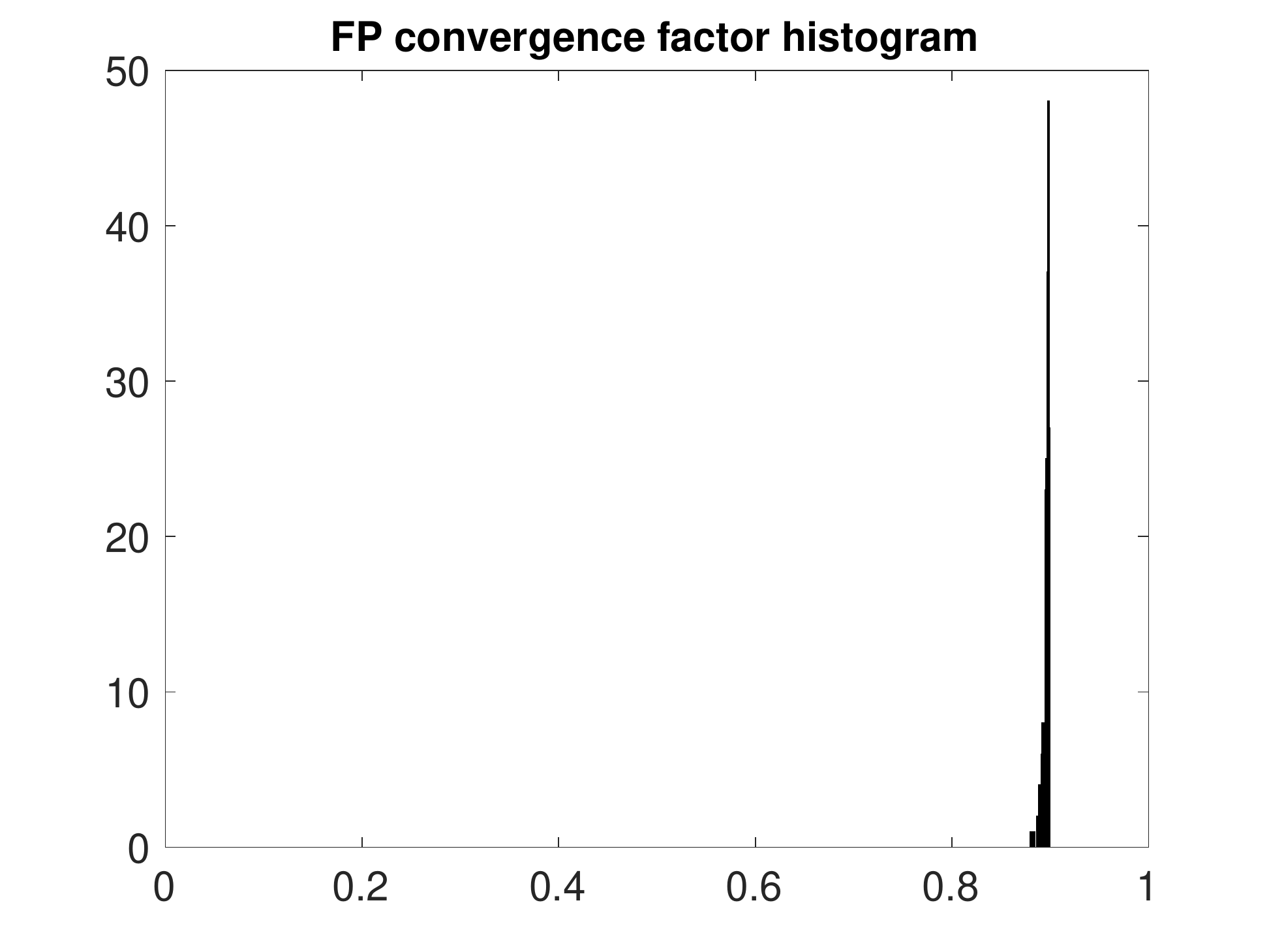}
\includegraphics[width=.325\textwidth]{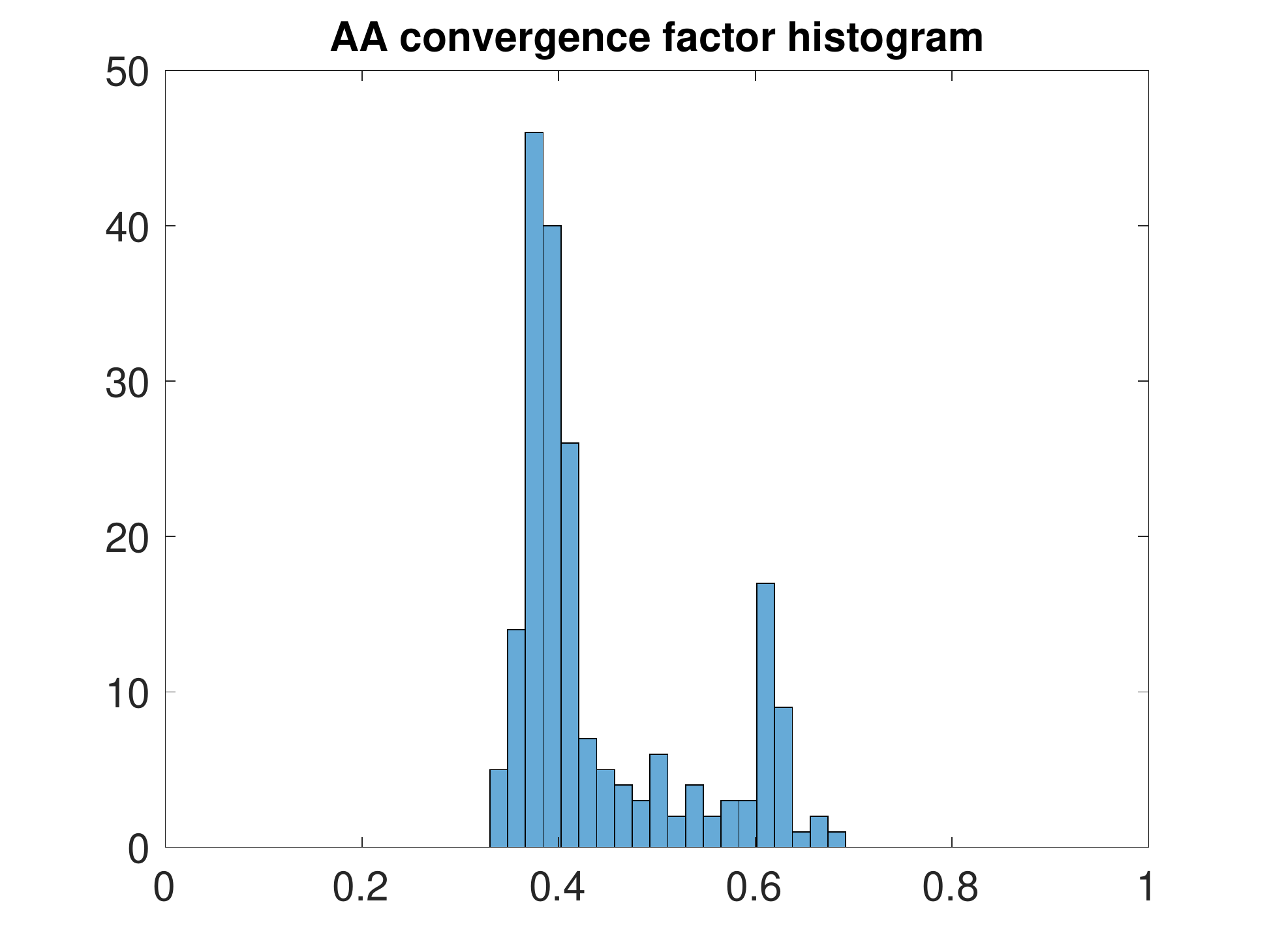}
\includegraphics[width=.325\textwidth]{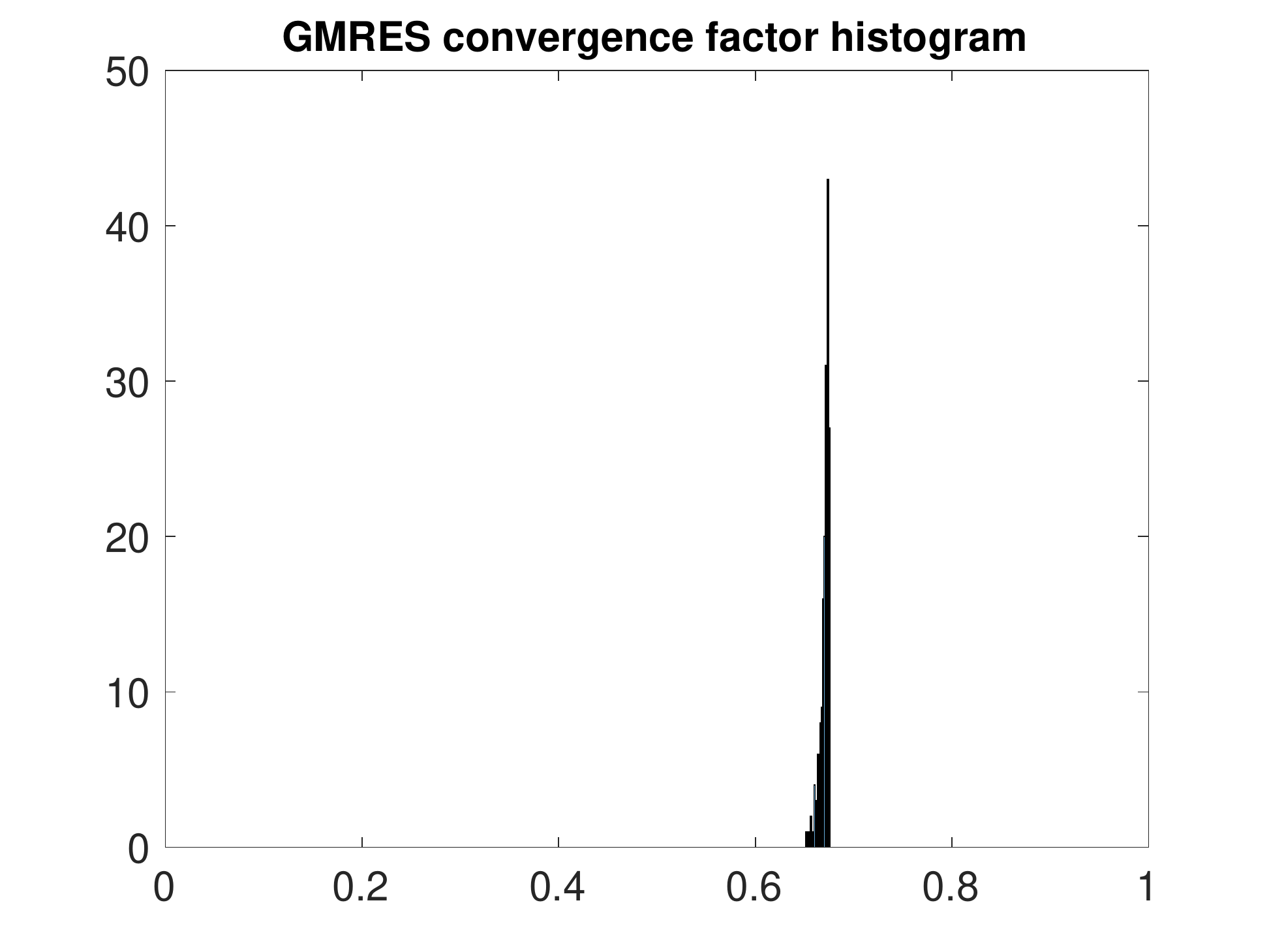}
\caption{
\cref{prob:linear200}: linear problem with $M \in \mathbb{R}^{200 \times 200}$ and $\lambda_1=0.9$, $\lambda_2=0.3$, $\lambda_3=-0.3$, and $\lambda_4=-0.3$, for 200 random initial guesses. Comparison of (windowed) AA(1) (red) with restarted GMRES(1) (blue). It can be observed that the asymptotic linear convergence factors of the AA(1) sequences strongly depend on the initial guess, but the asymptotic linear convergence factors of the restarted GMRES(1) sequences do not depend on the initial guess.} \label{prob18-gmres1}
\end{figure}

We next compare AA($m$) with a restarted version of AA($m$). In the restarted version of AA($m$), we simply restart the entire AA($m$) iteration every $m$ steps. Since AA is essentially equivalent to GMRES, this restarted AA($m$) iterations is essentially equivalent to restarted GMRES($m$) with window size $m$.

\cref{prob18-gmres1} compares AA(1) for \cref{prob:linear200} with restarted AA(1), equivalent to GMRES(1).
It is interesting to see that the asymptotic convergence factors of restarted AA(1) do not appear to depend on the initial guess.
Also, most of the sequences $\{x_k\}$ for the standard windowed AA(1) (without restart) appear to have an r-linear convergence factor that is
smaller than the r-linear convergence factor of the restarted AA(1).

\begin{figure}[h]
\centering
\includegraphics[width=.49\textwidth]{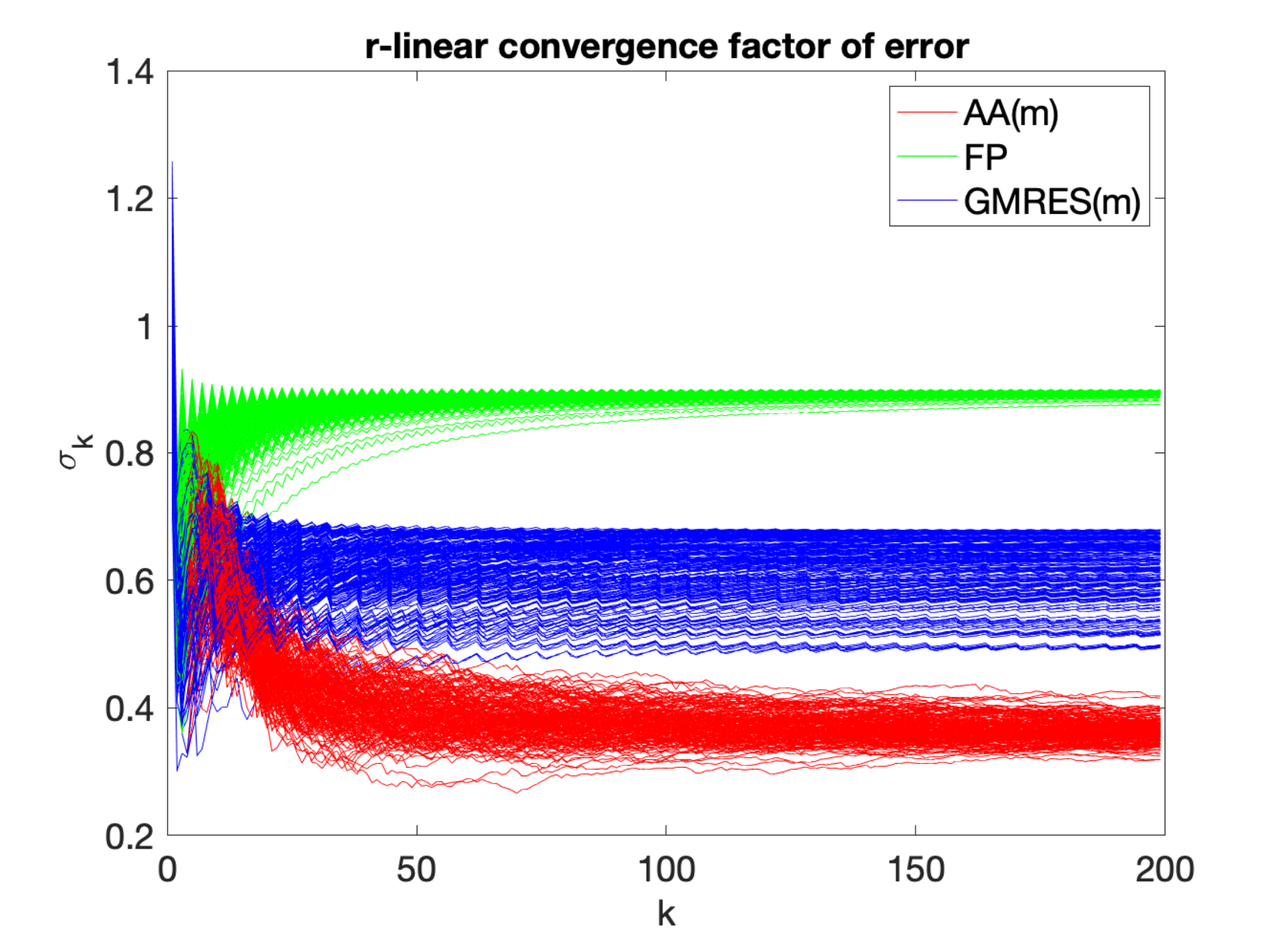}
\includegraphics[width=.49\textwidth]{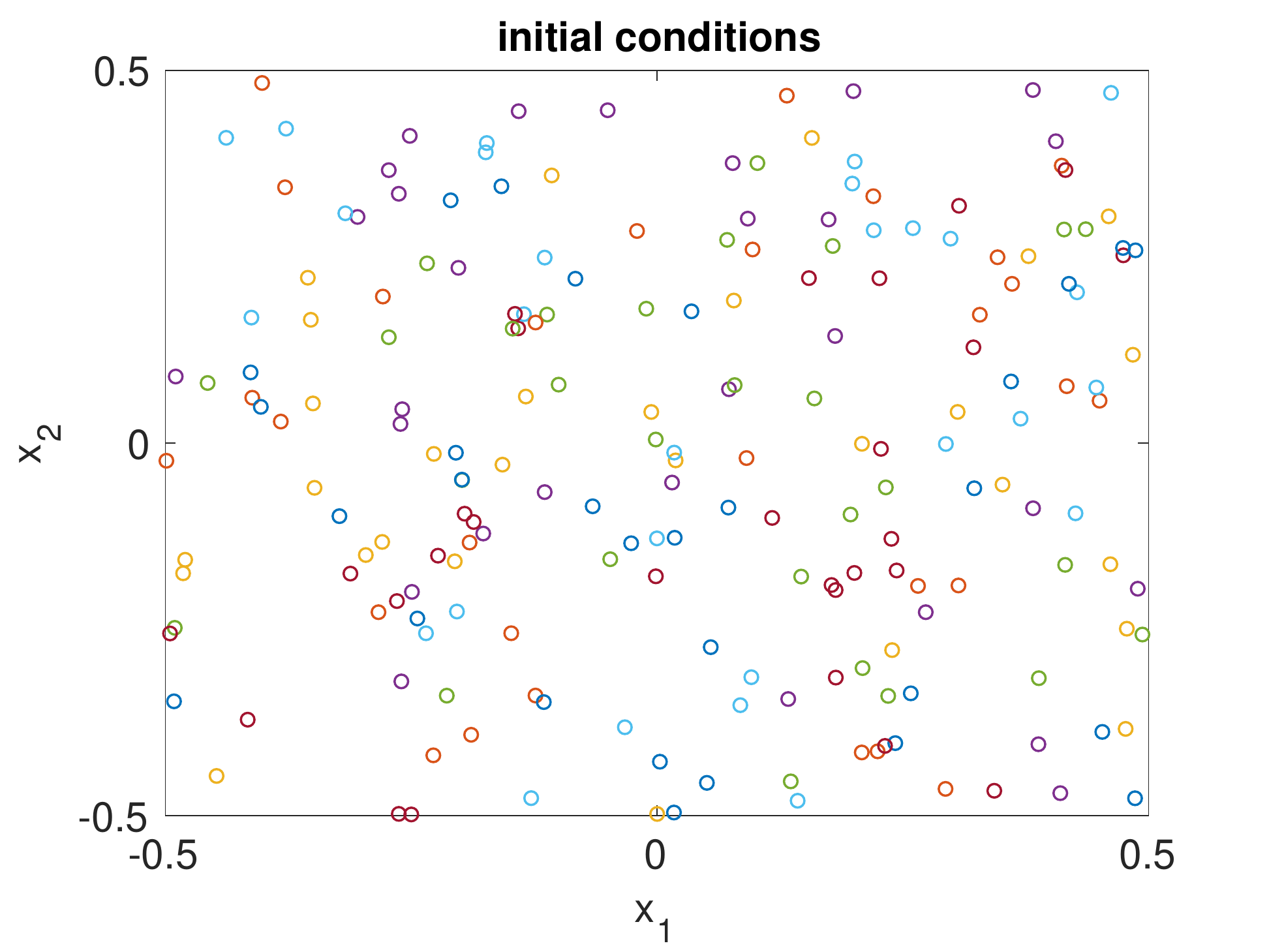}
\includegraphics[width=.32\textwidth]{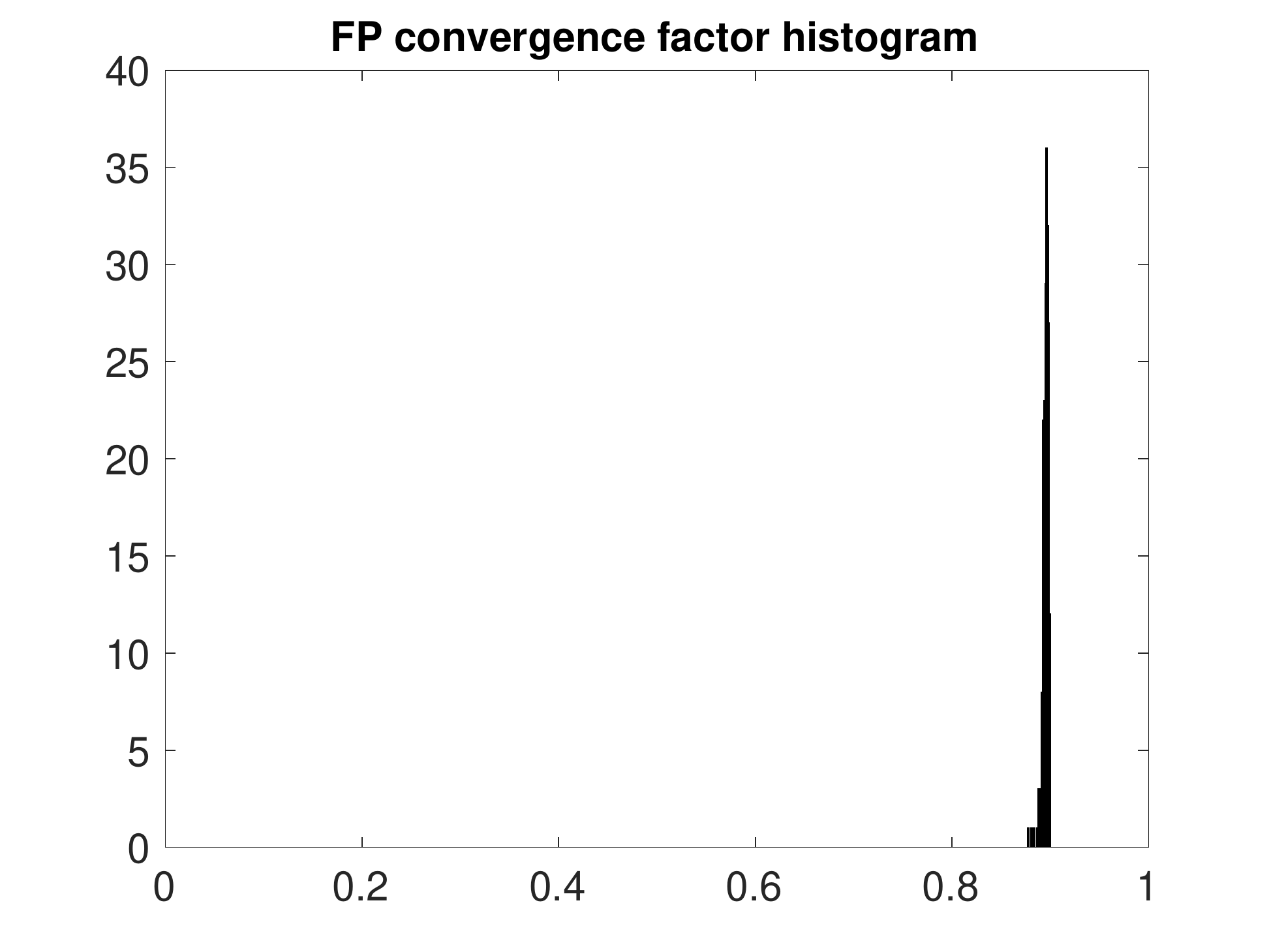}
\includegraphics[width=.32\textwidth]{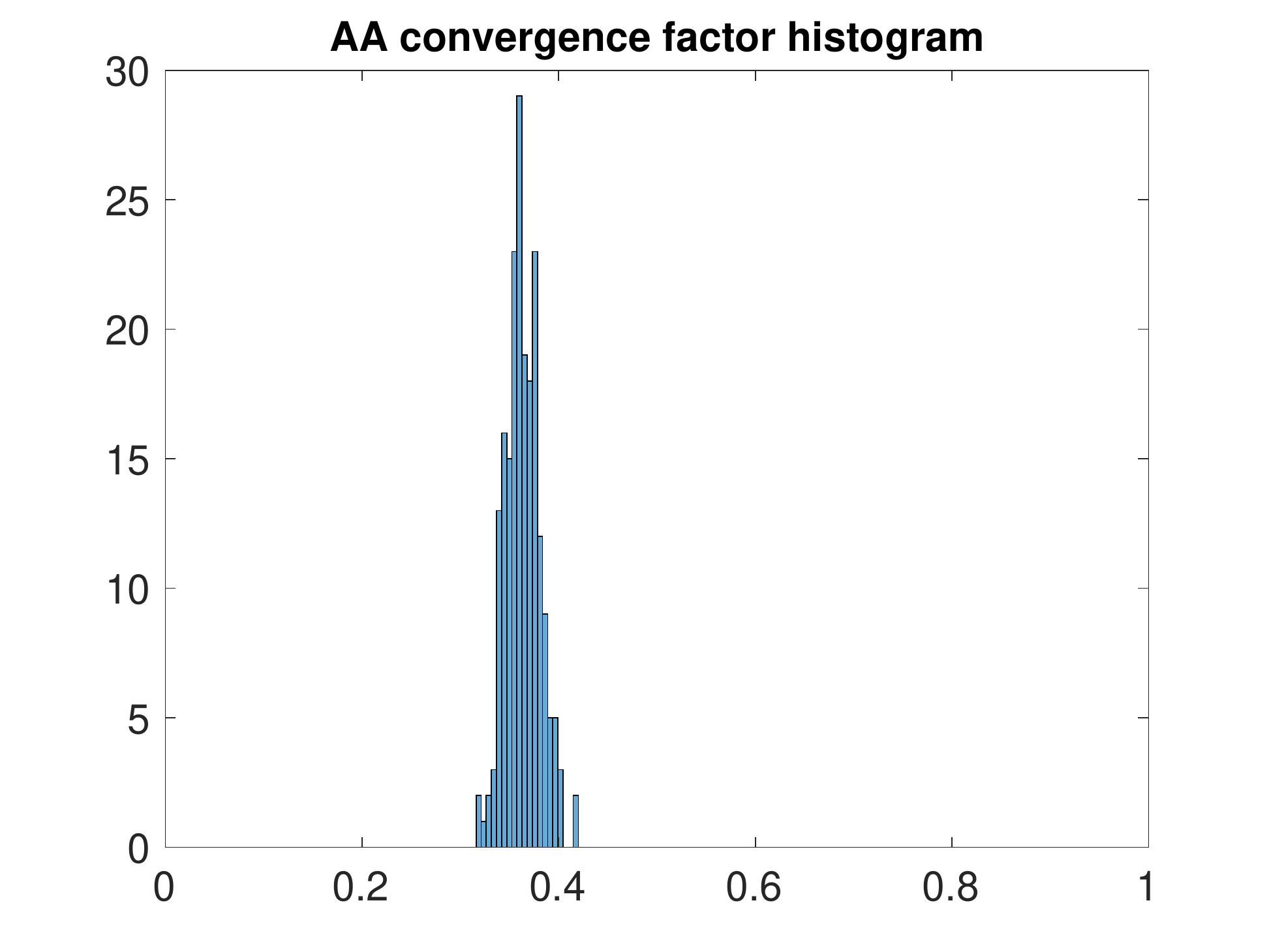}
\includegraphics[width=.32\textwidth]{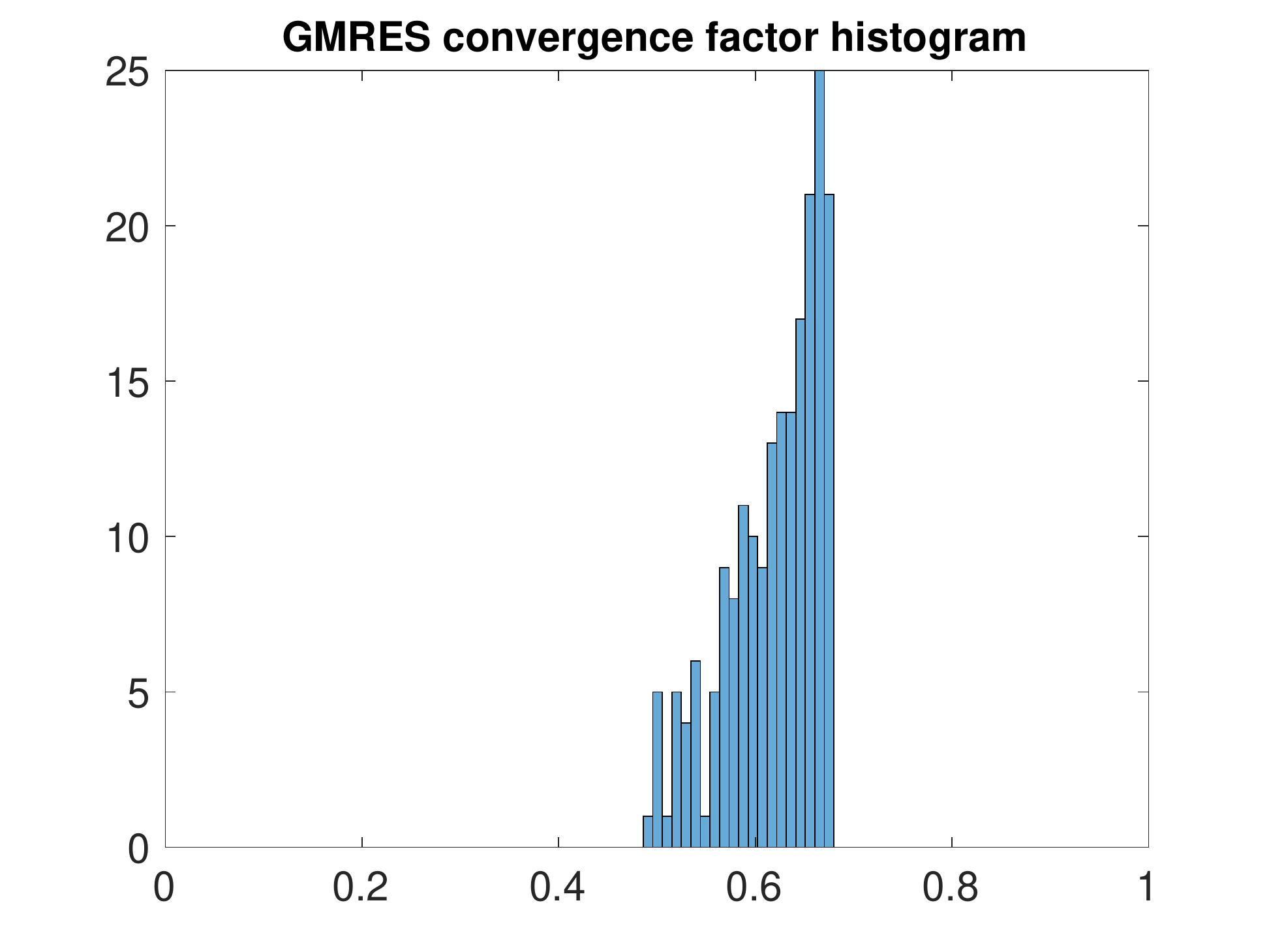}
\caption{
\cref{prob:linear200}: linear problem with $M \in \mathbb{R}^{200 \times 200}$ and $\lambda_1=0.9$, $\lambda_2=-0.9$, $\lambda_3=0.3$, and $\lambda_4=-0.3$, for 200 random initial guesses.
 Comparison of (windowed) AA(3) (red) with restarted AA(3) (blue), which is equivalent to GMRES(3). It can be observed that the asymptotic linear convergence factors of both the windowed AA(3) sequences and the restarted AA(3) sequences strongly depend on the initial guess. Windowed AA(3) converges faster than restarted AA(3).} \label{prob19-gmres3}
\end{figure}

Finally, \cref{prob19-gmres3} considers \cref{prob:linear200} with eigenvalues $\lambda_1=0.9$, $\lambda_2=-0.9$, $\lambda_3=0.7$, and $\lambda_4=-0.7$, comparing windowed AA(3) with restarted AA(3) (which is equivalent to GMRES(3)).
Interestingly, convergence factors for restarted AA(3) appear to depend strongly on the initial guess, similar to windowed AA(3), but unlike restarted AA(1) in \cref{prob18-gmres1}. We also see that windowed AA(3) generally converges faster than restarted AA(3), but this is not surprising because every iteration of windowed AA(3) uses information from three previous iterates (as soon as $k\ge3$), whereas iterations of restarted AA(3) use information from only two previous iterates on average.

The results of \cref{prob18-gmresinf,prob18-gmres1,prob19-gmres3} are interesting because they compare the asymptotic convergence speed of AA($m$) with GMRES and GMRES($m$), and the dependence of the r-linear convergence factor on the initial guess. Needless to say, these results raise many questions that require further investigation. For example, the dependence of GMRES($m$) convergence speed on the initial guess has been observed before and numerical results for small-size problems suggest dependence of GMRES(1) convergence with fractal patterns \cite{embree2003tortoise}, but as far as we know there are only limited theoretical results that explain, bound or quantify dependence of the asymptotic convergence factor on the initial guess for GMRES($m$).

\begin{figure}[h]
\centering
\includegraphics[width=.60\textwidth]{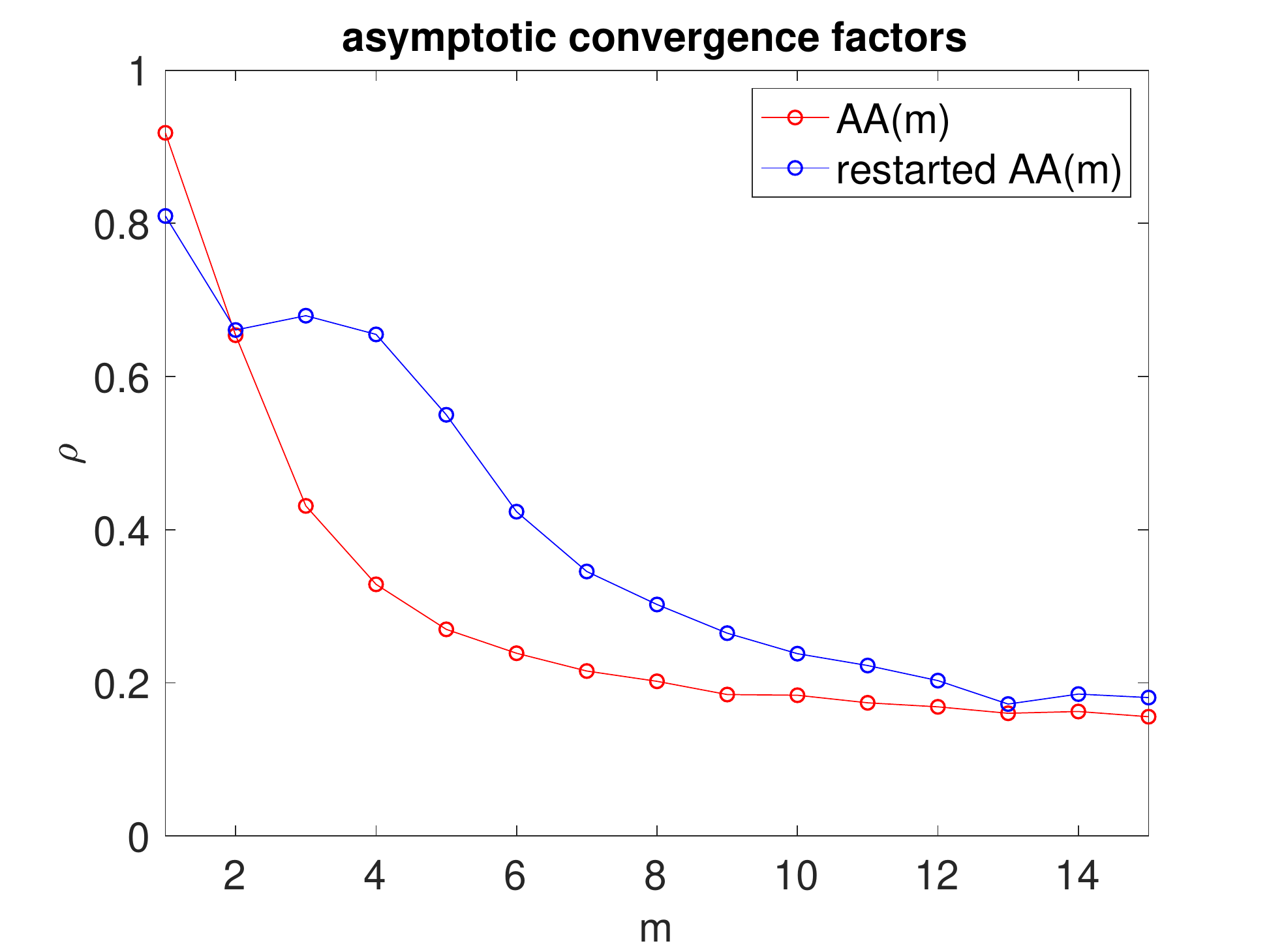}
\caption{
\cref{prob:linear200}: linear problem with $M \in \mathbb{R}^{200 \times 200}$ and $\lambda_1=0.9$, $\lambda_2=-0.9$, $\lambda_3=0.3$, and $\lambda_4=-0.3$. Comparison of worst-case asymptotic convergence factors for (windowed) AA(m) (red) and restarted AA(m) (blue), for 200 random initial guesses in each test.} \label{prob19-sweepm}
\end{figure}

We conclude our discussion of \cref{prob:linear200} with \cref{prob19-sweepm}, which shows how the worst-case r-linear asymptotic convergence factors for windowed AA($m$) and restarted AA($m$) over 200 random initial guesses depend on the window size, $m$. The results of \cref{prob19-sweepm} are for \cref{prob:linear200} with eigenvalues $\lambda_1=0.9$, $\lambda_2=-0.9$, $\lambda_3=0.7$, and $\lambda_4=-0.7$. Since there are four eigenvalues that are much greater than the cluster of 196 eigenvalues between 0.3 and 0, both the windowed and restarted AA($m$) show large gains from increasing $m$ up to $m=4$, after which the improvement tapers off. When $m$ increases beyond 5 (for windowed AA($m$)), the improvements become smaller, because the 196 eigenvalues that are smallest in magnitude are clustered; nevertheless, increasing $m$ further continues to improve the estimated $\rho_{\Psi,x^*}$. This behavior is similar to what could be expected for the well-understood behavior of GMRES without restart, which approximately takes out isolated eigenvalues one-by-one. As before for $m=3$ in \cref{prob19-gmres3}, windowed AA($m$) converges faster than restarted AA($m$) with the same $m$, as expected. However, this does point to an advantage of windowed AA($m$) over restarted AA($m$), because the memory requirements for the two algorithms are the same, and the additional amount of work per step for windowed AA($m$) is usually small because the small least-squares problems solved in AA($m$) tend to be inexpensive relative to the evaluation of the $q(x)$ iteration function.

\section{Conclusion}
\label{sec:disc}

In this paper, we have investigated the continuity and differentiability properties of the iteration function $\Psi(\boldsymbol{z})$ and acceleration coefficient function $\boldsymbol{\beta}(\boldsymbol{z})$ for AA($m$), Andersen acceleration with window size $m$. We have established, for window size $m=1$ and $m>1$, the continuity and Gateaux-differentiability of $\Psi(\boldsymbol{z})$ at the fixed point $\boldsymbol{z}^*$, despite $\boldsymbol{\beta}(\boldsymbol{z})$ not being continuous at $\boldsymbol{z}^*$. These findings shed light on remarkable properties of the asymptotic converge of AA($m$) that we have revealed in numerical experiments, for linear and nonlinear problems. We find that AA($m$) sequences converge r-linearly but their r-linear converge factors depend on the initial guess on a set of nonzero measure, which is consistent with the non-differentiability of $\Psi(\boldsymbol{z})$ at $\boldsymbol{z}^*$. The discontinuity of $\boldsymbol{\beta}(\boldsymbol{z})$ at $\boldsymbol{z}^*$ is consistent with the observed oscillatory behaviour of $\boldsymbol{\beta}^{(k)}$ as $\{x_k\}$ converges to $x^*$.
In exact arithmetic, the rank-deficient case is handled properly by the pseudo-inverse formula of \cref{eq:AAm-beta-form-pseudo}
which computes the minimum-norm solution when the system is singular, and our analysis shows that, while
$\boldsymbol{\beta}(\boldsymbol{z})$ is not continuous at $\boldsymbol{z}^*$, $\Psi(\boldsymbol{z})$ is
continuous and Gateaux-differentiable at $\boldsymbol{z}^*$ so the discontinuity of $\boldsymbol{\beta}(\boldsymbol{z})$ does
not preclude convergence of $\{x_k\}$ to $x^*$.

It is interesting to also relate the findings of this paper to the results from \cite{desterck2020,wang2020} on asymptotic convergence for a stationary version of AA($m$). While, as we have seen in this paper, the asymptotic convergence factor $\rho_{\Psi,x^*}$ of AA($m$) in iteration \cref{eq:AA-fixed-point} cannot easily be computed since $\Psi(\boldsymbol{z})$ is not differentiable, \cite{desterck2020,wang2020} consider a stationary version of AA($m$) where the AA coefficients $\beta^{(k)}_i$ are fixed over all iterations $k$. With fixed coefficients $\beta_i$ in \cref{eq:AA-fixed-point}, $\Psi(\boldsymbol{z})$ in iteration \cref{eq:AA-fixed-point} is differentiable and the linear asymptotic convergence factor of the stationary AA iteration is computable as  $\rho_{\Psi,x^*}=\rho(\Psi'(x^*))$. This enables choosing the stationary coefficients $\beta_i$ that minimize $\rho(\Psi'(x^*))$, if $x^*$ and $q'(x)$ are known. This approach is used in \cite{desterck2020,wang2020} to provide insight in the convergence improvement that results from the optimal stationary AA($m$) iteration, based on how AA($m$) improves the eigenvalue spectrum of $q'(x^*)$. Empirical results in \cite{desterck2020,wang2020} for AA($m$) acceleration of large canonical tensor decompositions by the alternating least-squares method, and of large machine learning optimization problems solved by the alternating direction method of multipliers, show that the convergence improvement obtained by the stationary AA($m$) iteration is similar to the convergence improvement provided by the non-stationary AA($m$). The work in \cite{desterck2020,wang2020}, however, as well as the fixed-point analysis of AA($m$) presented in this paper, leave open the question of determining $\rho_{\Psi,x^*}$ for the non-stationary AA($m$) that is widely used in science and engineering applications.


\appendix

\section{Proof of \cref{prop:Psix*}}
\label{app:proof-psi-lipsch}

We first prove the proposition for the linear case, obtaining an explicit global Lipschitz constant.
We then prove local Lipschitz continuity for the nonlinear case.

Consider $\boldsymbol{z}=\begin{bmatrix} x\\ x\end{bmatrix}$ and
$\boldsymbol{d}=\begin{bmatrix} d_1\\ d_2\end{bmatrix}$.

In the linear case, when $d_1\neq d_2$ we obtain from \cref{eq:Psilin}
\begin{equation}\label{eq:general-z-d-form}
 \Psi(\boldsymbol{z}+\boldsymbol{d}) =  \begin{bmatrix}
  (I-A)(x+d_1)+b- \ds \frac{(A(x+d_1)-b)^TA(d_1-d_2)}{(d_1-d_2)^TA^TA(d_1-d_2)}(I-A)(d_1-d_2)\\
   x+d_1
  \end{bmatrix},
 \end{equation}
and when $d_1= d_2$ we get from \cref{eq:form-Psi-x=y-linear}
\begin{equation}
   \Psi(\boldsymbol{z}+\boldsymbol{d})
   =\begin{bmatrix}
  q(x+d_1) \\
   x+d_1
  \end{bmatrix}=\begin{bmatrix}
  (I-A)(x+d_1)+b \\
   x+d_1
  \end{bmatrix}.
\end{equation}
We consider $\|\Psi(\boldsymbol{z}^*+\boldsymbol{d})-\Psi(\boldsymbol{z}^*) \|$ for two cases:
\begin{itemize}
\item  If $d_1\neq d_2$,
\begin{align*}
\|\Psi(\boldsymbol{z}^*+\boldsymbol{d})-\Psi(\boldsymbol{z}^*) \|&= \left\|\begin{bmatrix}
(I-A)d_1- (I-A)(d_1-d_2) \ds \frac{(Ad_1)^TA(d_1-d_2)}{(d_1-d_2)^TA^TA(d_1-d_2)} \\
d_1\end{bmatrix} \right\|, \\
&= \left\|\begin{bmatrix} (I-A)d_1-(A^{-1}-I) \ds \frac{A(d_1-d_2)}{\|A(d_1-d_2)\|} \ds \frac{(A(d_1-d_2))^T}{\|A(d_1-d_2)\|} Ad_1\\ d_1 \end{bmatrix} \right\|, \\
&\leq  \|(I-A)d_1-(A^{-1}-I) \ds \frac{A(d_1-d_2)}{\|A(d_1-d_2)\|} \ds \frac{(A(d_1-d_2))^T}{\|A(d_1-d_2)\|} Ad_1\|+ \|d_1\|, \\
& \leq   \|I-A\| \|d_1\|+\|A^{-1}-I\| \|\ds \frac{A(d_1-d_2)}{\|A(d_1-d_2)\|} \ds \frac{(A(d_1-d_2))^T}{\|A(d_1-d_2)\|}\| \|A d_1\| + \|d_1\|,\\
& \leq   \left(\|I-A\| +\|A^{-1}\|\|I-A\| \|A\|+1\right) \|d_1\|, \\
& \leq   \left( ( \|A^{-1}\| \|A\| + 1 ) \|I-A\|+1\right) \|\boldsymbol{d}\|.
\end{align*}
\item If $d_1=d_2=d$,
\begin{align*}
  \|\Psi(\boldsymbol{z}^*+\boldsymbol{d})-\Psi(\boldsymbol{z}^*) \| &= \left\|\begin{bmatrix}
  q(x^*+d)-q(x^*) \\  x^*+d-x^*  \end{bmatrix}\right \|  = \left\|\begin{bmatrix}
  (I-A)d \\  d  \end{bmatrix}\right \|,
\\
  &\leq (\|I-A\|)\|d\| +\|d\|,\\
  &\leq  (\|I-A\|+1) \|d\|,\\
  &\leq  (\|I-A\|+1) \|\boldsymbol{d}\|.
\end{align*}
\end{itemize}
Thus, for all $\boldsymbol{d}$,
\begin{equation*}
 \|\Psi(\boldsymbol{z}^*+\boldsymbol{d})-\Psi(\boldsymbol{z}^*) \| \leq L \|\boldsymbol{d}\|,
\end{equation*}
where $L=( \|A^{-1}\| \|A\| + 1 ) \|I-A\|+1$.
This means that $\Psi(\boldsymbol{z})$ is Lipschitz continuous at $\boldsymbol{z}^*$.


We now give the proof of local Lipschitz continuity for the nonlinear case.

Since we assume that $r'(x)$ is nonsingular and continuous for $x$ sufficiently close to $x^*$,
the smallest eigenvalue of  $r'^T(x)r'(x)$ is bounded below when $x$ is sufficiently close to $x^*$:
%
there exist $\delta_0>0$ and $c_r>0$  such that
\begin{equation}\label{eq:r'r-lower-bound}
 0< c_r^2 \leq \lambda_{\min} (r'(x^*+d)^T r'(x^*+d)), \, \forall d\in B(0,\delta_0):=\{d: \|d\|< \delta_0 \}.
 \end{equation}
%
Consider $\boldsymbol{z}^*=\begin{bmatrix} x^*\\ x^*\end{bmatrix}$ and $\boldsymbol{d}=\begin{bmatrix} d_1\\ d_2\end{bmatrix}$.
We first prove that $\Psi(\boldsymbol{z})$ is Lipschitz continuous at $\boldsymbol{z}^*$.

Note that
\begin{equation}\label{eq:q-taylor-exp}
q(x^*+d)=q(x^*)+ q'(x^*) d +Q_1(d)d \quad  \text{with}\quad \lim_{d \rightarrow 0} Q_1( d)=0.
\end{equation}
Let $L_1=\|q'(x^*)\|>0$. Then, there exists $\delta_1>0$ such that
\begin{equation}\label{eq:epsilon-definition-limit}
  \|Q_1(d)\|\leq L_1 \quad \forall d\in B(0,\delta_1).
\end{equation}
We first consider $d_1=d_2=d$. For  $d\in B(0,\delta_1)$,  from \cref{eq:Psi1,eq:beta1} we have
\begin{align*}
 \| \Psi(\boldsymbol{z}^*+\boldsymbol{ d})- \Psi(\boldsymbol{z}^*)\| &= \left\|\begin{bmatrix} q(x^*+d)-q(x^*) \\ d \end{bmatrix} \right\|,\\
     &\leq \|q(x^*+d)-q(x^*)\| +\|d\|,\\
     & \leq (\|q'(x^*)\|+ L_1)\|\boldsymbol{d}\| +\|\boldsymbol{d}\|,\\
     & \leq (2L_1+1)\|\boldsymbol{d}\|.
\end{align*}
When $d_1\neq d_2$,  from \cref{eq:Psi1,eq:beta1}, we have
\begin{align*}
  \| \Psi(\boldsymbol{z}^*+\boldsymbol{ d})-\Psi(\boldsymbol{z}^*)\| & =\left\| \begin{bmatrix}
   q(x^*+d_1)+\beta(\boldsymbol{z}^*+\boldsymbol{ d})\big(q(x^*+d_1)-q(x^*+d_2)\big)-q(x^*)\\
   x^*+d_1-x^*
   \end{bmatrix}\right\|,\\
   &\leq  \underbrace{\|q(x^*+d_1)-q(x^*)\|}_{=:E_1} + \underbrace{\|\beta(\boldsymbol{z}^*+\boldsymbol{ d})\big(q(x^*+d_1)-q(x^*+d_2)\big)\|}_{=:E_2}+\|\boldsymbol{d}\|.
\end{align*}
In the following we estimate $E_1$ and $E_2$ separately. From \cref{eq:q-taylor-exp,eq:epsilon-definition-limit},
\begin{equation}\label{eq:E1-term}
E_1\leq (\|q'(x^*)\|+L_1)\|d_1\|=2L_1\|d_1\|, \quad \forall d_1 \in B(0,\delta_1).
\end{equation}


Next, we estimate $E_2$.  Using  \cref{eq:Psi1,eq:beta1}, and $r(x)=x-q(x)$, we have
\begin{align*}
E_2&=\left \|\beta(\boldsymbol{z}^*+\boldsymbol{ d})\big(q(x^*+d_1)-q(x^*+d_2)\big)\right\|,\\
& =\left \|\frac{r(x^*+d_1)^T \big(r(x^*+d_1)-r(x^*+d_2)\big)}{\|r(x^*+d_1)-r(x^*+d_2)\|^2} \big(q(x^*+d_1)-q(x^*+d_2)\big)\right \|,\\
& = \left \|\frac{\big(q(x^*+d_1)-q(x^*+d_2)\big) \big(r(x^*+d_1)-r(x^*+d_2)\big)^T}{\|r(x^*+d_1)-r(x^*+d_2)\|^2}r(x^*+d_1) \right\|,\\
& = \left \|\frac{\big(r(x^*+d_1)-r(x^*+d_2)-(d_1-d_2)\big) \big(r(x^*+d_1)-r(x^*+d_2)\big)^T}{\|r(x^*+d_1)-r(x^*+d_2)\|^2}r(x^*+d_1) \right\|,\\
&\leq \left \|\frac{(r(x^*+d_1)-r(x^*+d_2)) (r(x^*+d_1)-r(x^*+d_2))^T}{\|r(x^*+d_1)-r(x^*+d_2)\|^2} \right \|\,\big\|r(x^*+d_1)\big\| \\
& \quad  +  \left \|\frac{(d_1-d_2)(r(x^*+d_1)-r(x^*+d_2))^T}{\|r(x^*+d_1)-r(x^*+d_2)\|^2}\right \|\,\big\|r(x^*+d_1)\big\|, \\
&\leq \|r(x^*+d_1)\| +  \underbrace{\left \|\frac{(d_1-d_2)(r(x^*+d_1)-r(x^*+d_2))^T}{\|r(x^*+d_1)-r(x^*+d_2)\|^2}\right \|}_{=:W}\, \big\|r(x^*+d_1)\big\|,\\
&= (1+W)\|r(x^*+d_1)\|.
\end{align*}
Using $r(x^*)=0$, we have
\begin{equation}\label{eq:dif-q1-q2}
r(x^*+d_1)=r'(x^*)d_1+Q_2(d_1)d_1 \quad \text{with} \quad  \lim_{d \rightarrow 0} Q_2( d_1)=0.
\end{equation}
Let $L_2=\|r'(x^*)\|$. Then there exists $\delta_2>0$ such that
\begin{equation}\label{eq:epsilon-definition-limit-r}
  \|Q_2(d_1)\|\leq L_2 \quad \text{for} \quad \forall d_1\in B(0,\delta_2).
\end{equation}
Using \cref{eq:dif-q1-q2} and \cref{eq:epsilon-definition-limit-r}, we have
\begin{equation}\label{eq:E2-estimate}
E_2 \leq (1+W)(\|r'(x^*)\| + L_2) \|d_1\|=2(1+W)L_2\|d_1\|, \quad \forall d_1\in B(0,\delta_2).
\end{equation}
Next, we estimate $W$.  Since,
\begin{equation}\label{eq:q-d1-d2-taylor-exp}
r(x^*+d_1)-r(x^*+d_2)=r'(x^*+d_2)(d_1-d_2) +Q_3(d_1-d_2)(d_1-d_2)
\end{equation}
with $\lim_{d_1-d_2 \rightarrow 0} Q_3( d_1-d_2)=0$, we have
\begin{align*}
W&= \left \|\frac{(d_1-d_2)(r(x^*+d_1)-r(x^*+d_2))^T}{\|r(x^*+d_1)-r(x^*+d_2)\|^2}\right \|,\\
  & \leq  \frac{\|d_1-d_2\|}{\|r(x^*+d_1)-r(x^*+d_2)\|},\\
  & =  \frac{\|d_1-d_2\|}{\|r'(x^*+d_2)(d_1-d_2) +Q_3(d_1-d_2)(d_1-d_2)\|},\\
  & = \frac{1}{\|r'(x^*+d_2)(d_1-d_2)/\|d_1-d_2\| +Q_3(d_1-d_2)(d_1-d_2)/\|d_1-d_2\|\|}.
\end{align*}
In the following, we show that $T =\|r'(x^*+d_2)(d_1-d_2)/\|d_1-d_2\| +Q_3(d_1-d_2)(d_1-d_2)/\|d_1-d_2\|\|$ is bounded from below.
Let $L_3= \frac{c_r}{2}$ where $c_r$ is defined in  \cref{eq:r'r-lower-bound}. From \cref{eq:q-d1-d2-taylor-exp}, there exists $\delta_3$ such that
\begin{equation*}
  \|Q_3(d_1-d_2)\|\leq L_3, \forall d_1-d_2\in B(0,\delta_3).
\end{equation*}
It follows that
\begin{equation}\label{eq:Q3-d1d2}
  \Big \|Q_3(d_1-d_2)(d_1-d_2)/\|d_1-d_2\| \Big\|\leq \|Q_3(d_1-d_2)\|\leq \frac{c_r}{2}.
\end{equation}
Using \cref{eq:r'r-lower-bound} and \cref{eq:Q3-d1d2} with $d_2\in B(0,\delta_0)$, we have, since $\|a+b\|\geq \big|\|a\|-\|b\|\big|$,
\begin{align*}
T &=\Big \|r'(x^*+d_2)(d_1-d_2)/\|d_1-d_2\| +Q_3(d_1-d_2)(d_1-d_2)/\|d_1-d_2\| \Big \| \\
 & \geq   \Big | \big\|r'(x^*+d_2)(d_1-d_2)/\|d_1-d_2\| \big\| - \big \|Q_3(d_1-d_2)(d_1-d_2)/\|d_1-d_2\| \big\| \Big |,\\
& \geq  \Big | \sqrt{\lambda_{\min} \big(r'(x^*+d_2)^Tr'(x^*+d_2)\big)}  -  \|Q_3(d_1-d_2)(d_1-d_2)/\|d_1-d_2\| \Big |,\\
& \geq   c_r -\frac{c_r}{2} ,\\
& =\frac{c_r}{2},
\end{align*}
where in the second-to-last inequality, we use the following property: for any Hermitian matrix $H$,
\begin{equation*}
  x^T H x\geq \lambda_{\min}(H), \quad \text{where} \quad x^Tx=1.
\end{equation*}
Thus,
\begin{equation}\label{eq:estimate-W}
 W =\frac{1}{T}\leq \frac{2}{c_r}.
\end{equation}
To use \cref{eq:E1-term,eq:E2-estimate,eq:estimate-W} together, we need to find a $\delta>0$ such that $\boldsymbol{d}=\begin{bmatrix} d_1 \\ d_2\end{bmatrix}\in B(0,\delta)$ satisfies $d_1\in B(0,\delta_i)$ with $i=1,2$, $d_2\in B(0,\delta_0)$, and $d_1-d_2\in B(0,\delta_3)$. Let  $\delta=\frac{1}{2}\min\big \{ \delta_0,\delta_1,\delta_2,\delta_3  \big \}$  and  $\boldsymbol{d}=\begin{bmatrix} d_1 \\ d_2\end{bmatrix}\in B(0,\delta)$. Then  we have $d_1, d_2 \in B(0,\delta)$ and $d_1-d_2\in B(0, 2\delta)$. It follows that $d_1\in B(0,\delta_i)$ with $i=1,2$, $d_2\in B(0,\delta_0)$, and $d_1-d_2 \in B(0,\delta_3)$.

Now, using \cref{eq:E1-term,eq:E2-estimate,eq:estimate-W}, for $\boldsymbol{d}\in B(0,\delta)$,  we have the following estimate:
\begin{align*}
\| \Psi(\boldsymbol{z}^*+\boldsymbol{ d})-\Psi(\boldsymbol{z}^*)\| & \leq E_1+E_2+ \|\boldsymbol{d}\|,\\
 & \leq 2L_1\|d_1\| +  2(1+W)L_2 \|d_1\|+ \|\boldsymbol{d}\|,\\
 &\leq \big( 2 \|q'(x^*)\|  +  2(1+2/c_r) \|r'(x^*)\| +1\big)\|\boldsymbol{d}\|,\\
 &\leq \big( 2 \|I-r'(x^*)\|  +  2(1+2/c_r) \|r'(x^*)\| +1\big)\|\boldsymbol{d}\|,\\
 & \leq \big( 3+ (4+4/c_r)\|r'(x^*)\|\big)\|\boldsymbol{d}\|.
\end{align*}
Thus, for all $\boldsymbol{d}\in B(0,\delta)$,
\begin{equation*}
 \|\Psi(\boldsymbol{z}^*+\boldsymbol{d})-\Psi(\boldsymbol{z}^*) \| \leq L \|\boldsymbol{d}\|,
\end{equation*}
where $L=3+(4+4/c_r)\|r'(x^*)\|$. This means that $\Psi(\boldsymbol{z})$ is locally Lipschitz continuous at $\boldsymbol{z}^*$.

\section{Proof of \cref{thm:x=y-directional-dif}}
\label{app:proof-directional}

Consider $\boldsymbol{z}^*=\begin{bmatrix} x^*\\ x^*\end{bmatrix}$, $\boldsymbol{d}=\begin{bmatrix} d_1\\ d_2\end{bmatrix}$ and scalar $h\neq 0$.

When $d_1=d_2$, from \cref{eq:Psi1} and \cref{eq:beta1}, we have
\begin{equation*}
\Psi(\boldsymbol{z}^*+h\boldsymbol{d})-\Psi(\boldsymbol{z}^*) =
\begin{bmatrix}
 q(x^*+hd_1)-q(x^*)\\
   x^*+hd_1-x^*
  \end{bmatrix}.
\end{equation*}
Since
\begin{equation}\label{eq:taylor-exp-two-term}
q(x^*+ hd_1)=q(x^*)+ q'(x^*)hd_1 +Q(hd_1) hd_1\quad \text{with} \quad  \lim_{h \rightarrow 0} Q(hd_1)=0,
\end{equation}
\begin{align*}
 \mathfrak{ D} \Psi(\boldsymbol{z}^*,\boldsymbol{ d})&= \lim_{h\downarrow 0}\ds \frac{\Psi(\boldsymbol{z}^*+h\boldsymbol{d})-\Psi(\boldsymbol{z}^*)}{h},\\
 & = \lim_{h\downarrow 0}\ds  \frac{1}{h}\begin{bmatrix}
        q'(x^*)hd_1 +Q(hd_1) hd_1 \\
      hd_1
    \end{bmatrix},\\
  &=\begin{bmatrix} q'(x^*)d_1\\ d_1\end{bmatrix},\\
  & = \begin{bmatrix} M d_1\\ d_1\end{bmatrix},\\
  & = \begin{bmatrix} (1+0)M & 0\cdot M \\
                      I  & 0
    \end{bmatrix} \boldsymbol{d}.
\end{align*}
When $d_1\neq d_2$, we obtain from \cref{eq:Psi1} and \cref{eq:beta1},
\begin{align*}
\Psi(\boldsymbol{z}^*+h\boldsymbol{d})-\Psi(\boldsymbol{z}^*) &=
\begin{bmatrix}
  q(x^*+hd_1) +\beta(\boldsymbol{z}^*+h\boldsymbol{d}) (q(x^*+hd_1)-q(x^*+hd_2))-q(x^*) \\
    x^*+hd_1- x^*
\end{bmatrix},\\
& = \begin{bmatrix}
A_1+ A_2\\
hd_1
\end{bmatrix},
\end{align*}
where $A_1=q(x^*+hd_1)-q(x^*)$ and
\begin{align*}
 A_2&=\beta(x^*+hd_1,x^*+hd_2)(q(x^*+hd_1)-q(x^*+hd_2)),\\
    &=\frac{-r(x^*+hd_1)^T \big(r(x^*+hd_1)-r(x^*+hd_2)\big)}{\|r(x^*+hd_1)-r(x^*+hd_2)\|^2}\big(q(x^*+hd_1)-q(x^*+hd_2)\big).
\end{align*}
Using \cref{eq:taylor-exp-two-term}, we have
\begin{equation}\label{eq:limit-A1-term}
\lim_{h\downarrow 0} \ds \frac{A_1}{h}  =q'(x^*)d_1=Md_1.
\end{equation}
Next, we simplify $A_2$. First,
\begin{equation}\label{eq:q(x)-expansion-2}
  q(x^*+hd_2)-q(x^*) = q'(x^*) hd_2 + Q(hd_2)(hd_2) \quad \text{with}\quad  \lim_{h\rightarrow 0} Q(hd_2)=0.
\end{equation}
From  \cref{eq:taylor-exp-two-term} and \cref{eq:q(x)-expansion-2}, we have
\begin{equation*}
q(x^*+hd_1)-q(x^*+hd_2) = q'(x^*)(d_1-d_2) h + Q(hd_1)(hd_1)+Q(hd_2)(hd_2)=:P,
\end{equation*}
and
\begin{align*}
r(x^*+hd_1) &=x^*+hd_1-q(x^*+hd_1),\\
            &=x^*+hd_1-(q(x^*)+q'(x^*)hd_1 +Q(hd_1)(hd_1)),\\
            &= h(I-q'(x))d_1-Q(hd_1)(hd_1), \\
            & =h Ad_1-Q(hd_1)(hd_1).
\end{align*}
Furthermore,
\begin{align*}
r(x^*+hd_1)-r(x^*+hd_2)&=h Ad_1- Q(hd_1)(hd_1))-\big(h Ad_1- Q(hd_1)(hd_1)\big), \\
&= hA(d_1-d_2) - Q(hd_1)(hd_1)+Q(hd_2)(hd_2).
\end{align*}
Thus,
\begin{align*}
  \lim_{h\downarrow 0} \ds \frac{A_2}{h}  & =  \ds \frac{-\big( hA d_1 -Q(hd_1)(hd_1)\big)^T ( hA(d_1-d_2)-Q(hd_1)(hd_1)+Q(hd_2)(hd_2))P }{ h\| hA(d_1-d_2) +Q(hd_1)(hd_1)+Q(hd_2)(hd_2)\|^2},\\
&= \ds \frac{-(Ad_1)^TA(d_1-d_2) M(d_1-d_2)}{(A(d_1-d_2))^TA(d_1-d_2))},\\
&= -\ds \frac{d_1^TA^TA(d_1-d_2) M(d_1-d_2)}{(d_1-d_2)^TA^TA(d_1-d_2)}.
\end{align*}
It follows
\begin{equation}\label{eq:limit-A1-A2}
  \lim_{h\downarrow 0}\ds \frac{A_1+A_2}{h} = Md_1 - \ds \frac{d_1^TA^TA(d_1-d_2)}{(d_1-d_2)^TA^TA(d_1-d_2)} M(d_1-d_2).
\end{equation}
According to the definition of directional derivative in \cref{def-direction-diff} and to \cref{eq:limit-A1-A2}, we have
\begin{align}
 \mathfrak{D} \Psi(\boldsymbol{z}^*,\boldsymbol{d})&=
  \lim_{h\downarrow 0}\ds \frac{\Psi(\boldsymbol{z}^*+h\boldsymbol{d})-\Psi(\boldsymbol{z}^*)}{h},\nonumber\\
  & =  \lim_{h\downarrow 0}\ds \frac{1}{h}
  \begin{bmatrix}
  A_1+A_2\\
  hd_1
  \end{bmatrix},\nonumber\\
  &=
 \begin{bmatrix}
  (1+\widehat{\beta}(\boldsymbol{d})) M & -\widehat{\beta}(\boldsymbol{d}) M \\
  I  & 0
  \end{bmatrix}\boldsymbol{d},\label{eq:DPsi}
\end{align}
where
\begin{equation*}
   \widehat{\beta}(\boldsymbol{d}) = - \ds \frac{d_1^TA^TA(d_1-d_2)}{(d_1-d_2)^TA^TA(d_1-d_2)}.
\end{equation*}
Thus, for all $\boldsymbol{d}$, we have  proved \cref{AA(1)-direction-diff}.

\section{$\beta(\boldsymbol{z})$ and $\Psi(\boldsymbol{z})$ for AA(1) in the scalar case ($n=1$)}
\label{app:scalar}

For a scalar problem, with, as before, $r_k =x_k-q(x_k)$, we have from \cref{eq:AA-1-step-beta} that
\begin{equation}\label{eq:beta_scalar}
  \beta_k =\ds \frac{-r_k}{r_k-r_{k-1}}.
\end{equation}
It is well-known that, in the scalar case, AA(1) method \cref{eq:anderson-1-step}
reduces to the secant method for solving
$$f(x)=x-q(x),$$
as follows:
\begin{eqnarray*}
x_{k+1} &=&\ds q(x_k) +\ds \beta_k (q(x_k)-q(x_{k-1})),\\
&=&\ds \frac{-r_{k-1}}{r_k-r_{k-1}} q(x_k) +\ds \frac{r_k}{r_k-r_{k-1}}q(x_{k-1}),\\
&=& \ds \frac{-f(x_{k-1})(x_k-f(x_k))+f(x_k)(x_{k-1}-f(x_{k-1}))}{f(x_k)-f(x_{k-1})},\\
&=&\ds \frac{x_{k-1}f(x_k)-x_k f(x_{k-1})}{f(x_k)-f(x_{k-1})}=x_k -\ds \frac{x_k -x_{k-1}}{f(x_k)-f(x_{k-1})} f(x_k).
\end{eqnarray*}
It is also well-known that the secant method, for a simple root, converges $q$-superlinearly with order $p=\ds \frac{1+\sqrt{5}}{2}$, that is,
\begin{equation*}
  \lim_{k\rightarrow \infty}  \ds \frac{|x_k-x^*|}{|x_{k-1}-x^*|^{p}} =L>0.
\end{equation*}
It follows that
\begin{equation*}
   \lim_{k\rightarrow \infty} \ds \frac{|x_k-x^*|}{|x_{k-1}-x^*|} =0
   \qquad \textrm{and}
   \qquad
   \lim_{k\rightarrow \infty} \ds |x_k-x^*|^{1/k}=0.
\end{equation*}

\cref{tab:continuity-differentiability-Lip} indicates that $\beta(\boldsymbol{z})$ is not continuous
at $\boldsymbol{z}^*=\begin{bmatrix} x^*\\ x^*\end{bmatrix}$, and $\Psi(\boldsymbol{z})$ is not differentiable at $\boldsymbol{z}^*$.
Let $\beta_k=\beta(x_k,x_{k-1})$ with $\{x_k\}$ the sequence generated by AA(1).
We now show that, when $n=1$, $\lim_{k \rightarrow \infty} \beta_k =0$, despite
$\beta(\boldsymbol{z})$ not being continuous at $\boldsymbol{z}^*$.
\cref{rmk-Psi-diff-n=1} also indicates that
$ \mathfrak{ D} \Psi(\boldsymbol{z}^*,\boldsymbol{d})=B\boldsymbol{d}$,
with $\rho(B)=0$, except when $d_1=d_2$.
We will discuss how these results for $n=1$ lead to markedly different convergence behavior
for $\beta_k$ and the root-averaged error $\sigma_k$ from \cref{eq:sigma}, compared to the case $n>1$.

\begin{theorem}\label{thm:limit-betak}
Let $\beta_k=\beta(x_k,x_{k-1})$ with $\{x_k\}$ the sequence generated by AA(1) applied to iteration \cref{eq:fixed-point} in the scalar case ($n=1$). Then
\begin{equation}\label{eq:beta-limit-scalar}
    \lim_{k \rightarrow \infty} \beta_k =0.
\end{equation}
\end{theorem}
\begin{proof}
Using  \cref{eq:beta_scalar}, $r(x_k)=r(x^*)+ r'(x^*)(x_k-x^*)+Q(x_k-x^*) (x_k-x^*)$ with $\lim_{k\rightarrow \infty} Q(x_k-x^*)=0$,
and $r(x_{k-1})=r(x^*)+ r'(x^*)(x_{k-1}-x^*)+Q(x_{k-1}-x^*) (x_{k-1}-x^*)$ with $\lim_{k\rightarrow \infty} Q(x_{k-1}-x^*)=0$,
 we obtain
\begin{align*}
\beta_k& = - \ds \frac{r'(x^*)(x_k-x^*)+Q(x_k-x^*) (x_k-x^*)}{r'(x^*)(x_k-x_{k-1})+Q(x_k-x^*) (x_k-x^*)-Q(x_{k-1}-x^*) (x_{k-1}-x^*)},\\
  & = - \ds \frac{x_k-x^*}{x_{k-1}-x^*} \ \frac{r'(x^*)+Q(x_k-x^*)}{\ds \frac{ r'(x^*)(x_k-x^*+x^*-x_{k-1})+Q(x_k-x^*)(x_k-x^*)}{x_{k-1}-x^*}-Q(x_{k-1}-x^*) }.
\end{align*}
Since $\ds \frac{x_k-x^*}{x_{k-1}-x^*}\rightarrow 0$ as $k\rightarrow \infty $ due to the superlinear convergence of the secant method, and since $r'(x^*)\neq 0$ ($r'(x)$ is nonsingular),
\begin{equation*}
\lim_{k \rightarrow \infty} \beta_k=- 0\cdot \frac{r'(x^*)}{r'(x^*)}=0.
\end{equation*}
\end{proof}
We now consider a simple scalar example to illustrate how the theoretical results relate to numerical convergence behavior.

\begin{problem}\label{prob:scalar}
We solve the scalar equation $x^2-x-1=0$ ($n=1$) using iteration \cref{eq:fixed-point} with
\begin{equation*}
  q(x) = 1+\ds \frac{1}{x}.
\end{equation*}
\end{problem}

\begin{figure}[h]
\centering
\includegraphics[width=.49\textwidth]{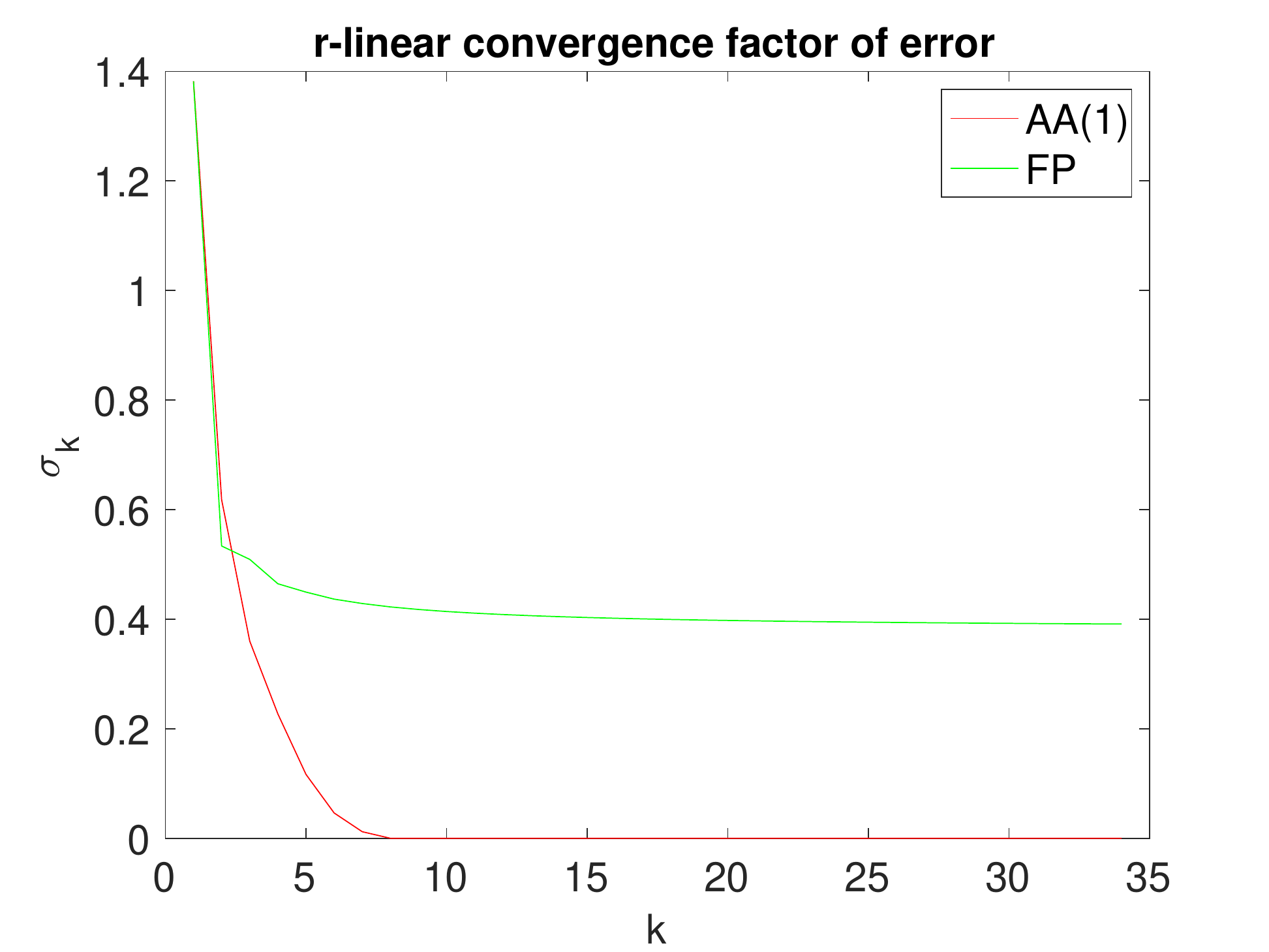}
\includegraphics[width=.49\textwidth]{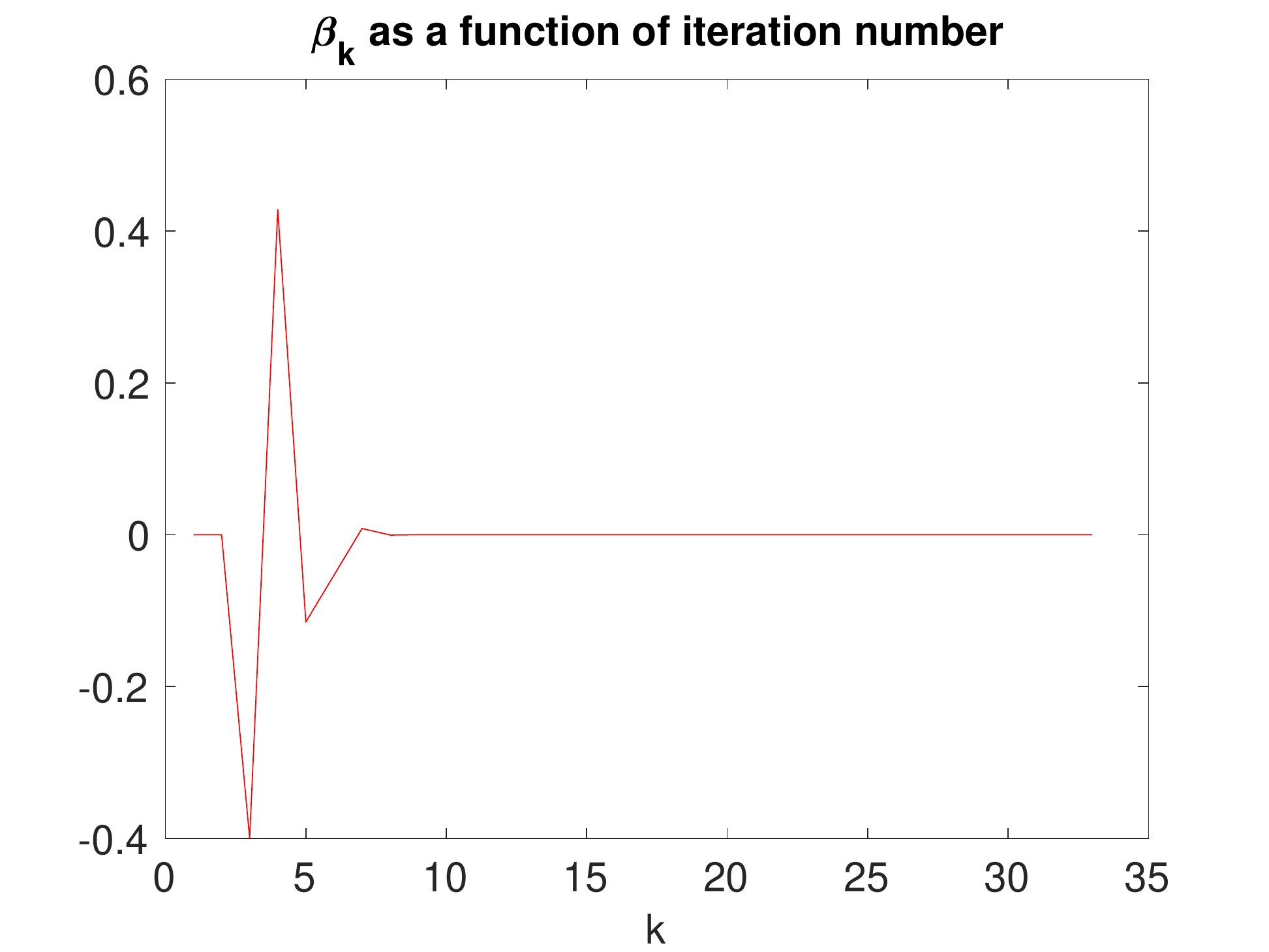}
\caption{\cref{prob:scalar} (scalar, $n=1$). (left panel) $\sigma_k$ as a function of iteration number $k$. (right panel) $\beta_k$ as a function of iteration number $k$.}  \label{AA-scalar-prob-plot}
\end{figure}

Consider initial guess $x_0=0.5$.  \cref{AA-scalar-prob-plot} shows $\beta_k$ and the root-averaged error $\sigma_k$ from \cref{eq:sigma} for iterations \cref{eq:fixed-point} and AA(1) as functions of iteration number $k$, for $x_k$ converging to $x^*=\ds \frac{1+\sqrt{5}}{2}$.
Consistent with the result from
\cref{thm:limit-betak} for this scalar problem, $\beta_k$ for AA(1) goes to 0 as $x_k\rightarrow x^*$ for $k\rightarrow \infty$.
The root-averaged error $\sigma_k$ for AA(1) converges to 0 as $x_k\rightarrow x^*$, consistent with the superlinear convergence of the secant method.
This superlinear convergence is also reflected in the finding of \cref{rmk-Psi-diff-n=1} that
$\mathfrak{ D} \Psi(\boldsymbol{z}^*,\boldsymbol{d})=B\boldsymbol{d}$ with $\rho(B)=0$ for almost all vectors $\boldsymbol{d}$ (except when $d_1=d_2$).
We also observe linear convergence for iteration \cref{eq:fixed-point}, with
$\rho_{q,x^*}=|q'(x^*)| =|-\ds \frac{1}{(x^*)^2}|=\ds \frac{2}{3+\sqrt{5}}\approx 0.382$.

In contrast, for the problem with $n>1$ in \cref{fig:simple-beta}, $\Psi(\boldsymbol{z})$ is not differentiable at $\boldsymbol{z}^*$ and the AA(1) convergence factor $\rho_{AA(1),x^*}$ cannot be determined from $\rho(\Psi'(\boldsymbol{ z}^*))$; AA(1) convergence is no longer superlinear. For $n>1$, $\beta_k$ no longer converges to 0 as $x_k \rightarrow x^*$, but is oscillatory, with $\beta(\boldsymbol{z})$ not being continuous at $\boldsymbol{z}^*$.
Despite $\beta(\boldsymbol{z})$ not being continuous at $\boldsymbol{z}^*$, $\Psi(\boldsymbol{z})$ is continuous
at $\boldsymbol{z}^*$, and AA(1) converges, albeit linearly instead of superlinearly, and despite the oscillations in $\beta_k$.

\section{Proof of \cref{AAm-prop-noncon-rank-zm+1=x*}}
\label{app:proof-lipsch-m}

Consider  $\boldsymbol{z} =\begin{bmatrix} z_{m+1}^T & z_m^T  & \ldots & z_1^T \end{bmatrix}^T $ with $z_{m+1}=x^*$ and such that $R(\boldsymbol{z})$ is rank-deficient.  Denote
$\boldsymbol{d} =\begin{bmatrix} d_{m+1}^T & d_m^T &  \ldots & d_1^T \end{bmatrix}^T$.
In the linear case, where $q(x)=Mx+b$ with $M=I-A$, we have
\begin{align}
  &R(\boldsymbol{z}) = AD(\boldsymbol{z}), \label{eq:R-form-linear}\\
  &Q(\boldsymbol{z}) = MD(\boldsymbol{z})=(I-A)D(\boldsymbol{z}), \label{eq:Q-form-linear} \\
  &\boldsymbol{\beta}(\boldsymbol{z}) = - \big(AD(\boldsymbol{z})\big)^{\dagger} (Az_{m+1}-b). \label{eq:beta-form-linear}
\end{align}
It then follows that
\begin{align}
  R(\boldsymbol{z}+\boldsymbol{d}) &= AD(\boldsymbol{z}+\boldsymbol{d})= AD(\boldsymbol{z})+AD(\boldsymbol{d})=: AD_1+AD_2=A\widehat{D}, \label{eq:R-z+d}\\
  Q(\boldsymbol{z}+\boldsymbol{d}) &=MD(\boldsymbol{z}+\boldsymbol{d})=(I-A)\widehat{D}, \label{eq:Q-z+d}\\
  \boldsymbol{\beta}(\boldsymbol{z}+\boldsymbol{d}) &= -\big( R(\boldsymbol{z}+\boldsymbol{d})\big)^{\dagger}(A(z_{m+1}+d_{m+1})-b).\label{eq:beta-z+d}
\end{align}
Since $z_{m+1}=x^*$, $\boldsymbol{\beta}(\boldsymbol{z})$   is zero.  From \cref{eq:Psi-AAm}, we have
\begin{equation*}
 \Psi(\boldsymbol{z}+\boldsymbol{d}) -\Psi(\boldsymbol{z})=
 \begin{bmatrix}
  q(x^*+d_{m+1})-q(x^*) +Q(\boldsymbol{z}+\boldsymbol{ d})\boldsymbol{\beta}(\boldsymbol{z}+\boldsymbol{ d})\\
   d_{m+1}\\
   \vdots \\
   d_2
  \end{bmatrix}.
\end{equation*}
Using \cref{eq:R-z+d,eq:Q-z+d,eq:beta-z+d}, we have
\begin{align}
 Q(\boldsymbol{z}+\boldsymbol{ d})\boldsymbol{\beta}(\boldsymbol{z}+ \boldsymbol{ d})=& -(I-A)\widehat{D}\big(A\widehat{D} \big)^{\dagger}(A(x^*+d_{m+1})-b),\nonumber \\
 =& -(A^{-1}-I)A\widehat{D}\big( A\widehat{D} \big)^{\dagger} Ad_{m+1}.\label{eq:Q-beta-simplification}
\end{align}
Since $A\widehat{D}\big( A\widehat{D}\big)^{\dagger} $ is  an orthogonal projection operator, $\|A\widehat{D}\big( A\widehat{D} \big)^{\dagger} \|_2 = 1$, or 0 when $\widehat{D}=0$.
Based on the above discussion, we can obtain
\begin{align*}
 \|\Psi(\boldsymbol{z}+\boldsymbol{d}) -\Psi(\boldsymbol{z})\|& \leq  \|q(x^*+d_{m+1})-q(x^*)\| +\|Q(\boldsymbol{z+d})\boldsymbol{\beta}(\boldsymbol{z+d})\| +\|\boldsymbol{d}\|,\\
 & \leq \|Md_{m+1}\| +\|A^{-1}-I\| \|Ad_{m+1}\| + \|\boldsymbol{d}\|,\\
 &\leq \big(\|I-A\| + \|A^{-1}-I\| \|A\| + 1\big) \|\boldsymbol{d}\|,\\
&\leq  \left( ( 1+ \|A^{-1}\| \|A\| ) \|I-A\|+1\right) \|\boldsymbol{d}\|.
\end{align*}
Thus, for all $\boldsymbol{d}$,
\begin{equation*}
 \|\Psi(\boldsymbol{z}+\boldsymbol{d})-\Psi(\boldsymbol{z}) \| \leq L \|\boldsymbol{d}\|,
\end{equation*}
where $L=( 1+\|A^{-1}\| \|A\| ) \|I-A\|+1$. This means that $\Psi(\boldsymbol{z})$ is Lipschitz continuous, and, hence, continuous, at $\boldsymbol{z}$ when $z_{m+1}=x^*$ and $R(\boldsymbol{z})$ is rank-deficient.

\section{Proof of \cref{thm:proof-directional-m}}
\label{app:proof-directional-m}

Let $\boldsymbol{d}=\begin{bmatrix} d_{m+1}^T & d_{m}^T & \ldots & d_{1}^T\end{bmatrix}^T$. Consider $\boldsymbol{z}=\boldsymbol{z}^*$. From \cref{eq:Psi-AAm}, we have
\begin{equation*}
 \Psi(\boldsymbol{z}^*+h\boldsymbol{d}) -\Psi(\boldsymbol{z}^*)=
 \begin{bmatrix}
  q(x^*+hd_{m+1})-q(x^*) +Q(\boldsymbol{z}^*+h\boldsymbol{ d})\boldsymbol{\beta}(\boldsymbol{z}^*+h\boldsymbol{ d})\\
   hd_{m+1}\\
   \vdots \\
   hd_2
  \end{bmatrix}.
\end{equation*}
First,
\begin{equation*}
q(x^*+hd_{m+1})-q(x^*)=M(x^*+hd)+b- (Mx^*+b) = hMd_{m+1}.
\end{equation*}
Then, we consider $Q(\boldsymbol{z}^*+h\boldsymbol{d})\boldsymbol{\beta}(\boldsymbol{z}^*+h\boldsymbol{d})$.
Note that
\begin{equation}\label{eq:AAm-Qbeta}
  Q(\boldsymbol{z}^*+h\boldsymbol{d})\boldsymbol{\beta}(\boldsymbol{z}^*+h\boldsymbol{d}) =- Q(\boldsymbol{z}^*+h\boldsymbol{d})\big(R(\boldsymbol{z}^*+h\boldsymbol{d})\big)^{\dagger}r(x^*+hd_{m+1}),
\end{equation}
where $R(\boldsymbol{z}^*+h\boldsymbol{d})$ is defined in \cref{eq:beta-continuous}, and
\begin{equation}\label{eq:expansion-r}
   r(x^*+h d_{m+1})=r'(x^*) hd_{m+1} +P(h d_{m+1}) h d_{m+1}\quad   \text{with}\quad \lim_{h \rightarrow 0} P(hd_{m+1})=0.
\end{equation}
It follows that
\begin{equation}\label{eq:limit-r}
  \lim_{h \rightarrow 0}\frac{r(x^*+hd_{m+1})}{h}= r'(x^*)d_{m+1}=Ad_{m+1}.
\end{equation}
Next, we claim that when $D(\boldsymbol{d})$ is full rank,
\begin{equation*}
  \lim_{h \rightarrow 0} \frac{  Q(\boldsymbol{z}^*+h\boldsymbol{d})\boldsymbol{\beta}(\boldsymbol{z}^*+h\boldsymbol{d})}{h}
  =-MD(\boldsymbol{d})\big(AD(\boldsymbol{d})\big)^{\dagger} Ad_{m+1}.
\end{equation*}
From \cref{eq:AAm-Qbeta} and \cref{eq:limit-r}, we only need to prove that
\begin{equation}\label{eq:to-prove}
   \lim_{h \rightarrow 0} - Q(\boldsymbol{z}^*+h\boldsymbol{d})\big(R(\boldsymbol{z}^*+h\boldsymbol{d})\big)^{\dagger}
   =-MD(\boldsymbol{d})\big(AD(\boldsymbol{d})\big)^{\dagger}.
\end{equation}
We prove the above statement in the following.
Using \cref{eq:expansion-r} for $x^*+hd_j$, we have
\begin{equation*}
    R(\boldsymbol{z}^*+h\boldsymbol{d}) =hr'(x^*)D(\boldsymbol{d})+h \Delta \quad   \text{with}\quad \lim_{h \rightarrow 0}\Delta=0,
\end{equation*}
where $\Delta$ is given by
\begin{equation*}
  \begin{bmatrix}
  P(hd_{m+1})d_{m+1}-P(hd_m)d_m & \ldots &P(hd_{m+1})d_{m+1}-P(hd_1)d_1
  \end{bmatrix}.
\end{equation*}

According to the definition of $Q$ (see \cref{eq:def-Q-general}) and $q(x)=x-r(x)$,
we have
\begin{equation*}
   Q(\boldsymbol{z}^*+h\boldsymbol{d}) =hD(\boldsymbol{d}) - h r'(x^*)D(\boldsymbol{d})-h \Delta.
\end{equation*}
Now we have
\begin{align*}
 - Q(\boldsymbol{z}^*+h\boldsymbol{d})\big(R(\boldsymbol{z}^*+h\boldsymbol{d})\big)^{\dagger}
   &=-h\big( D(\boldsymbol{d}) -   r'(x^*)D(\boldsymbol{d})-  \Delta\big)\big(h r'(x^*)D(\boldsymbol{d})+h\Delta\big)^{\dagger},\\
    &=\big( -D(\boldsymbol{d}) +  r'(x^*)D(\boldsymbol{d})+ \Delta\big)\big(r'(x^*)D(\boldsymbol{d})+\Delta\big)^{\dagger},\\
     &= (-D(\boldsymbol{d})+AD(\boldsymbol{d})+\Delta)\big( AD(\boldsymbol{d})+\Delta \big)^{\dagger},\\
      &= (-MD(\boldsymbol{d})+\Delta)\big( AD(\boldsymbol{d})+\Delta \big)^{\dagger}.
\end{align*}
When $D(\boldsymbol{d})$ is full rank and $h$ is small enough, rank($AD(\boldsymbol{d})+\Delta$)=rank($AD(\boldsymbol{d})$) because singular values are continuous and $A$ is nonsingular. Thus, according to \cite[Corollary 3.5]{stewart1977perturbation},
\begin{equation}\label{eq:limit-speudo-inverse}
  \lim_{h\rightarrow 0} \big( AD(\boldsymbol{d})+\Delta \big)^{\dagger} =\big( AD(\boldsymbol{d})  \big)^{\dagger}.
\end{equation}
Then \cref{eq:to-prove} follows.

Let $\widehat{\Delta}=q(x^*+hd_{m+1})-q(x^*) +Q(\boldsymbol{z}^*+h\boldsymbol{ d})\boldsymbol{\beta}(\boldsymbol{z}^*+h\boldsymbol{ d})$. Then
\begin{equation*}
\lim_{h\rightarrow 0}  \ds \frac{\widehat{\Delta}}{h} =Md_{m+1} -MD(\boldsymbol{d})\big(AD(\boldsymbol{d})\big)^{\dagger}Ad_{m+1} =Md_{m+1} +MD(\boldsymbol{d})\widehat{\boldsymbol{\beta}}(\boldsymbol{d}).
\end{equation*}
It follows that  $\displaystyle\lim_{h\rightarrow 0}   \frac{\Psi(\boldsymbol{z}^*+h\boldsymbol{d})-\Psi(\boldsymbol{z}^*)}{h}$  can be written as in \cref{eq:-direct-nonlinear-AAm-matrix}.



\bibliographystyle{siamplain}
\bibliography{AAref}
\end{document}